\documentclass[reqno,11pt]{amsart}

\usepackage[top=2.0cm,bottom=2.0cm,left=3cm,right=3cm]{geometry}
\usepackage{amsthm,amsmath,amssymb,dsfont}
\usepackage{mathrsfs,amsfonts,functan,extarrows,mathtools}
\usepackage[colorlinks]{hyperref}

\usepackage[utf8]{inputenc}
\usepackage[T1]{fontenc}

\usepackage{marginnote}
\usepackage{xcolor}
\usepackage{stmaryrd,bm}
\usepackage{esint}
\usepackage{graphicx}
\usepackage[english]{babel}

\newtheorem{theorem}{Theorem}[section]
\newtheorem{proposition}{Proposition}[section]

\newtheorem{remark}{Remark}[section]
\newtheorem{lemma}{Lemma}[section]

\numberwithin{equation}{section}
\allowdisplaybreaks

\def\p{\partial}

\def\d{\mathrm{d}}
\def\no{\nonumber}
\def\R{\mathbb{R}}
\def\eps{\epsilon}

\def\w{\mathfrak{w}}

\def\Id{\mathcal{I}}
\def\tr{\mathrm{tr}}

\def\eff{\mathrm{eff}}

\newcounter{wronumber}

\newcommand{\skpa}[2]{\left\langle #1,\, #2 \right\rangle}

\newcommand{\skpt}[2]{\langle #1,\, #2 \rangle}

\begin{document}
\title[Compressible NSLLG system]{Compressible Navier-Stokes-Landau-Lifshitz-Gilbert system: derivations and well-posedness}

\author[B. L. Guo]{Boling Guo}
\address[Boling Guo]{\newline   Institute of Applied Physics and Computational Mathematics in Beijing, P.O. Box 8009, Beijing
	100088, China}
\email{gbl@iapcm.ac.cn}

\author[N. Jiang]{Ning Jiang}
\address[Ning Jiang]{\newline School of Mathematics and Statistics, Wuhan University, Wuhan, 430072, P. R. China}
\email{njiang@whu.edu.cn}

\author[H. Liu]{Hui Liu}
\address[Hui Liu]
{\newline School of Finance and Mathematics, Huainan Normal University, Huainan 232038, P. R. China}
\email{huiliu@hnnu.edu.cn}

\author[Y.-L. Luo]{Yi-Long Luo}
\address[Yi-Long Luo]
{\newline School of Mathematics, Hunan University, Changsha 410082, P. R. China}
\email{luoylmath@hnu.edu.cn}

\author[T.-F. Zhang]{Teng-Fei Zhang}
\address[Teng-Fei Zhang]{\newline School of Mathematics and Physics, China University of Geosciences, Wuhan, 430074, P. R. China}
\email{zhangf@cug.edu.cn}

\thanks{\today}

\maketitle

\begin{abstract}
 In this paper, we first derive the compressible Navier-Stokes/Landau-Lifshitz-Gilbert (NS-LLG) model for magnetoelastic materials via the energetic variational approach (EnVarA). It is important to emphasize that the manner in which the evolution of magnetoelastic materials is influenced by the fluid motion--specifically through the deformation gradient--determines the kinematics of the magnetization and consequently leads to distinct governing equations. Subsequently, we establish the local-in-time existence of solutions to the compressible NS-LLG system under finite initial energy. Finally, near the constant equilibrium for magnetoelasticity in the absence of an external magnetic field, we reformulate the evolutionary model, which allows an additional dissipative term to be identified from the elastic stress. Based on this reformulation, we justify the global well-posedness of the evolutionary magnetoelasticity system with zero external magnetic field, provided the initial data are sufficiently small. In particular, when the magnetic field $M$ vanishes, this model reduces to the viscoelastic model. Our results significantly relax the previous initial data requirements, only assume the most basic structural condition $\rho_{0} \operatorname{det} F_{0} = 1$.\\


\end{abstract}

\maketitle
\section{Introduction and main results} 
\label{sec:the_target_model}



\subsection{A magnetoelastic model: the compressible Navier-Stokes-Landau-Lifshitz-Gilbert system}

The phenomenon of magnetoelasticity, describing the interaction between elastic and magnetic effects, was identified as early as the 1960s (see \cite{ref3}). In a typical manifestation, applying a magnetizing field to a ferromagnetic rod alters not only its magnetization but also its length. Conversely, subjecting the rod to tension changes both its length and its magnetization. Early constitutive modeling of such materials was developed by Brown \cite{ref3} and Tiersten \cite{ref31,ref32}. From an analytical perspective, much of the existing work focuses on static configurations via energy minimization \cite{ref7,ref8,ref14}. For dynamic, rate-independent evolution, the framework of energetic solutions has been employed \cite{ref17,ref26}. In micromagnetics, the dynamics is typically governed by the Landau--Lifshitz--Gilbert (LLG) equation \cite{ref12,ref13,ref18}, whose analytical theory is well-developed (see, e.g., \cite{ref1,ref5,ref24,ref25}). Coupling the LLG equation with elasticity in the small-strain regime has also been studied \cite{ref4,ref6}. Very recently, a thermodynamic model for viscoelastic deformable magnetic materials was proposed in \cite{ref34}. That model incorporates the Landau theory for ferro-para-magnetic phase transition, a gradient theory for magnetization with a mechanically dependent coefficient, hysteresis in the magnetization evolution via an objective corotational time derivative in the LLG equation, and the effects of the demagnetizing field.

We begin by deriving a compressible Navier-Stokes/Landau-Lifshitz-Gilbert system using the energetic variational approach (EnVarA). The resulting system of equations is as follows:
\begin{align}\label{CMHED-O}
\begin{cases}
  \p_t \rho + \nabla \cdot (\rho v) = 0 \,, \\
  \begin{aligned}
    \rho (\p_t v + v \cdot \nabla v) +\; & \nabla \left( P (\rho) - A |\nabla M|^2 + \mu_0 M \cdot H_{ext} \right) 
    \\[5pt] & 
    = \nabla \cdot \left( \Sigma (v) - 2 A \nabla M \odot \nabla M + \tfrac{W'(F)F^\top}{\det F} \right) + \mu_0 (\nabla H_{ext})^\top M \,, 
  \end{aligned} \\
  \partial_t F + (v \cdot \nabla) F = \nabla v F \,, \\
  \p_t M + v \cdot \nabla M = - \gamma M \times (2 A \Delta M + \mu_0 H_{ext}) - \lambda M \times M \times (2 A \Delta M + \mu_0 H_{ext}) \,, \\ 
  |M(x,t)| = 1 \,. 
\end{cases}
\end{align}
for $(t,x) \in (0, T) \times \Omega $. Here $\Sigma (v) = \mu (\nabla v + \nabla^T v) + \xi (\nabla \cdot v) I$ denotes the viscosity dissipation tensor. The scalar function $W (F)$ represents the elastic energy, with $W' (F)$ being the Piola-Kirchhoff tensor; consequently, $\frac{W' (F) F^\top}{\det F}$ is the corresponding Cauchy-Green stress tensor. The exterior magnetic field $H_{ext} (t,x)$ is a prescribed vector-valued function. The material motion is described by the map $x = x (X, t)$, defined as the solution to
\begin{equation}\label{FM-x}
  \begin{aligned}
    \frac{\d x}{\d t} = v (x(X,t),t), \quad x (X,0)=X.
  \end{aligned}
\end{equation}
Associated to the flow, we introduce the deformation gradient tensor $F$ (the Jacobian of the flow map):
  \begin{align}
    F(X,t)= \frac{\p x(X,t)}{\p X}, \quad \text{or, in components},\ F_{ij} = \frac{\p x_i}{\p X_j}.
  \end{align}
For compressible flows, the relation $\rho \det F=1$ holds, which reduces to $\det F=1$ in incompressible case. We consider a polytropic (or $\gamma$-law) pressure for the compressible fluid, i.e., $P(\rho) = a \rho^\gamma $ with $\gamma >1$, which corresponds to the internal energy density $$w(\rho) = \frac{a}{\gamma -1} \rho^\gamma .$$
The physical parameters in \eqref{CMHED-O} satisfy 
    \begin{equation*}
        \begin{aligned}
            a > 0, \ \gamma > 1, \ A > 0, \ \lambda > 0, \ \mu > 0, \ \mu + \xi > 0 \,.
        \end{aligned}
    \end{equation*}

The compressible system \eqref{CMEH-eq} is derived by applying the Energetic Variational Approach (EnVarA) developed by C. Liu and collaborators; see, e.g., \cite{GKL-18notes,WL-22entropy} for a comprehensive exposition. Following this framework, we define the total energy as the sum of the kinetic energy $\mathcal{K}$ and the Helmholtz free energy $\mathcal{F}$: 
\begin{align}
  \mathcal{E} = \mathcal{K} + \mathcal{F}, 
\end{align}
where 
\begin{align}
  \label{Def-K}\mathcal{K} & = \int_{\Omega^t} \tfrac{1}{2} \rho |v|^2 \d x, \\
  \label{Def-F} \mathcal{F} & = \int_{\Omega^t} \left[ w(\rho) + A|\nabla M|^2 + \frac{W(F)}{\det F} - \mu_0 M \cdot H_{ext} \right] \d x \,. 
\end{align}
Here, $\Omega^t$ denotes the current configuration. The dissipation in the system is characterized by the functional
\begin{align}\label{eq:dissipation-u-M}
  \no \mathcal{D} & = \mathcal{D}(v) + \mathcal{D}(M) \\
  \mathcal{D}(v) & = \int_{\Omega} \left[ \mu |\nabla v|^2 + (\mu + \xi) |\nabla \cdot v|^2 \right]  \d x \,, \\
   \no \mathcal{D}(M) & = \int_{\Omega} \lambda |M \times (2 A \Delta M + \mu_0 H_{ext})|^2 \d x \,.  
\end{align}
A central kinematic assumption, following Bene$\check{\mathrm{s}}$ov\'{a}-Forster-Liu-Schl\"{o}merkemper \cite{BFLS-2018-SIAM}, is that the magnetization satisfies the Lagrangian condition $M(x(X,t),\ t) = M(X,t)$. Physically,  this means the magnetization vector is tied to material points and is transported with the flow without being stretched or rotated independently of the deformation (see \cite[Remark 2.1]{BFLS-2018-SIAM}). This viewpoint differs from that in, e.g., Forster's thesis \cite{Forster-16thesis}. Under this assumption, the application of the EnVarA yields the system \eqref{CMHED-O}. The detailed derivation is provided in Section \ref{Sec:EnVarA}.

When the elastic energy density is given by the Hookean linear elasticity, i.e., $W (F) = \frac{1}{2} |F|^2$, the system \eqref{CMHED-O} reduces to
\begin{align}\label{CMEH-eq}
	\begin{cases}
		\p_t \rho + \nabla \cdot (\rho v) = 0 \,, \\
		\begin{aligned}
			& \rho (\p_t v +\; v \cdot \nabla v) + \nabla \left( P(\rho) - A |\nabla M|^2 + \mu_0 M \cdot H_{ext} \right) 
			\\[5pt] & 
			=  \mu\Delta v+(\mu+\xi)\nabla(\nabla \cdot v)+\nabla\cdot(\rho F F^\top)-2A \nabla \cdot ( \nabla M \odot \nabla M) + \mu_0 (\nabla H_{ext})^\top M \,, 
		\end{aligned} \\
		\partial_t F + (v \cdot \nabla) F = \nabla v F \,, \\
		\p_t M + v \cdot \nabla M = 2A\lambda \Delta M+\lambda \mu_0 H_{ext}+\Gamma(M)M-\gamma M \times(2A \Delta M+\mu_0 H_{ext})
		\,, \\ 
		|M(x,t)| = 1 \,,
	\end{cases}
\end{align}
where the Lagrangian multiplier is $\Gamma(M)=2\lambda A|\nabla M|^2-\lambda \mu_0 M\cdot H_{ext}$. The reformulation of the magnetization evolution equation follows the standard procedure; see, for instance, \cite{JLL-JDE-2023}.

The objective of this paper is to study the system \eqref{CMEH-eq} on the spatial domain $\Omega = \mathbb{T}^3$, subject to the initial conditions
\begin{equation}\label{IC-1}
  \begin{aligned}
    ( \rho, v, F, M ) (0, x) = (\rho_0, v_0, F_0, M_0) (x)
  \end{aligned}
\end{equation}
with $|M_0| = 1$.

\subsection{Main results}

Before stating the main results, we introduce the notation used throughout the paper. We denote by $ L^{p}(\Omega) $ the standard Lebesgue space, abbreviated as $ L^{p} $ for $ p\in[1,\infty] $. Its norm is given by 
$$ \|f\|_{L^{p}}=\left(\int_{\Omega}|f|^{p}\mathrm{d}x\right)^{\frac{1}{p}} ( 1 \leq p < \infty) \,,  \quad  \|f\|_{L^{\infty}}=\operatorname*{ess\, sup}_{x\in \Omega}|f(x)| \,. $$ 
For $ p=2 $, the inner product on $ L^{2} $ is written as $ \langle\cdot,\cdot\rangle $. The unit sphere in $\mathbb{R}^3$ is denoted by $ \mathbb{S}^{2}=\{v \in\mathbb{R}^3;|v |=1\} $. The spatial gradient and divergence are denoted by $ \nabla $ and $ \nabla\cdot $, respectively. We write $ \partial_{i} = \partial_{x_{i}} $ for $ i=1,2,3 $, and the Laplacian is $ \Delta=\partial_{i}\partial_{i} $ (summation convention used). For a matrix-valued function $ G=(G^{ij})_{1\leq i,j\leq 3} $, its divergence is defined componentwise by $ (\nabla\cdot G)_{i}=\partial_{j}G^{ji} $ for $ 1\leq i\leq 3 $. 

Given a multi-index $ m=(m_{1},m_{2}, m_{3})\in\mathbb{N}^{3} $, we set $ partial^{m}=\partial_{x_{1}}^{m_{1}}\partial_{x_{2}}^{m_{2}} \partial_{x_{3}}^{m_{3}} $ and $|m|=m_{1}+m_{2}+ m_{3}$. For two multi-indices $m, \widetilde{m}$, the notation $ m \leq \widetilde{m} $ means $m_i \leq \widetilde{m}_i$ for all $i$, the strict inequality $ m<\widetilde{m} $ means $ m\leq\widetilde{m} $ and $ |m|<|\widetilde{m}| $. 

The Sobolev space $ W^{s,p}(\Omega) $ (abbreviated $W^{s,p}$) is equipped with the norm
$$ \|f\|_{W^{s,p}}=\left(\sum_{|m|\leq s}\int_{\Omega}|\partial^{m}f(x)|^{p}\mathrm{d}x\right)^{\frac{1}{p}}. $$
When $ p=2 $, we write $ H^{s}(\Omega) =W^{s,2} (\Omega) $. For a given positive weight function $\rho: \Omega \to \R_+$, the weighted spaces $L^p_\rho$, $W^{s,p}_\rho$ and $H^s_\rho$ are defined analogously by replacing the Lebesgue measure $\d x$ with the weighted measure $\rho (x) \d x$ in the corresponding norms. 

For any integer $s \in \mathbb{N}$, we define the following local energy functional $\mathcal{E}_s (t)$ and local dissipation-rate functional $\mathcal{D}_s (t)$:
\begin{equation}\label{EDs-loc}
	\begin{aligned}
		\mathcal{E}_s (t) & := \| \rho \|^2_{H^s} +  \| F \|^2_{H^s} + \| v \|^2_{H^s_\rho} +\|\nabla M\|^2_{H^s}\,, \\
		\mathcal{D}_s (t) & := \mu \| \nabla v \|^2_{H^s} + (\mu+\xi)\|\nabla\cdot v\| ^{2}_{H^{s}}+2\lambda A\|\Delta M\|^2_{H^s}  \,.
	\end{aligned}
\end{equation}
Correspondingly, the initial local energy is
\begin{equation}\label{E_s-in}
	\begin{aligned}
		\mathcal{E}_s^{in} & := \| \rho_0 \|^2_{H^s} +  \| F_0 \|^2_{H^s} + \| v_0 \|^2_{H^s_{\rho_0}} +\|\nabla M_0 \|^2_{H^s} \,.
	\end{aligned}
\end{equation}
With these preliminaries, we can now state our first main theorem.

\begin{theorem}[Local well-posedness]\label{Thm1}
    Let $s \geq 3$ be an integer. Suppose the initial data $(\rho_0, v_0, F_0, M_0)$ in \eqref{IC-1} satisfy
    $$ \mathcal{E}_s^{in} < \infty, \quad \inf_{x \in \mathbb{T}^3} \rho_0 (x) > 0 \,. $$
    Assume further that the external magnetic field satisfies $H_{ext} \in L^\infty (0, T_0; H^s)$ for some $T_0 > 0$. Then there exists a time $T \in ( 0, T_0 ]$ such that the Cauchy problem \eqref{CMEH-eq}-\eqref{IC-1} admits a unique solution $(\rho, v, F, M) $ on $[0,T] \times \mathbb{T}^3$ with the following properties:
    \begin{equation*}
        \begin{aligned}
        & |M (t,x)| = 1, \quad \inf_{(t,x) \in [0, T] \times \mathbb{T}^3} \rho(t,x) > 0,  \\
        & \rho, F, v, \nabla M \in L^\infty (0, T; H^s (\mathbb{T}^3)), \nabla v, \Delta M \in L^2 (0, T; H^s (\mathbb{T}^3)) \,.
        \end{aligned}
    \end{equation*}
    Moreover, the solution satisfies the energy estimate
    \begin{equation*}
        \begin{aligned}
            \sup_{t \in [0, T]} \mathcal{E}_s (t) + \int_0^T \mathcal{D}_s (t) d t \leq C
        \end{aligned}
    \end{equation*}
    where the constant $C > 0$ depends on $s$, $T$ and $\mathcal{E}_s^{in}$.
\end{theorem}

We now consider the global well-posedness of \eqref{CMEH-eq} near the constant equilibrium $(1, 0, I, M_e) \in \mathbb{R} \times \mathbb{R}^3 \times \mathbb{R}^{3 \times 3} \times \mathbb{S}^2$ in the absence of an external magnetic field, i.e., $H_{ext} = 0$, on the spatial domain $x \in \mathbb{T}^3$. Introducing the fluctuation variables
\begin{equation*}
    \begin{aligned}
        \rho=1+\theta \,, \quad M=d+M_e \,, \quad v=u \,,
    \end{aligned}
\end{equation*}
and using the constraint $|M|= |M_e|=1$, which implies $|d|^2+2 M_e\cdot d=0$, we obtain from \eqref{CMEH-eq} the following system for $(\theta,u,F,d)$:
\begin{align}\label{CMEH-eqG}
	\begin{cases}
		\p_t \theta +u\cdot\nabla\theta+(1+\theta)\nabla\cdot u = 0 \,, \\
		\begin{aligned}
			(1&+\theta)(\p_t u+u\cdot\nabla u) + \nabla \left( P(1+\theta) - A |\nabla d|^2 \right) 
			\\[5pt] & 
			=  \mu\Delta u+(\mu+\xi)\nabla(\nabla \cdot u)+\nabla\cdot[(1+\theta)FF^\top]-2A \nabla \cdot ( \nabla d \odot \nabla d)  \,, 
		\end{aligned} \\
		\partial_t F + (u \cdot \nabla) F = \nabla u F \,, \\
		\p_t d + u\cdot \nabla d = 2A\lambda \Delta d+2A\lambda|\nabla d|^2(d+M_e)-2A\gamma(d+M_e)\times \Delta d
		\,, \\ 
	|d|^2+2M_e\cdot d=0 \,. 
	\end{cases}
\end{align}
Define the global energy functional
\begin{equation}\label{close-eq5}
\begin{aligned}
\mathbf{E}_s(t)=\|\theta\|^2_{H^s}+\|u\|^2_{H^s}+ |\bar{d} |^2+\|d-\bar{d}\|^2_{H^s}+\|\nabla d\|^2_{H^s}+\|F^{-1} - I\|^2_{H^s}\,,
\end{aligned}
\end{equation}
and the corresponding global dissipation-rate functional
\begin{equation}\label{close-eq6}
\begin{aligned}
\mathbf{D}_s(t)=\|\nabla u\|^2_{H^s}+\|\nabla\theta\|_{H^{s-1}}+\|F^{-1} - I|^2_{H^s}+\|\nabla d \|^2_{H^s}+\|\Delta d \|^2_{H^s}+\|\nabla \cdot u \|^2_{H^s}\,.
\end{aligned}
\end{equation}
The associated initial energy is
\begin{equation}\label{Esin-g}
    \begin{aligned}
        \mathbf{E}_s^{in} = \|\rho_0 - 1\|^2_{H^s}+\|v_0 \|^2_{H^s}+ | \overline{M_0 - M_e} |^2+\|M_0-\overline{M_0}\|^2_{H^s}+\|\nabla M_0\|^2_{H^s}+\|F^{-1}_0 - I\|^2_{H^s} \,,
    \end{aligned}
\end{equation}
where the overline denotes the spatial average: 
\begin{equation}\label{f-bar}
    \begin{aligned}
        \bar{f} (t) = \frac{1}{|\mathbb{T}^3|} \int_{\mathbb{T}^3} f (t,x) \d x .
    \end{aligned}
\end{equation}
With the above notations, we state the following global well-posedness result.

\begin{theorem}[Global well-posedness]\label{Thm2}
    Let $s \geq 3$ be an integer and assume the external magnetic field is absent, $H_{ext} = 0$. Let the initial data satisfy the compressibility conditions
    \begin{equation*}
        \begin{aligned}
            \rho_0 \det F_0 = 1 \,, \inf_{x \in \mathbb{T}^3} \rho_0 (x) > 0.
        \end{aligned}
    \end{equation*}
    Then there exists a constant $\eps_0 > 0$, sufficiently small, such that if 
    \begin{equation*}
        \begin{aligned}
            \mathbf{E}_s^{in} \leq \eps_0 \,,
        \end{aligned}
    \end{equation*}
    the perturbed system \eqref{CMEH-eqG} admits a unique global-in-time solution $(\theta, u, F, d)$ satisfying
    \begin{equation*}
        \begin{aligned}
            \theta, u, d - \bar{d}, \nabla d, F^{-1} - I \in L^\infty (\mathbb{R}_+; H^s) \,, \bar{d} \in L^\infty (\mathbb{R}_+) \,, \nabla u, \Delta d \in L^2 (\mathbb{R}_+; H^s) \,.
        \end{aligned}
    \end{equation*}
    Furthermore, there exists a constant $c_0 > 0$ such that the following energy estimate holds:
    \begin{equation*}
        \begin{aligned}
            \sup_{t \geq 0} \mathbf{E}_s(t) + \int_0^\infty \mathbf{D}_s (t) d t \leq c_0 \mathbf{E}_s^{in} \,.
        \end{aligned}
    \end{equation*}
\end{theorem}

\subsection{Historical remarks}

The system \eqref{CMHED-O} represents a nonlinear coupling between the hydrodynamics of viscoelasticity and the Landau-Lifshitz-Gilbert equation. While the analytical theory for each subsystem has been extensively developed over the past two decades, the well-posedness analysis of the coupled system \eqref{CMHED-O} remains largely open. In \cite{BFLS-2018-SIAM}, global-in-time weak solutions were constructed via a Galerkin method combined with a fixed-point argument, adapting ideas from the work on liquid crystal flows by Lin and Zhang \cite{ref19} and on the Landau-Lifshitz equation by Carbou and Fabrie \cite{ref5}. Related progress in this direction can be found in \cite{ref16,ref33}. We note that the notion of ``weak solution'' in \cite{BFLS-2018-SIAM} and \cite{ref16} is non-standard: the spatial regularity required for the magnetization $M$ is not merely $H^1$, but higher ($H^3$ in \cite{BFLS-2018-SIAM} and $H^2$ in \cite{ref16}). Moreover, both works consider a regularized version of the transport equation for the deformation gradient, involving a term $\kappa \Delta F$, rather than the physical case $\kappa = 0$ corresponding to incompressible version of system \eqref{CMHED-O}. Additionally, these results require the initial deviation of $F$ from the identity matrix $I$ to be small. These restrictive assumptions stem from fundamental analytical difficulties: the lack of regularity for the transport equation and the geometric constraint $|M(t,x)| = 1$, which is notoriously difficult to preserve under weak approximations. In fact, handling this constraint remains highly nontrivial even in the higher-regularity framework adopted in the present paper. We also mention the work \cite{ref33}, where a local well-posedness and a blow-up criterion for classical solutions are established for a modified version of \eqref{CMHED-O}---specifically, with the constraint $|M|=1$ replaced by a Ginzburg-Landau penalization. In contrast, Jiang, Liu, and Luo \cite{JLL-JDE-2023} studied the incompressible counterpart of the physical system \eqref{CMHED-O} while preserving the geometric constraint $|M(t,x)| = 1$. Working in the setting of classical solutions rather than the weak-strong framework of \cite{BFLS-2018-SIAM,ref16}, they established the first global existence theory for classical solutions to the incompressible version of the full evolutionary magnetoelastic system \eqref{CMHED-O}.

We note that by removing the magnetization vector field \(M(t,x)\), the system \eqref{CMHED-O} reduces to a compressible viscoelastic fluid system of Oldroyd type, which describes a class of non-Newtonian fluids exhibiting phenomena such as elastic memory effects; see, e.g., \cite{ref11z,ref12z,ref14z,ref18z,ref21z}. In such fluids, the elastic response is captured by an energy functional that depends on the deformation gradient tensor \(F\).

The incompressible counterpart of this viscoelastic model has been extensively studied. The local well-posedness in the Sobolev space \(H^s\) was established by Lin et al.~\cite{ref18z}, Chen and Zhang~\cite{ref6z}, Lei et al.~\cite{ref16z}, and Lin and Zhang~\cite{ref19z}; global well-posedness for small initial data was also proved. While the local theory follows from standard energy estimates, the global result relies on a more delicate analysis that reveals a damping mechanism for the deformation gradient fluctuation \(F-I\). A key observation in this analysis is that certain linearly appearing terms are actually of higher order. For example, Lei et al.~\cite{ref16z} identified the identity
\[
\nabla_{k} H^{ij}-\nabla_{j} H^{ik}=H^{lj}\nabla_{l} H^{ik}-H^{lk}\nabla_{l} H^{ij},
\]
where \(H = F - I\). This relation implies, in particular, that \(\nabla \times H\) is a higher-order term.

For the incompressible system augmented with a linear damping term in the evolution equation for the Cauchy-Green strain tensor $FF^{\top}$, Chemin and Masmoudi \cite{ref3z} established the existence of both local solutions and global small solutions in critical Besov spaces; see also \cite{ref22z} for related well-posedness results in critical spaces. In the compressible setting, Lei and Zhou \cite{ref15z} proved the local existence of smooth solutions for the two-dimensional system with initial data that are small perturbations of a general incompressible state. Furthermore, they showed the convergence of these solutions to the global solution of the incompressible equations via a compactness argument. For a compressible viscoelastic system without pressure, Lei et al. \cite{ref17z} established the global existence of smooth solutions near equilibrium.

The global existence of weak solutions, however, remains an open problem even for the incompressible case. When the contribution of the strain rate (the symmetric part of $\nabla v$) is neglected in the constitutive equation, Lions and Masmoudi \cite{ref20z} proved the global existence of weak solutions for general initial data. For the full compressible system, Qian and Zhang \cite{QZ-ARMA-2010} first established its well-posedness in the largest possible functional framework--namely, in critical spaces (which, in two dimensions, are very close to the natural spaces for weak solutions). Second, they explicitly characterized the smoothing and damping effects generated by the linearized system with convection terms, a key ingredient in proving global existence for small data. We remark that their analysis relies on the following initial structural conditions:
\begin{align*}
\nabla \cdot \left( \frac{F_{0}^{\top}}{\operatorname{det} F_{0}} \right) &= 0, \qquad \nabla \cdot \left( \rho_{0} F_{0}^{\top} \right) = 0, \\
\rho_{0} \operatorname{det} F_{0} &= 1, \qquad F_{0}^{l k} \nabla_{l} F_{0}^{i j} - F_{0}^{l j} \nabla_{l} F_{0}^{i k} = 0.
\end{align*}
It is worth noting that, in order to reveal weak dissipative mechanisms for $F-I$, the works \cite{ref6z,ref16z,ref18z} introduced and exploited certain special quantities. Finally, we recall the seminal work of Danchin \cite{ref7z}, who proved the global well-posedness for the compressible Navier--Stokes equations near equilibrium in critical spaces. Compared with that case, the analysis for viscoelastic fluids must also handle the intricate coupling among the velocity, density, and deformation tensor, while utilizing intrinsic structural properties of the viscoelastic system.

\subsection{Novelties and sketch of proofs}

\subsubsection{Novelties}\label{Subsec:n}

We now describe the novelties of the present paper.

First, we derive the compressible Navier–Stokes/Landau–Lifshitz–Gilbert system \eqref{CMHED-O} via the energetic variational approach. The derivation begins with the definitions of the kinetic energy $\mathcal{K}$ and the Helmholtz free energy $\mathcal{F}$, as introduced in \eqref{Def-K} and \eqref{Def-F}, respectively. From these, we define the action functional
\[
\mathcal{A}( {x} ) = \int_0^T \big( \mathcal{K} - \mathcal{F} \big) \, \d t,
\]
where $ {x} = {x}(X, t)$ denotes the flow map given in \eqref{FM-x}. The evolution equation for the velocity field ${v}$ is obtained from the principle of force balance, which equates the variational derivative of the action functional with the dissipation due to viscosity. In the $L^2$ sense, this balance reads
\begin{equation*}
\big( \delta_{x} \mathcal{A} \big)_{L^2_{t,x}} = \Big( \frac{1}{2} \, \delta_{\dot{x}} \mathcal{D} \Big)_{L^2_x},
\end{equation*}
hence, \eqref{eq:force-balance-x} below. Here, $\delta_{x}$ denotes the variational derivative with respect to the flow map $x$, $\delta_{\dot{x}}$ denotes the derivative with respect to the ``virtual velocity'' $\dot{x}$, and $\mathcal{D}$ is the dissipation functional given in \eqref{eq:dissipation-u-M}. The resulting equation corresponds to the momentum balance in \eqref{CMHED-O}; see the detailed computation leading to \eqref{CMHED-O}.

The equation for the magnetization $ {M}$ is derived from a separate force balance associated with its variation. Let $\mathring{M}$ denote the total kinematic transport of the magnetization, defined as in \eqref{eq:transport-M0} below. The force balance in the $L^2$ sense is then stated in Remark \ref{Rmk-2.1} below, namely,
\begin{equation*} 
\big( \delta_{ {M}} \mathcal{A} \big)_{L^2_{t,x}} = \Big( \frac{1}{2} \, \delta_{\mathring{M}} \mathcal{D} \Big)_{L^2_x},
\end{equation*}
where $\delta_{ {M}}$ is the variational derivative with respect to $ {M}$. This yields the Landau-Lifshitz-Gilbert type equation for $ {M}$ in \eqref{CMHED-O}. We remark that $\rho$-equation follows from the mass conservation law, and the $F$-equation is a consequence of the flow map \eqref{FM-x}.

Second, for the compressible viscoelastic counterpart, the only initial structural condition required to prove the global-in-time result near the constant equilibrium is $\rho_0 \det F_0 = 1$. Indeed, if one removes the magnetization $M$, the model \eqref{CMHED-O} reduces to the exact compressible viscoelastic fluid system of Oldroyd type. In the previous work \cite{QZ-ARMA-2010}, the following initial structural conditions are imposed:
\begin{equation*}
  \begin{aligned}
    & \nabla \cdot\left(\frac{F_{0}^{\top}}{\det F_{0}}\right)=0, \quad \nabla \cdot\left(\rho_{0} F_{0}^{\top}\right)=0, \\
    & \rho_{0} \det F_{0}=1, F_{0}^{l k} \nabla_{l} F_{0}^{i j}-F_{0}^{l j} \nabla_{l} F_{0}^{i k}=0.
  \end{aligned}
\end{equation*}
These conditions simplify the compressible viscoelastic fluid system to a hyperbolic–parabolic system, for which global well-posedness with small initial data can be established in critical Besov spaces. In the present paper, we only require the initial structural condition $\rho_{0} \det F_{0}=1$, which implies $\rho \det F=1$ for all $t \geq 0$. Since $\rho > 0$, we have $\det F > 0$. We focus on the fluctuation variable $U = F^{-1} - I$, which satisfies the curl-free property $\partial_{i} U^{j k} = \partial_{k} U^{j i}$ and is subject to the evolution \eqref{G} below. Then the div–curl theorem yields $U^{ij} = \partial_i \psi^j$. Consequently, the evolution of $F$ can be transformed into the third equation in \eqref{CMEH-eqGd}, namely,
$$\partial_{t} \psi+u+u \cdot \nabla \psi = 0 \,.$$
Moreover, the divergence of the Cauchy–Green stress tensor for Hookean linear elasticity,
$$ \nabla \cdot (\rho F F^{\top}) = - \Delta \psi+\ \text{some higher order terms} \,,$$
(see \eqref{U000} below) provides the required dissipative effect $-\Delta \psi$ for the viscoelastic part.

Third, in proving global existence, we need to design the instant energy functional $\mathbb{E}_{s;\delta,\eta,\epsilon}(t)$ (see \eqref{close-eq7}) and the instant dissipation rate $ \mathbb{D}_{s;\delta,\eta,\epsilon}(t) $ (see \eqref{close-eq8}) for the compressible model \eqref{CMHED-O}, which differs from the incompressible NS-LLG model studied in \cite{JLL-JDE-2023}. To derive the dissipation effect of $\psi$ in the energy estimate, we rewrite the $ \psi $-equation in the form
$$
\partial_t \Delta \psi + \frac{1}{\mu} \Delta \psi + \frac{1}{\mu} \Delta w = -\Delta(u \cdot \nabla \psi)
$$
with $ w = \mu(u - \bar{u}) - (\psi - \bar{\psi}) $. Then the quantity $w$ requires further controlled. Using the $u$-equation in \eqref{CMEH-eqGd}, $w$ satisfies 
$$
-\Delta w = (\mu + \xi) \nabla (\nabla \cdot u) - \nabla P(1 + \theta) + F_{com},
$$
where $F_{com}$ is readily controlled. In the incompressible model, $w$ satisfies the Stokes system
$$
\left\{
\begin{aligned}
-\Delta w + \nabla q &= F_{{incom}}, \\
\nabla \cdot w &= -\nabla \cdot \psi,
\end{aligned}
\right.
$$
where the source term $F_{{incom}}$ is controllable, and $\nabla \cdot \psi = \operatorname{tr} U$ is in fact nonlinear due to the incompressibility. Because of the linear effect contained in $(\mu + \xi) \nabla (\nabla \cdot u) - \nabla P(1 + \theta)$ in the compressible model, the previous ideas fail in the present work. We will instead employ the relation
$$ -\Delta \psi \sim \partial_{t} u + \nabla \theta -\mu \Delta u +(\mu + \xi) \nabla (\nabla \cdot u) + \cdots  $$
derived from the $u$-equation in \eqref{CMEH-eqGd} to control $w$. More precisely, we apply the relation \eqref{gpsi-eq7} to handle the quantity $C_2$ in \eqref{gpsi-eq6}, which motivates the design of the instant energy functional and instant dissipation rate in \eqref{close-eq7}-\eqref{close-eq8}. Together with Lemma \ref{Lmm-Glob-ED}, this allows us to close the global energy estimate \eqref{close-eq21} in Proposition \ref{Prop-Global}.

Fourth, owing to compressibility, we estimate separately the average part $\bar{d}$ and the zero-mean part $d-\bar{d}$ of the fluctuation $d = M - M_e$ of the magnetization. Indeed, $d$ satisfies
$$ \partial_{t}d + u \cdot \nabla d = 2A \lambda \Delta d + \text{nonlinear terms} \,. $$
When establishing the $L^{2}$-estimate, the quantity
$$
\langle u \cdot {\nabla} {d}, d \rangle = - \frac{1}{2} \langle {\nabla} \cdot u, |d|^{2} \rangle
$$
has to admit the upper bound $\mathbf{E}_{s}^{\frac{1}{2}}(t) \mathbf{D}_{s} (t)$. For the incompressible model, this quantity vanishes because $\nabla \cdot u = 0$. The dissipation $\mathbf{D}_{s}(t)$ involves only $\| \nabla d \|_{H^s}^{2} + \|\Delta d \|_{H^s}^{2}$ (the dissipative effect for the magnetization). However, $\| d \|_{L^4}^2$ does not admit the upper bound $\mathbf{E}_{s}^{\frac{1}{2}}(t) \| \nabla d \|_{L^2}$ ($\lesssim \mathbf{E}_{s}^{\frac{1}{2}}(t) \mathbf{D}_{s}^{\frac{1}{2}}(t)$), because the average $\bar{d}$ does not vanish (the Poincaré inequality fails). Inspired by the micro-macro decomposition method developed for the hydrodynamic limits of kinetic equations (see \cite{Guo-CPAM-2006} for the Boltzmann equation and \cite{JL-APDE-2022,JLZ-ARMA-2023} for the Vlasov–Maxwell–Boltzmann equations, for instance), we estimate the average part $\bar{d}$ and the zero-mean part $d-\bar{d}$ separately, corresponding respectively to the macroscopic and microscopic parts in kinetic theory. The key observation is the equation \eqref{d1} for $\bar{d}$, i.e.,
$$
\partial_{t} \bar{d}= \frac{1}{| \mathbb{T}^3 |} \int_{\mathbb{T}^3} (\nabla \cdot u)( {d}-\bar{d}) \d x + \frac{ 2 A \lambda }{| \mathbb{T}^3 |} \int_{\mathbb{T}^3} | \nabla d |^{2}(d-\bar{d}+M_e + \bar{d}) \d x,
$$
where the right-hand side is nonlinear. This is the required property.

\subsubsection{Sketch of proofs}

The first result of this paper is the local well-posedness of \eqref{CMHED-O} with large smooth initial data. We first derive the local-in-time a priori estimate \eqref{Ener-Bnds-Local} in Proposition \ref{Prop-Loc} by employing the nonlinear energy approach. Due to the nonlinear constraint $|M(t,x)| = 1$ for the magnetization, we carefully design a nonlinear iterative approximation scheme \eqref{Ite-App-Eq} for the NS-LLG system \eqref{CMHED-O}. The nonlinear iterative approximation scheme was initially developed in \cite{JL-SIMA-2019} to study the local well-posedness of the hyperbolic Ericksen-Leslie liquid crystal model. It was then successfully applied in the works \cite{JLL-JDE-2023, JLT-M3AS-2020}. Based on the approximation \eqref{Ite-App-Eq} and the a priori estimate \eqref{Ener-Bnds-Local}, we establish the local well-posedness of \eqref{CMHED-O} by a continuity argument.

The second result of the present work is to establish the global-in-time existence of solutions to \eqref{CMHED-O} in the absence of an external magnetic field and near a constant equilibrium. As stated in the second part of Subsection~\ref{Subsec:n}, we exploit the additional damping effect of the deformation gradient \(F\) by employing the curl-free property of \(F^{-1} - I\) together with the initial structural condition \(\rho_0 \det F_0 = 1\). More precisely, the damping effect originates from the Cauchy-Green stress tensor \(\nabla \cdot (\rho F F^\top)\) for Hookean linear elasticity. We then focus on the reformulated system \eqref{CMEH-eqGd}, which features an additional explicit damping structure \(\Delta \psi\).

Our next goal is to derive the global energy estimate \eqref{close-eq21} in Proposition \ref{Prop-Global} for the instant energy functional \( \mathbb{E}_{s; \delta,\eta,\epsilon} (t) \) and the instant dissipation rate \( \mathbb{D}_{s; \delta,\eta,\epsilon} (t) \). As stated in the last part of Subsection~\ref{Subsec:n}, we estimate separately the average part \( \bar{d} \) and the zero-average part \( d - \bar{d} \) of the fluctuated magnetization \( d \). These estimates are established in Lemma~\ref{Lmm-G1} and Lemma~\ref{Lmm-G2}, respectively. Moreover, we carry out the estimate of \( \nabla {d} \) in Lemma~\ref{Lmm-G3}, thereby completing the estimate for \( d \).  

We then derive the estimate for \(\theta\) (the fluctuated density) in Lemma~\ref{Lmm-G4} and the estimate for \(u\) (the velocity) in Lemma~\ref{Lmm-G5}. We emphasize that there is a quantity \(\sum_{|m| \leq s} \langle P'(1+\theta) \nabla \partial^m \theta, \partial^m u \rangle\), which cancels in the \(\theta\) and \(u\) estimates. Moreover, the interaction dissipation rate \(-\sum_{|m| \leq s} \langle \nabla \partial^m \psi, \nabla \partial^m u \rangle\) is contained in the \(u\)-estimate. However, in order to obtain the damping effect for \(\theta\) from the pressure gradient \(\nabla P(1+\theta)\), we establish Lemma~\ref{Lmm-G6}. At this stage, the damping structure \(a \gamma \| \nabla \theta \|^2_{H^{s-1}_{\w (\theta)}}\) is obtained. The cost, however, is the generation of the interaction energy \(\|u + \nabla \theta \|_{H^{s-1}}^2 - \|u\|_{H^{s-1}}^2 - \|\nabla \theta \|_{H^{s-1}}^2\) and the interaction dissipation rates 
\[
- c_1 \|\nabla \psi \|_{H^s} \| \nabla \theta \|_{H^{s-1}_{\w (\theta)}} 
\quad \text{and} \quad
- c_0 \bigl[ \mu \| \nabla u\|_{H^s} \| \nabla \theta \|_{H^{s-1}_{\w (\theta)}} + (\mu + \xi) \|\nabla \cdot u\|_{H^s} \| \nabla \theta \|_{H^{s-1}_{\w (\theta)}} \bigr].
\]
These interaction quantities are undesirable.  

It remains to estimate the unknown \(\psi\) (corresponding to the deformation gradient tensor \(F\)). The damping effect \(\Delta \psi\) is contained in the \(u\)-equation but not in the \(\psi\)-equation. In order to exploit the damping \(\Delta\psi\), we transfer it to the \(\psi\)-equation by introducing the function \(w = \mu (u -\bar{u}) - (\psi-\bar{\psi})\). Consequently, the quantity \(C_2\) associated with \(w\) (see \eqref{gpsi-eq6}) must be controlled using the relation \eqref{gpsi-eq7} (the reformulation of the \(u\)-equation). Unfortunately, this process results in the interaction energy
\[
-\frac{2}{\mu} \sum_{|m|\leq s} \langle \partial^m (\psi-\bar{\psi}), \partial^m u \rangle
\]
and the interaction dissipation rates
\begin{align*}
  & -\sum_{|m|\leq s} \left\langle \nabla \partial^{m} \psi, \nabla \partial^{m} u \right\rangle,\quad
-\frac{\mu + \xi}{\mu} \sum_{|m|\leq s} \left\langle \nabla \cdot \partial^{m} \psi, \nabla \cdot \partial^{m} u \right\rangle, \\
& -c_{1} a \gamma \left( \mu \| \nabla u \|_{H^{s}}+\| \nabla \psi \|_{H^{s}} \right) \| \nabla \theta \|_{H^{s - 1}_{\w (\theta)}};
\end{align*}
see Lemma~\ref{Lmm-G7}.

In order to control the interaction quantities, we construct the instant energy \( \mathbb{E}_{s; \delta,\eta,\epsilon} (t) \) and the instant dissipation rate \( \mathbb{D}_{s; \delta,\eta,\epsilon} (t) \), depending on the parameters \( \delta, \eta, \epsilon > 0 \). In Lemma~\ref{Lmm-Glob-ED}, we show that there exist constants \( \delta, \eta, \epsilon > 0 \) such that \( \mathbb{E}_{s; \delta,\eta,\epsilon} (t) \) and \( \mathbb{D}_{s; \delta,\eta,\epsilon} (t) \) are equivalent to the energy \( \mathbf{E}_{s}(t) \) (see \eqref{close-eq5-1}) and the dissipation rate \( \mathbf{D}_{s}(t) \) (see \eqref{close-eq6-1}), respectively. Here \( \mathbf{E}_{s}(t) \) and \( \mathbf{D}_{s}(t) \) are both positive. Consequently, the global estimate \eqref{close-eq21} in Proposition~\ref{Prop-Global} is established.

Finally, based on \eqref{close-eq21}, we extend the local solution constructed in Theorem~\ref{Thm1} to the time interval $[0,\infty)$ by a continuity argument. This completes the proof of Theorem~\ref{Thm2}.

\subsection{Organization of this paper}

In the next section, we derive the model \eqref{CMHED-O} using EnVarA. Section~\ref{Sec:3} provides the proof of local well-posedness of the system \eqref{CMEH-eq}. In Section~\ref{Sec:4}, we first transform the fluctuation system \eqref{CMEH-eqG} into the form \eqref{CMEH-eqGd}, which explicitly exhibits the damping mechanism. Then we state the global energy estimate \eqref{close-eq21} in Proposition~\ref{Prop-Global} and prove the global existence of the fluctuation system \eqref{CMEH-eqGd} with small initial data. In Subsection~\ref{Sec:Global-es}, we derive the global energy estimates and thereby prove Proposition~\ref{Prop-Global}.

\section{Derivation of system \eqref{CMHED-O} by the energetic variational approach}\label{Sec:EnVarA}

\subsection{EnVarA for the compressible Navier-Stokes equations with magnetoelastic mechanics}

Define the action functional $\mathcal{A} (x) = \int_0^T (\mathcal{K} - \mathcal{F}) \d t$, where the kinetic energy $\mathcal{K}$ and the Helmholtz free energy $\mathcal{F}$ are given in \eqref{Def-K} and \eqref{Def-F}, respectively. In Lagrangian coordinates, the action functional can be rewritten as
  \begin{align} \label{eq:energy-lagr}
    \no \mathcal{A} (x) = & \int_0^T (\mathcal{K} - \mathcal{F}) \d t \\ 
    = & \int_0^T \int_{\Omega^0} \tfrac{1}{2} \rho_0(X) |\dot x|^2 \d X \d t - \int_0^T \int_{\Omega^0} \left[ w \left( \frac{\rho_0(X )}{\det F} \right) + A|\nabla_X M(X,t) F^{-1}|^2 \right. \\
        \no & \qquad \qquad \qquad \qquad \qquad \left. + \frac{W(F)}{\det F} - \mu_0 M(X,t) \cdot H_{ext}(x(X,t)) \right] \det F \d X \d t \,, 
  \end{align}
  where $\dot{x} = \dot{x} (X, t) = \frac{\d x}{\d t} (X, t)$. Using the relation $\rho(x (X, t),t) = \frac{\rho_0(X) }{\det F}$ and integration by parts, \eqref{eq:energy-lagr} yields
  \begin{align}\label{eq:K-vari-x}
    \delta_x \int \mathcal{K} \d t =\ & \int_0^T \int_{\Omega^0} \rho_0(X) \dot{x} \cdot \tfrac{\d}{\d t} \delta x \d X \d t \\\no
    =\ & - \int_0^T \int_{\Omega^0} \frac{\rho_0(X) }{\det F} \tfrac{\d }{\d t} v \cdot \delta x \det F \d X \d t \\\no
    =\ & - \int_0^T \int_{\Omega^t} \rho(x,t) \tfrac{\d }{\d t} v \cdot \delta x \d x \d t \\\no
    =\ & \skpt{-\rho(\p_t v + v \cdot \nabla_x v)}{\delta x}_{L^2_{t,x}}.
  \end{align}
  Here $\delta_x$ denotes the spatial variation, and $\delta x$ is a function such that $x + h \delta x$ is a variation of $x$ as $h \to 0$.  Moreover, $\langle \cdot, \cdot \rangle_{L^2_{t,x}}$ denotes the standard $L^2$-inner product in $L^2_{t,x}$ with respect to the variables $t$ and $x$.

For the spatial variation of the Helmholtz free energy, we decompose it into four terms: $\mathcal{F} = \mathcal{F}_0 + \mathcal{F}_1 + \mathcal{F}_2 + \mathcal{F}_3$. We then have
\begin{align}
\mathcal{F}_0 = \int_{\Omega^t} w(\rho(x,t)) \d x 
= \int_{\Omega^0} w \left( \frac{\rho_0(X )}{\det F} \right) \det F \d x.
\end{align}
The variation of the above expression yields
  \begin{align}\label{eq:F0-vari-x}
    \no \delta_x \int_0^T \mathcal{F}_0 \d t 
    =\ & \delta_x \int_0^T \int_{\Omega^0} w \left( \frac{\rho_0(X )}{\det F} \right)  \det F \d X \d t \\\no
    =\ & \int_0^T \int_{\Omega^0} \left[ -w'(\tfrac{\rho_0(X )}{\det F}) \frac{\rho_0(X )}{\det F} + w(\tfrac{\rho_0(X )}{\det F}) \right] \;\tr(\nabla_X (\delta x) F^{-1}) \det F \d X \d t 
    \\
    =\ & \int_0^T \int_{\Omega^0} \left[ -w'(\tfrac{\rho_0(X )}{\det F}) \frac{\rho_0(X )}{\det F} + w(\tfrac{\rho_0(X )}{\det F}) \right] \nabla_x \cdot (\delta x) \det F \d X \d t 
    \\\no
    =\ & \int_0^T \int_{\Omega^t} \left[ -w'(\rho(x,t)) \rho(x,t) + w(\rho(x,t)) \right] \nabla_x \cdot (\delta x) \d x \d t 
    \\\no
    =\ & \int_0^T \int_{\Omega^t} \nabla \left[ w'(\rho) \rho - w(\rho) \right] \delta x \d x \d t \\\no
    =\ & \skpa{\nabla \left[ w'(\rho) \rho - w(\rho) \right] }{\delta x}_{L^2_{t,x}} 
    = \skpa{\nabla p(\rho) }{\delta x}_{L^2_{t,x}}\,, 
  \end{align}
where we have used the relation $p(\rho) = w'(\rho) \rho - w(\rho) = a \rho^\gamma $. 

Next, we have
  \begin{align}\label{eq:F1-vari-x}
    \delta_x \int_0^T \mathcal{F}_1 \d t 
    =\ & \delta_x \int_0^T \int_{\Omega^0} A|\nabla_X M(X,t) F^{-1}|^2 \det F \d X \d t \\\no
    =\ & \int_0^T \int_{\Omega^0} 2 A \left[ \nabla_X M(X,t) F^{-1} \right] : \nabla_X M(X,t) \delta_x F^{-1} \det F \d X \d t 
          \\[3pt] \no & \qquad 
          + \int_0^T \int_{\Omega^0} 2 A |\nabla_X M(X,t) F^{-1}| \delta_x \det F \d X \d t \\\no
    =\ & \int_0^T \int_{\Omega^0} 2 A \left[ \nabla_X M(X,t) F^{-1} \right] : \nabla_X M(X,t) \left[ - F^{-1} \nabla_X (\delta x) F^{-1} \right] \det F \d X \d t 
          \\[3pt] \no & \qquad 
          + \int_0^T \int_{\Omega^0} A |\nabla_X M(X,t) F^{-1}|^2 \;\tr \left[ \nabla_X (\delta x) F^{-1} \right] \det F \d X \d t \\\no
    =\ & -\int_0^T \int_{\Omega^t} 2 A \nabla M : [\nabla M (\nabla \delta x)] \d x \d t 
          + \int_0^T \int_{\Omega^t} A |\nabla M|^2 (\nabla \cdot \delta x) \delta x \d x \d t \\\no
    =\ & \int_0^T \int_{\Omega^t} 2 A \nabla \cdot (\nabla M \odot \nabla M) \delta x \d x \d t 
          - \int_0^T \int_{\Omega^t} A \nabla (|\nabla M|^2) \cdot \delta x \d x \d t \\\no
    =\ & \skpa{\nabla \cdot (2 A \nabla M \odot \nabla M) - \nabla (A |\nabla M|^2)}{\delta x}_{L^2_{t,x}}.
  \end{align}
A similar argument yields
  \begin{align}\label{eq:F2-vari-x}
    \delta_x \int_0^T \mathcal{F}_2 \d t 
    =\ & \delta_x \int_0^T \int_{\Omega^0} W(F) \d X \d t \\\no
    =\ & \int_0^T \int_{\Omega^0} W'(F) \delta_x \frac{\p x(X,t)}{\p X} \d X \d t 
    = \int_0^T \int_{\Omega^0} W'(F) \left[ \nabla_X (\delta x) \right] \d X \d t  \\\no
    =\ & \int_0^T \int_{\Omega^0} W'(F) \left[ F^\top \nabla (\delta x) \right] \d X \d t  
    = \int_0^T \int_{\Omega^t} \frac{W'(F)F^\top}{\det F} \left[ \nabla (\delta x) \right] \d x \d t  \\\no
    =\ & -\int_0^T \int_{\Omega^t} \nabla \cdot \left( \frac{W'(F)F^\top}{\det F} \right) \delta x \d x \d t = \skpa{-\nabla \cdot \left( \frac{W'(F)F^\top}{\det F} \right) }{\delta x}_{L^2_{t,x}} \,, 
  \end{align}
and 
  \begin{align}\label{eq:F3-vari-x}
    \delta_x \int_0^T \mathcal{F}_3 \d t 
    =\ & - \delta_x \int_0^T \int_{\Omega^0} \mu_0 M(X,t) \cdot H_{ext}(x(X,t)) \det F \d X \d t \\\no
    =\ & - \int_0^T \int_{\Omega^0} \mu_0 M(X,t) \cdot \nabla^\top H_{ext}(x(X,t)) \delta x \det F \d X \d t  \\[3pt] \no & \quad 
         - \int_0^T \int_{\Omega^0} \mu_0 M(X,t) \cdot H_{ext}(x(X,t)) \delta_x \det F \d X \d t \\\no
    =\ & - \int_0^T \int_{\Omega^0} \mu_0 M(X,t) \cdot \nabla^\top H_{ext}(x(X,t)) \delta x \det F \d X \d t  \\[3pt] \no & \quad 
         - \int_0^T \int_{\Omega^0} \mu_0 M(X,t) \cdot H_{ext}(x(X,t)) \;\tr \left[ \nabla_X (\delta x) F^{-1} \right] \det F \d X \d t \\\no 
    =\ & - \int_0^T \int_{\Omega^t} \mu_0 M(x,t) \cdot \nabla^\top H_{ext}(x,t) \delta x \d x \d t  
      \\[3pt] \no & \quad 
          - \int_0^T \int_{\Omega^t} \mu_0 M(x,t) \cdot H_{ext}(x,t) \; [\nabla \cdot (\delta x)] \d x \d t \\\no
    =\ & - \int_0^T \int_{\Omega^t} (\mu_0 \nabla^\top H_{ext} M) \cdot \delta x \d x \d t 
          + \int_0^T \int_{\Omega^t} \nabla \left( \mu_0 M \cdot H_{ext} \right) \cdot (\delta x) \d x \d t \\\no
    =\ & \skpa{- (\mu_0 \nabla^\top H_{ext} M) + \nabla \left( \mu_0 M \cdot H_{ext} \right) }{\delta x}_{L^2_{t,x}} \,. 
  \end{align}
Combining \eqref{eq:F0-vari-x}, \eqref{eq:F1-vari-x}, \eqref{eq:F2-vari-x}, and \eqref{eq:F3-vari-x}, we obtain
  \begin{align}\label{eq:F-vari-x}
    \delta_x \int_0^T \mathcal{F} \d t 
    & = \skpa{\nabla \cdot \left( 2 A \nabla M \odot \nabla M - \tfrac{W'(F)F^\top}{\det F} \right) - \mu_0 \nabla^\top H_{ext} M }{\delta x}_{L^2_{t,x}} \\[3pt] \no & \qquad 
        + \skpa{\nabla \left( p(\rho) - A |\nabla M|^2 + \mu_0 M \cdot H_{ext} \right) }{\delta x}_{L^2_{t,x}} 
        \,.
  \end{align}

On the other hand, for the variation of the dissipation, recalling the viscous dissipation tensor $\Sigma (v) = \mu (\nabla v + \nabla^\top v) + \xi (\nabla \cdot v) I $, we obtain 
  \begin{align}\label{eq:D-vari-v}
    \delta_{v} \left( \tfrac{1}{2}\mathcal{D} \right) 
    & = \int_{\Omega^t} \left\{ \mu \nabla v \nabla (\delta v) + (\mu + \xi) (\nabla \cdot v)\; [\nabla \cdot (\delta v)] \right\} \d x 
    \\\no & 
    = \skpa{-\nabla \cdot \left[ \mu \nabla v + (\mu + \xi) (\nabla \cdot v) \Id \right] }{\delta v}_{L^2_x} 
    = \skpa{-\nabla \cdot \Sigma (v) }{\delta v}_{L^2_x} \,,
  \end{align}
  where $\mathcal{D}$ is given in \eqref{eq:dissipation-u-M}. It then follows from the force balance 
\begin{align}\label{eq:force-balance-x}
  \left( \delta_x \int_0^T \mathcal{A} \d t \right)_{ L^2_{t,x} }  = \left( \delta_x \int_0^T \mathcal{K} \d t \right)_{ L^2_{t,x} } - \left( \delta_x \int_0^T \mathcal{F} \d t \right)_{ L^2_{t,x} } =  \left( \delta_{\dot{x}} ( \tfrac{1}{2}\mathcal{D} ) \right)_{ L^2_{x} }
\end{align} 
that
  \begin{multline}\label{eq:momentum}
    \rho (\p_t v + v \cdot \nabla v) + \nabla \left( p(\rho) - A |\nabla M|^2 + \mu_0 M \cdot H_{ext} \right) 
    \\ 
    = \nabla \cdot \left( \Sigma (v) - 2 A \nabla M \odot \nabla M + \tfrac{W'(F)F^\top}{\det F} \right) + \mu_0 \nabla^\top H_{ext} M \,. 
  \end{multline}
  Here the subscripts $X = L^2_{t,x}$ or $L^2_x$ in the variational form $\bigl( \delta_x \int_0^T \mathcal{Y} \,\mathrm{d} t \bigr)_X$ for $\mathcal{Y} = \mathcal{A}, \mathcal{K}, \mathcal{F}, \frac{1}{2}\mathcal{D}$ indicate the spaces in which the variations are carried out.
  
Define $P (\rho, M, \nabla M, H_{ext}) = p(\rho) - A |\nabla M|^2 + \mu_0 M \cdot H_{ext} $ and $\mathcal{T} (v, F, M, \nabla M, H_{ext}) = \Sigma (v) - 2 A \nabla M \odot \nabla M + \tfrac{W'(F)F^T}{\det F} $. Then we immediately obtain the equation for the velocity field \(v\):
  \begin{align}
    \rho (\p_t v + v \cdot \nabla v) + \nabla P = \nabla \cdot \mathcal{T} + \mu_0 \nabla^\top H_{ext} M \,. 
  \end{align}
In addition, the mass equation reads
  \begin{align}
     \p_t \rho + \nabla \cdot (\rho v) = 0 \,. 
  \end{align}

\subsection{EnVarA for LLG equation of magnetization}\label{sec:envara_for_LLG}

We begin with the balance relation of angular momentum (see, e.g., \cite{DALS18,WXL-13arma}): 
\begin{align}\label{eq:angular-momentum-beta}
  g + h = \beta M \,, 
\end{align}
where $g$ is the kinematic transport of the magnetization vector $M$, $h = \delta_M \mathcal{F} $ denotes the magnetic field, and the term $\beta$ on the right-hand side is the Lagrangian multiplier for the unit-length constraint $|M|=1$. Note that \(h\) differs from the effective field \(H_{\text{eff}}\) by a minus sign, i.e., \(H_{\text{eff}} = -h\). 

In the absence of damping, the kinematics of the magnetization vector \(M\) can be written as
\begin{align}
\dot M = \frac{\d }{\d t} M = -\gamma M \times H_{\eff} = \gamma M \times h \,, 
\end{align}
with $\gamma >0$. The reader is referred to \cite{Gilbert-04ieee} for further details.  Consequently, the total kinematic transport of \(M\) can be expressed as
\begin{align}\label{eq:transport-M0}
\mathring{M} = \dot M - \gamma M \times h = \p_t M + v \cdot \nabla M - \gamma M \times h \,, 
\end{align}
which suggests that we write $g = \tfrac{1}{\lambda} \mathring{M} $. Multiplying \eqref{eq:angular-momentum-beta} by \(M\) and using \(g \cdot M = 0\) (a consequence of the unit-length constraint \(|M|=1\)), we obtain \(\beta = h \cdot M\).  Substituting this back into \eqref{eq:angular-momentum-beta} yields
\begin{align}\label{eq:transport-M}
\frac{1}{\lambda} \mathring{M} = (h \cdot M) M - h 
\end{align}
Using the vector triple product identity \(a \times (b \times c) = (a \cdot c) b - (a \cdot b) c\), we have 
$$(h \cdot M) M - h = (M \cdot h) M - (M \cdot M)h = M \times (M \times h).$$ 
Thus, \eqref{eq:transport-M} reduces to
\begin{align}\label{eq:angular-momentum-h}
\p_t M + v \cdot \nabla M - \gamma M \times h = \lambda M \times (M \times h).
\end{align}

We next consider the variation with respect to the magnetization variable \(M\):
  \begin{align}\label{eq:KF-vari-M}
    \delta_M \int_0^T \mathcal{K} \d t & = 0, \\ 
    \delta_M \int_0^T \mathcal{F} \d t
    & = \int_0^T \int_{\Omega} \left[ w(\rho) + A|\nabla M|^2 + \frac{W(F)}{\det F} - \mu_0 M \cdot H_{ext} \right] \d x \d t
    \\\no & 
    = \int_0^T \int_{\Omega} \left( 2 A \nabla M : \nabla \delta M - \mu_0 H_{ext} \cdot \delta M \right) \d x \d t
    \\\no & 
    = \skpa{-2 A \Delta M - \mu_0 H_{ext} }{\delta M}_{L^2_{t,x}} \,.
  \end{align}
Thus $h = \delta_M \mathcal{F} = -(2 A \Delta M + \mu_0 H_{ext})$. Substituting this into \eqref{eq:angular-momentum-h} we finally obtain
\begin{align}\label{eq:angular-momentum}
  \p_t M + v \cdot \nabla M = - \gamma M \times (2 A \Delta M + \mu_0 H_{ext}) - \lambda M \times M \times (2 A \Delta M + \mu_0 H_{ext}) \,. 
\end{align}

\begin{remark}[Variation on magnetization dissipation]\label{Rmk-2.1}
  We may explain \eqref{eq:angular-momentum-beta} from the perspective of EnVarA.  Consider the magnetization dissipation functional
  \begin{align}\label{eq:dissipation-M-dt}
    \mathcal{D}(M) = \int_{\Omega} \tfrac{1}{\lambda} |M \times \mathring{M}|^2 \d x \,.
  \end{align}
The variational process (maximum dissipation principle) with respect to \(\mathring{M}\) of the total transport derivative yields
  \begin{align}
    \delta_{\mathring{M}} (\tfrac{1}{2}\mathcal{D}) 
    & = \int_{\Omega} \tfrac{1}{\lambda} (M \times \mathring{M}) \cdot [M \times (\delta \mathring{M})] \d x 
    \\\no & 
    = \int_{\Omega} \tfrac{1}{\lambda} [(M \times \mathring{M}) \times M] \cdot (\delta \mathring{M}) \d x
    = \skpa{ \tfrac{1}{\lambda} (M \times \mathring{M}) \times M }{\delta \mathring{M}}_{L^2_{x}} 
    \\\no & 
    = \skpa{ \tfrac{1}{\lambda} [(M \cdot M) \mathring{M} - (M \cdot \mathring{M}) M] }{\delta \mathring{M}}_{L^2_{x}} 
    = \skpa{ \tfrac{1}{\lambda} \mathring{M} }{\delta \mathring{M}}_{L^2_{x}} \,.
  \end{align}

Analogously to \eqref{eq:force-balance-x}, the force balance relation
$$ \left( \delta_M \int_0^T \mathcal{K} \d t \right)_{L^2_{t,x }} - \left( \delta_M \int_0^T \mathcal{F} \d t \right)_{L^2_{t,x }} = \left( \delta_{\mathring{M}} \tfrac{1}{2}\mathcal{D} \right)_{L^2_{x }}$$ 
implies that
\begin{align}
  \delta_M \mathcal{F} + \tfrac{1}{\lambda} \mathring{M} = h + g = 0 \,.
\end{align}
This is exactly the angular momentum balance equation \eqref{eq:angular-momentum-beta} without the Lagrange multiplier on the right-hand side.

Incidentally, by taking the tensor product of \eqref{eq:transport-M} with \(M\), one readily verifies the equivalence between the magnetization dissipation rate functionals in \eqref{eq:dissipation-M-dt} and that in \eqref{eq:dissipation-u-M}.
\end{remark}

\section{Local well-posedness with large initial data} \label{Sec:3}

In this section we study the local well-posedness of system \eqref{CMEH-eq} with initial data \eqref{IC-1} using energy methods.  We first derive a priori estimates for system \eqref{CMEH-eq} in higher-order Sobolev spaces.  Then, following the works \cite{JLL-JDE-2023,JL-SIMA-2019,JLT-M3AS-2020}, we construct a nonlinear iterative approximation scheme to prove the local existence of solutions to \eqref{CMEH-eq}.

\subsection{A priori estimates for local well-posedness}

In this subsection we derive the a priori estimates needed for the proof of local well-posedness of system \eqref{CMEH-eq}.  Recalling the definitions of $\mathcal{E}_s(t)$ and $\mathcal{D}_s(t)$ in \eqref{EDs-loc}, we set
\begin{equation*}
	\begin{aligned}
		\mathcal{E}_s (t) & := \| \rho \|^2_{H^s} +  \| F \|^2_{H^s} + \| v \|^2_{H^s_\rho} +\|\nabla M\|^2_{H^s}\,, \\
		\mathcal{D}_s (t) & := \mu \| \nabla v \|^2_{H^s} + (\mu+\xi)\|\nabla\cdot v\| ^{2}_{H^{s}}+2\lambda A\|\Delta M\|^2_{H^s}  \,.
	\end{aligned}
\end{equation*}
Then we have the following a priori estimates.

\begin{proposition}\label{Prop-Loc}
    Let $s \ge 3$ be an integer and let the external magnetic field satisfy $H_{\text{ext}} \in L^\infty(\mathbb{R}_+; H^s(\mathbb{T}^3))$.  Let $\rho_0(x)$ be the initial density as given in Theorem~\ref{Thm1}.  Assume that $(\rho, v, F, M)$ is a sufficiently smooth solution to system \eqref{CMEH-eq}.  Then there exists a constant $c>0$ such that
\begin{equation}\label{Ener-Bnds-Local}
\begin{aligned}
& \frac{\d}{\d t} \mathcal{E}_s(t) + \mathcal{D}_s(t) \lesssim (1+\|H_{ext}\|^2_{H^s})(1 + \|\tfrac{1}{\rho_0}\|^{s+1}_{L^\infty} e^{c (s+1) t\sup_{\tau \in [0,t]} \mathcal{E}_s^{\frac{1}{2}}(\tau)}) \big( 1 + \mathcal{E}_s^{\frac{s}{2} + \gamma + 3} (t) \big) .
\end{aligned}
\end{equation}
\end{proposition}

Next we focus on the derivation of the above estimate.

\subsubsection{\bf Estimates for $\rho$-evolution}

In this subsection, we present the following lemma.

\begin{lemma}\label{Lmm-rho-loc}
For any multi-index $m \in \mathbb{N}^3$ with $|m| \leq s$ and $s \geq 3$, we have
    \begin{equation}\label{rho-eq}
        \begin{aligned}
            \tfrac{1}{2} \tfrac{\d}{\d t} \|\partial^{m}\rho\|^{2}_{L^{2}} \lesssim \|\tfrac{1}{\rho}\|^{\tfrac{1}{2}}_{L^\infty} \mathcal{E}_s^\frac{3}{2} (t) + \mathcal{E}_s (t) \mathcal{D}_s^\frac{1}{2} (t) \,,
        \end{aligned}
    \end{equation}
    where $\mathcal{E}_s (t)$ and $\mathcal{D}_s (t)$ are defined in \eqref{EDs-loc}.
\end{lemma}

\begin{proof}

We apply the $m$-th order derivative operator to the first $\rho$-equation in \eqref{CMEH-eq} for all $|m| \leq s$, take the $L^2$ inner product with $\partial^m \rho$, and integrate by parts over $x \in \mathbb{T}^3$. This yields
 \begin{equation}\label{rho-eq1}
	\begin{aligned}
		&\tfrac{1}{2} \tfrac{\d}{\d t} \|\partial^{m}\rho\|^{2}_{L^{2}}=-\langle \partial^{m}(v \cdot \nabla \rho),\p^m \rho \rangle-\langle\p^m(\rho\nabla\cdot v),\p^m \rho \rangle=I_1+I_2\,.
	\end{aligned}
\end{equation}
The term $I_1$ can be decomposed as
 \begin{equation}\label{rho-eq00}
	\begin{aligned}
		I_1=&\underset{I_{11}}{\underbrace{-\langle v\cdot\nabla\p^m\rho,\p^m\rho \rangle}}\underset{I_{12}}{\underbrace{ - \sum_{\substack{ m^{\prime}=1},m^{\prime}\leq m} C^{m^{\prime}}_{m} \langle (\p^{m^{\prime}} v \cdot \nabla \partial^{m-m^{\prime}}) \rho,\partial^{m} \rho \rangle}} \\
			&\quad \underset{I_{13}}{\underbrace{ - \sum_{\substack{ 2\leq |m^{\prime}|\leq s-1},m^{\prime}\leq m} C^{m^{\prime}}_{m} \langle (\p^{m^{\prime}} v \cdot \nabla \partial^{m-m^{\prime}}) \rho,\partial^{m} \rho \rangle}}- \underset{I_{14}}{\underbrace{ \langle\p^m v\cdot \nabla \rho,\p^m\rho \rangle}}
			\,.
	\end{aligned}
\end{equation}
Using Hölder's inequality and Sobolev embedding theory, we obtain
 \begin{equation}\label{rho-eqI11}
	\begin{aligned}
		I_{11}= \frac{1}{2}\langle \nabla \cdot v, |\partial^m \rho|^2 \rangle
        \lesssim \| \nabla \cdot v \|_{L^\infty} \| \partial^m \rho \|_{L^2}^2  \lesssim \| v \|_{H^3} \| \partial^m \rho \|_{L^2}^2
        \,,
	\end{aligned}
\end{equation}
and
\begin{equation}\label{rho-eqI12}
	\begin{aligned}
		I_{12} &\lesssim \| \nabla v \|_{L^\infty} \sum_{\substack{m^{\prime}\leq m, \\ |m^{\prime}| = 1}} \| \nabla \partial^{m -m^{\prime}} \rho \|_{L^2} \| \partial^m \rho \|_{L^2} \\
& \quad \lesssim \| \nabla v \|_{L^\infty} \| \rho \|_{H^s} \| \partial^m \rho \|_{L^2} \lesssim \| v \|_{H^3} \| \rho \|_{H^s} \| \partial^m p \|_{L^2}
        \,,
	\end{aligned}
\end{equation}
and
\begin{equation}\label{rho-eqI13}
	\begin{aligned}
		I_{13} \lesssim \sum_{\substack{m^{\prime} \leq m, \\ 2 \leq |m^{\prime}| \leq s - 1}} C_m^{m^{\prime}} \| \partial^{m^{\prime}} v \|_{L^4} \| \nabla \partial^{m - m^{\prime}} \rho \|_{L^4} \| \partial^m \rho \|_{L^2} \lesssim \| v \|_{H^s} \| \rho \|_{H^s} \| \partial^m \rho\|_{L^2}
        \,,
	\end{aligned}
\end{equation}
and
\begin{equation}\label{rho-eqI14}
	\begin{aligned}
		I_{14} \lesssim \| \partial^m v \|_{L^2} \| \nabla \rho \|_{L^\infty} \| \partial^m \rho \|_{L^2} \lesssim \| \rho \|_{H^3} \| \partial^m v \|_{L^2} \| \partial^m \rho \|_{L^2}
        \,.
	\end{aligned}
\end{equation}
Collecting the above estimates, we bound $I_1$ as
 \begin{equation}\label{rho-eq11}
	\begin{aligned}
		I_1=& \lesssim \|v\|_{H^3}\|\p^m\rho\|^2_{L^2}+\|v\|_{H^3}\|\rho\|_{H^s}\|\p^m\rho\|_{L^2}	+\|v\|_{H^s}\|\rho\|_{H^s}\|\p^m\rho\|_{L^2}\\
				&\quad+\|\rho\|_{H^3}	\|\p^m v\|_{L^2}\|\p^m\rho\|_{L^2}	\,.
	\end{aligned}
\end{equation}

Observe that
 \begin{equation}\label{rho-eq00}
	\begin{aligned}
		I_2= & \underset{I_{21}}{\underbrace{-\langle \p^m(\rho\nabla \cdot v),\p^m \rho \rangle}}\underset{I_{22}}{\underbrace{ -\langle \rho \nabla\cdot \p^m v,\p^m\rho \rangle}} \\
		& \underset{I_{23}}{\underbrace{ - \sum_{\substack{ m^{\prime}=1},m^{\prime}\leq m} C^{m^{\prime}}_{m} \langle (\p^{m^{\prime}} \rho\nabla \cdot  \partial^{m-m^{\prime}}) v,\partial^{m} \rho \rangle}} \\
        & \underset{I_{24}}{\underbrace{ - \sum_{\substack{ 2\leq |m^{\prime}|\leq s-1},m^{\prime}\leq m} C^{m^{\prime}}_{m} \langle (\p^{m^{\prime}} \rho\nabla \cdot \partial^{m-m^{\prime}}) v,\partial^{m} \rho \rangle}}	\,.
	\end{aligned}
\end{equation}
By H\"older's inequality and Sobolev embedding, we have
\begin{equation*}
	\begin{aligned}
I_{21}\lesssim \|\nabla \cdot v\|_{L^\infty} \|\partial^m \rho\|_{L^2}^2 \lesssim \|v\|_{H^3} \|\partial^m \rho\|_{L^2}^2\,,
	\end{aligned}
\end{equation*}
and
\begin{equation*}
	\begin{aligned}
I_{22}\lesssim \|\rho\|_{L^\infty} \|\nabla \cdot \partial^m v\|_{L^2} \|\partial^m \rho\|_{L^2} \lesssim \|\rho\|_{H^2} \|\partial^m \rho\|_{L^2} \|\nabla \partial^m v\|_{L^2}\,,
	\end{aligned}
\end{equation*}
and
    \begin{equation*}
	\begin{aligned}
I_{23}\lesssim \sum_{\substack{m^\prime\leq m \\ |m'| = 1}} \|\partial^{m^\prime} \rho\|_{L^\infty} \|\nabla \cdot \partial^{m - m^\prime} v\|_{L^2} \|\partial^m \rho\|_{L^2}^2 
\quad \lesssim \|\rho\|_{H^3} \|v\|_{H^s} \|\partial^m \rho\|_{L^2}^2 \,,
	\end{aligned}
\end{equation*}
and
\begin{equation*}
	\begin{aligned}
I_{24}\lesssim \sum_{\substack{m^\prime \leq m \\ 2 \leq |m^\prime| \leq s-1}} \|\partial^{m^\prime} \rho\|_{L^4} \|\nabla \cdot \partial^{m - m^\prime} v\|_{L^4} \|\partial^m \rho\|_{L^2} 
\quad \lesssim \|\rho\|_{H^s} \|v\|_{H^s} \|\partial^m \rho\|_{L^2}\,.
	\end{aligned}
\end{equation*}
Together with all the above estimates, the quantity $I_2$ can be bounded by
\begin{equation}\label{rho-eq12}
	\begin{aligned}
	I_2&\lesssim\|v\|_{H^3}\|\p^m\rho\|^2_{L^2}+\|\rho\|_{H^2}\|\p^m\rho\|_{L^2}\|\nabla\p^m v\|_{L^2}	+\|\rho\|_{H^3}\|v\|_{H^s}\|\p^m\rho\|_{L^2}\\
		&\quad+\|\rho\|_{H^s}	\| v\|_{H^s}\|\p^m\rho\|_{L^2}	\,.
	\end{aligned}
\end{equation}
Combining \eqref{rho-eq1}, \eqref{rho-eq11}, and \eqref{rho-eq12}, we obtain
\begin{equation}\label{rho-eq-1}
	\begin{aligned}
	\tfrac{1}{2} \tfrac{\d}{\d t} \|\partial^{m}\rho\|^{2}_{L^{2}} &\lesssim\|v\|_{H^s}\|\rho\|_{H^s}\|\p^m\rho\|_{L^2}+\|\rho\|^2_{H^s}\|\nabla\p^m v\|_{L^2}	\\
		&\lesssim \|\tfrac{1}{\rho}\|^{\tfrac{1}{2}}_{L^\infty}\|v\|_{H_\rho^s}\|\rho\|_{H^s}\|\partial^{m}\rho\|_{L^2}
		+\|\rho\|^2_{H^s}\|\nabla\p^m v\|_{L^2}
		 \,,
	\end{aligned}
\end{equation}
where we have used the fact that $\| v \|_{H^s} \lesssim \| \frac{1}{\rho} \|^{1/2}_{L^\infty} \| v \|_{H^s_\rho}$. Together with the definitions of $\mathcal{E}_s (t)$ and $\mathcal{D}_s (t)$ in \eqref{EDs-loc}, the estimate \eqref{rho-eq-1} completes the proof of the lemma.
\end{proof}

\subsubsection{\bf Estimates for $F$-evolution}

In this subsection, we present the following lemma.

\begin{lemma}\label{Lmm-F-loc}
    For any multi-index $m \in \mathbb{N}^3$ with $|m| \leq s$ and $s \geq 3$, we have
    \begin{equation}\label{F-eq}
	\begin{aligned}
		\tfrac{1}{2} \tfrac{\d}{\d t} \|\partial^{m}F\|^{2}_{L^{2}} \lesssim \|\tfrac{1}{\rho}\|^{\tfrac{1}{2}}_{L^\infty} \mathcal{E}_s^\frac{3}{2} (t) + \mathcal{E}_s (t) \mathcal{D}_s^\frac{1}{2} (t) \,,
        \end{aligned}
    \end{equation}
    where $\mathcal{E}_s (t)$ and $\mathcal{D}_s (t)$ are defined in \eqref{EDs-loc}.
\end{lemma}

\begin{proof}

We apply the $m$-th order derivative operator $\partial^m$ to the third equation for $F$ in \eqref{CMEH-eq} for all $|m| \leq s$, take the $L^2$ inner product with $\partial^{m} F$, and integrate by parts over $x \in \mathbb{T}^3$. This yields
 \begin{equation}\label{F-eq1}
	\begin{aligned}
		&\tfrac{1}{2} \tfrac{\d}{\d t} \|\partial^{m}F\|^{2}_{L^{2}} =-\langle \partial^{m}(v \cdot \nabla F),\p^m F \rangle+\langle\p^m(\nabla vF),\p^m F \rangle\,.
	\end{aligned}
\end{equation}
By similar arguments as in \eqref{rho-eq11}, we obtain
 \begin{equation}\label{F-eq11}
	\begin{aligned}
		-\langle \partial^{m}(v \cdot \nabla F),\p^m F \rangle
		& \leq \|v\|_{H^3}\|\p^m F\|^2_{L^2}+\|v\|_{H^3}\|F\|_{H^s}\|\p^mF\|_{L^2}	+\|v\|_{H^s}\|F\|_{H^s}\|\p^m F\|_{L^2}\\
		&\quad+	\|\p^m v\|_{L^2}\|F\|_{H^3}\|\p^mF\|_{L^2}	\,.
	\end{aligned}
\end{equation}
An analogous calculation as in \eqref{rho-eq12} gives
 \begin{equation}\label{F-eq12}
	\begin{aligned}
	\langle\p^m(\nabla vF),\p^m F \rangle & \leq \|v\|_{H^3}\|\p^mF\|^2_{L^2}+\|F\|_{H^2}\|\p^mF\|_{L^2}\|\nabla\p^m v\|_{L^2}	\\
		&\quad+\|F\|_{H^3}	\| v\|_{H^3}\|\p^mF\|_{L^2}+\|v\|_{H^s}\|F\|_{H^3}\|\p^mF\|_{L^2}	\,.
	\end{aligned}
\end{equation}
Combining \eqref{F-eq1},\eqref{F-eq11} and \eqref{F-eq12}, we have
\begin{equation}\label{F-eq-1}
	\begin{aligned}
		\tfrac{1}{2} \tfrac{\d}{\d t} \|\partial^{m}F\|^{2}_{L^{2}} 
	& \leq \|v\|_{H^3}\|\p^m F\|^2_{L^2}+\|v\|_{H^3}\|F\|_{H^s}\|\p^mF\|_{L^2}+\|v\|_{H^s}\|F\|_{H^s}\|\p^m F\|_{L^2}		\\
		&\quad+\|v\|_{H^s}\|F\|_{H^3}\|\p^m F\|_{L^2}+ \|F\|_{H^2}\|\p^m F\|_{L^2}\|\nabla\p^m v\|_{L^2}	\\
		&\lesssim\|v\|_{H^s}\|F\|_{H^s}\|\p^m F\|_{L^2}+\|F\|^2_{H^s}\|\nabla\p^mv\|_{L^2}\\
		&\lesssim \|\tfrac{1}{\rho}\|^{\tfrac{1}{2}}_{L^\infty}\|v\|_{H^s_\rho}\|F\|_{H^s}\|\p^m F\|_{L^2}+\|F\|^2_{H^s}\|\nabla\p^mv\|_{L^2}
		\,.
	\end{aligned}
\end{equation}
Together with the definitions of $\mathcal{E}_s(t)$ and $\mathcal{D}_s(t)$, estimate \eqref{F-eq-1} implies the bound \eqref{F-eq}. This completes the proof of the lemma.
\end{proof}

\subsubsection{\bf Estimates for $v$-evolution}

In this subsection, we present the following lemma.

\begin{lemma}\label{Lmm-v-loc}
    For any multi-index $m \in \mathbb{N}^3$ with $|m| \leq s$ and $s \geq 3$, we have
    \begin{equation}\label{v-equation}
	\begin{aligned}
&\frac{1}{2} \frac{\d}{\d t} \|\partial^m v\|_{L^2_\rho}^2 + \mu \|\nabla \partial^m v\|_{L^2}^2 + (\mu + \xi) \|\nabla \cdot \partial^m v\|_{L^2}^2 \\
\lesssim & \| \frac{1}{\rho} \|_{L^\infty}^{\frac{3}{2}} \mathcal{E}_s^2 (t) + \| \frac{1}{\rho} \|_{L^\infty}^{\frac{1}{2}} \|\partial_t v\|_{H^{s - 1}} \mathcal{E}_s (t) \\
& + \big[ \mathcal{E}_s^\frac{\gamma}{2} (t) + \mathcal{E}_s (t) + \mathcal{E}_s^\frac{3}{2} (t) + ( 1 + \mathcal{E}_s^\frac{1}{2} (t) ) \|H_{ext}\|_{H^s} \big] \mathcal{D}_s^\frac{1}{2} (t) \,,
	\end{aligned}
\end{equation}
where $\mathcal{E}_s (t)$ and $\mathcal{D}_s (t)$ are defined in \eqref{EDs-loc}.
\end{lemma}

\begin{proof}

We apply the $m$-th order derivative operator $\partial^m$ to the second equation for $v$ in \eqref{CMEH-eq} for all $|m| \leq s$, take the $L^2$ inner product with $\partial^{m} v$, and integrate by parts over $x \in \mathbb{T}^3$. This yields
\begin{equation}\label{v-eq0}
	\begin{aligned}
&\underset{V_1}{\underbrace{\langle	\p^m[\rho (\p_t v + v \cdot \nabla v)],\p^m v\rangle}} + \underset{V_2}{\underbrace{\langle \nabla \p^m \left( P(\rho) - A |\nabla M|^2 + \mu_0 M \cdot H_{ext} \right),\p^m v\rangle }} \\
    & + \mu\|\nabla \p^m v\|^2_{L^2}+(\mu+\xi)\|\nabla \cdot\p^m v\|^2_{L^2}\\
	= & \underset{V_3}{\underbrace{\langle \p^m(\rho F:\nabla F^\top),\p^m v\rangle}}\underset{V_4}{\underbrace{-2A\langle \nabla \cdot\p^m ( \nabla M \odot \nabla M) ,\p^mv \rangle}} \\
    & + \underset{V_5}{\underbrace{\mu_0\langle  \p^m(\nabla H_{ext})^\top M,\p^mv \rangle	}}
	\,.
	\end{aligned}
\end{equation}
We then estimate the quantities $V_i$ ($1 \leq i \leq 5$) term by term.

Observe that $V_1$ can be further decomposed as
\begin{equation}\label{v-eq1}
	\begin{aligned}
	V_1	=&\langle	\p^m[\rho (\p_t v + v \cdot \nabla v)],\p^m v\rangle \\
    = & \langle	\rho(\p_t\p^mv+\p^m(v\cdot\nabla v)),\p^m v  \rangle+
	\sum_{\substack{ 0\neq m^{\prime}\leq m}}\mathbf{1}_{|m|\geq 1} C^{m^{\prime}}_{m} \langle \p^{m^{\prime}} \rho (\p_t v+v\cdot\nabla v),\partial^{m} v \rangle\\
    = & \frac{1}{2} \frac{\d}{\d t} \|\p^mv\|^2_{L^2_p}\underset{V_{11}}{\underbrace{-\frac{1}{2}\langle\p_t\rho,|\p^mv|^2\rangle}}+\underset{V_{12}}{\underbrace{\langle\rho\p^m(v\cdot\nabla) ,\p^mv \rangle}}\\
    & +\underset{V_{13}}{\underbrace{\sum_{\substack{ 0\neq m^{\prime}\leq m}}\mathbf{1}_{|m|\geq 1} C^{m^{\prime}}_{m} \langle \p^{m^{\prime}} \rho\partial^{m} v  ,\partial^{m} v \rangle}} \\
    & +\underset{V_{14}}{\underbrace{\sum_{\substack{ 0\neq m^{\prime}\leq m}}\mathbf{1}_{|m|\geq 1} C^{m^{\prime}}_{m} \langle \p^{m^{\prime}} \rho\partial^{m-m^{\prime}} (v\cdot\nabla v,  ,\partial^{m} v \rangle}}
	\,.
	\end{aligned}
\end{equation}
Together with the $\rho$-equation in \eqref{CMEH-eq}, the term $V_{11}$ can be bounded using H\"older's inequality and Sobolev embedding:
\begin{equation}\label{v-eq11}
	\begin{aligned}
-V_{11}&= \frac{1}{2}\langle\p_t\rho,|\p^mv|^2\rangle=-\frac{1}{2}\langle v\cdot \nabla \rho+\rho\nabla\cdot v,|\p^m v|^2\rangle
\\&\lesssim(\||v\|_{L^\infty}\||\nabla\rho\|_{L^\infty}+\||\rho\|_{L^\infty}\||\nabla v\|_{L^\infty})\||\p^mv\|^2_{L^2}\\&\lesssim\|v\|_{H^3}\|\rho\|_{H^3}\|\p^mv\|^2_{L^2} \lesssim\|\frac{1}{\rho}\|^{\frac{3}{2}}_{L^\infty}\|v\|_{H^3_\rho}\|\rho\|_{H^3}\|\p^mv\|^2_{L^2_\rho}
\,.
	\end{aligned}
\end{equation}
The term $V_{12}$ can be split into finer parts:
	\begin{align}\label{v-eq12}
\no -V_{12} & = -\langle\rho\p^m(v\cdot\nabla) ,\p^mv \rangle\\
& =\underset{V_{121}}{\underbrace{-\langle\rho v\cdot\nabla\p^m v ,\p^mv \rangle}}\underset{V_{122}}{\underbrace{-\sum_{\substack{ m^{\prime}\leq m,|m^{\prime}|=1}}C^{m^{\prime}}_m\langle\rho \p^{m^{\prime}}v\cdot\nabla\p^{m-m^{\prime}} v ,\p^mv \rangle}}\\
\no & \underset{V_{123}}{\underbrace{-\sum_{\substack{ m^{\prime}\leq m,2\leq|m^{\prime}|\leq s-1}}C^{m^{\prime}}_m\langle\rho \p^{m^{\prime}}v\cdot\nabla\p^{m-m^{\prime}} v ,\p^mv \rangle}}\underset{V_{124}}{\underbrace{-\langle\rho\p^m v \cdot\nabla v,\p^mv \rangle}} \,.
	\end{align}
H\"older's inequality and Sobolev embedding then give
\begin{equation}\label{v-eq121}
	\begin{aligned}
V_{121}&\lesssim \|\nabla \cdot (\rho v)\|_{L^\infty} \|\partial^m v\|_{L^2}^2 \lesssim \|\frac{1}{\rho}\|_{L^\infty}^{\frac{3}{2}} \|\rho\|_{H^3} \|v\|_{H_\rho^3} \|\partial^m v\|_{L^2_\rho}^2 \,,
	\end{aligned}
\end{equation}
and
\begin{equation}\label{v-eq122}
	\begin{aligned}
V_{122}&\lesssim \sum_{\substack{m^\prime \leq m \\ |m'| = 1}} \|\rho \partial^{m^\prime} v\|_{L^\infty} \|\nabla \partial^{m - m^\prime} v\|_{L^2} \|\partial^m v\|_{L^2} \\&
\lesssim \|\frac{1}{\rho}\|_{L^\infty}^{\frac{3}{2}} \|\rho\|_{H^2} \|v\|_{H^3_\rho} \|v\|_{H^s_\rho} \|\partial^m v\|_{L^2_\rho}^2 \,,
	\end{aligned}
\end{equation}
and
\begin{equation}\label{v-eq123}
	\begin{aligned}
V_{123}&\lesssim \sum_{\substack{m^\prime \leq m \\ 2 \leq |m^\prime| \leq s-1}} \|\rho\|_{L^\infty} \|\partial^{m^\prime} v\|_{L^4} \|\nabla \partial^{m - m^\prime} v\|_{L^4} \|\partial^m v\|_{L^2}\\&
\lesssim \|\rho\|_{H^2} \|v\|_{H^s}^2 \|\partial^m v\|_{L^2} \lesssim \|\frac{1}{\rho}\|_{L^\infty}^{\frac{3}{2}} \|\rho\|_{H^2} \|v\|_{H^s_\rho}^2 \|\partial^m v\|_{L^2_\rho} \,,
	\end{aligned}
\end{equation}
and
\begin{equation}\label{v-eq124}
	\begin{aligned}
V_{124} \lesssim \|\rho \nabla v\|_{L^\infty} \|\partial^m v\|_{L^2}^2 \lesssim \|\rho\|_{H^2} \|v\|_{H^3} \|\partial^m v\|_{L^2}^2 \lesssim \|\frac{1}{\rho}\|_{L^\infty}^{\frac{3}{2}} \|\rho\|_{H^2} \|v\|_{H^3_\rho} \|\partial^m v\|_{L^2_\rho}^2 \,.
	\end{aligned}
\end{equation}
Consequently, $V_{12}$ satisfies the bound
\begin{equation}\label{v-eq120}
	\begin{aligned}
-V_{12} \lesssim & \|\frac{1}{\rho}\|^{\frac{3}{2}}_{L^\infty}\|v\|_{H^3_\rho}\|\rho\|_{H^3}\|\p^mv\|^2_{L^2_\rho}+ \|\frac{1}{\rho}\|^{\frac{3}{2}}_{L^\infty}\|v\|_{H^s_\rho}\|\rho\|_{H^2}\|\p^mv\|_{L^2_\rho}\\
& + \|\frac{1}{\rho}\|^{\frac{3}{2}}_{L^\infty}\|v\|^2_{H^s_\rho}\|\rho\|_{H^2}\|\p^mv\|_{L^2_\rho}+ \|\frac{1}{\rho}\|^{\frac{3}{2}}_{L^\infty}\|v\|_{H^3_\rho}\|\rho\|_{H^2}\|\p^mv\|^2_{L^2_\rho}
\,.
	\end{aligned}
\end{equation}
Similarly, using H\"older's inequality and Sobolev embedding, the term $V_{13}$ can be estimated as
\begin{equation}\label{v-eq13}
	\begin{aligned}
-V_{13}= & \sum_{\substack{ 0\neq m^{\prime}\leq m}}\mathbf{1}_{|m|\geq 1} C^{m^{\prime}}_{m} \langle \p^{m^{\prime}} \rho\partial^{m} v  ,\partial^{m} v \rangle\\
\lesssim & \sum_{\substack{ m^{\prime}\leq m,|m^{\prime}|=1}}\|\p^{m^{\prime}}\rho\|_{L^{\infty}}\|\p_t\p^{m-m^{\prime}}v\|_{L^2}\|\p^mv\|_{L^2} \\
& +\sum_{\substack{ m^{\prime}\leq m,2\leq|m^{\prime}|\leq s-1}}\|\p^{m^{\prime}}\rho\|_{L^4}\|\p_t\p^{m-m^{\prime}}v\|_{L^4}\|\p^mv\|_{L^2}\\
& +\|\p^m\rho\|_{L^2}\|\p_t v\|_{L^\infty}\|\p^m v\|_{L^2}\\
\lesssim & \| \rho \|_{H^3} \| \partial_t v \|_{H^{s-1}} \| \partial^m v \|_{L^2} + \| \rho\|_{H^s} \| \partial_t v \|_{H^{s-1}} \| \partial^m v \|_{L^2} + \| \rho \|_{H^s} \| \partial_t v \|_{H^2} \| \partial^m v \|_{L^2} \\
\lesssim & \| \frac{1}{\rho} \|_{L^\infty}^{\frac{1}{2}} \| \partial_t v \|_{H^{s-1}} \left( \| \rho \|_{H^3} + \| \rho\|_{H^s} \right) \| \partial^m v \|_{L^2_\rho} \,.
	\end{aligned}
\end{equation}
We now turn to $V_{14}$. It can be decomposed as follows:
	\begin{align}\label{v-eq14}
\no -V_{14} = & -\sum_{\substack{0\neq m'\leq m \\ |m|\geq 1}} \mathbf{1}_{|m|\geq 1}C_m^{m^{\prime}}\langle \partial^{m'} \rho \partial^{m - m'} (v \cdot \nabla v), \partial^m v \rangle \\
= & \underset{V_{141}}{\underbrace{ -\sum_{\substack{m'\leq m \\ |m'|=1 }}\mathbf{1}_{|m|\geq 1}C_m^{m^{\prime}}\langle \partial^{m'} \rho \partial^{m - m'} (v \cdot \nabla v), \partial^m v \rangle}} \\
\no & \underset{V_{142}}{\underbrace{-
\sum_{\substack{m'\leq m \\ 2\leq |m'|\leq s-1 }}\mathbf{1}_{|m|\geq 1}C_m^{m^{\prime}}\langle \partial^{m'} \rho \partial^{m - m^{\prime}} (v \cdot \nabla v), \partial^m v \rangle}} \ \underset{ V_{143} }{ \underbrace{ - \mathbf{1}_{|m|\geq 1} \langle \partial^m \rho (v \cdot \nabla v), \partial^m v \rangle}} \,.
	\end{align}
By H\"older's inequality and Sobolev embedding, we obtain
\begin{equation}\label{v-eq141}
	\begin{aligned}
    V_{141}&\lesssim \sum_{\substack{m^\prime \leq m \\ |m^\prime| = 1}} \|\partial^{m^\prime} \rho\|_{L^\infty} \|\partial^{m - m^\prime} (v \cdot \nabla v)\|_{L^2} \|\partial^m v\|_{L^2} \lesssim \|\rho\|_{H^3} \|v \cdot \nabla v\|_{H^{s-1}} \|\partial^m v\|_{L^2} \\
    & \lesssim \|\rho\|_{H^3} \|v\|_{H^s}^2 \|\partial^m v\|_{L^2} \lesssim \|\frac{1}{\rho}\|_{L^\infty}^{\frac{3}{2}} \|\rho\|_{H^3} \|v\|_{H^s_\rho}^2 \|\partial^m v\|_{L_\rho^2} \,,
	\end{aligned}
\end{equation}
and
\begin{equation}\label{v-eq142}
	\begin{aligned}
    V_{142} & \lesssim \sum_{\substack{m^\prime \leq m \\ 2 \leq |m^\prime| \leq s-1}} \|\partial^{m^\prime} \rho\|_{L^4} \|\partial^{m - m^\prime} (v \cdot \nabla v)\|_{L^4} \|\partial^m v\|_{L^2} \lesssim \|\rho\|_{H^s} \|v \cdot \nabla v\|_{H^{s-1}} \|\partial^m v\|_{L^2} \\
    & \lesssim \|\rho\|_{H^s} \|v\|_{H^s}^2 \|\partial^m v\|_{L^2} \lesssim \|\frac{1}{\rho}\|_{L^\infty}^{\frac{3}{2}} \|\rho\|_{H^s} \|v\|_{H^s_\rho}^2 \|\partial^m v\|_{L_\rho^2} \,,
	\end{aligned}
\end{equation}
and
\begin{equation}\label{v-eq143}
	\begin{aligned}
    V_{143} \lesssim \|\partial^m \rho\|_{L^2} \|v \cdot \nabla v\|_{L^\infty} \|\partial^m v\|_{L^2} \lesssim \|\rho\|_{H^s} \|v\|_{H^3}^2 \|\partial^m v\|_{L^2} \lesssim \|\frac{1}{\rho}\|_{L^\infty}^{\frac{3}{2}} \|\rho\|_{H^s} \|v\|_{H^3_\rho}^2 \|\partial^m v\|_{L_\rho^2} \,.
	\end{aligned}
\end{equation}
Thus $V_{14}$ admits the bound
\begin{equation}\label{v-eq140}
	\begin{aligned}
 -V_{14}& = \sum_{\substack{0\neq m^{\prime}\leq m \\ |m|\geq 1}} \mathbf{1}_{|m|\geq 1}C_m^{m^{\prime}} \langle \partial^{m'} \rho\partial^{m - m'} (v \cdot \nabla v), \partial^m v \rangle \\&
\lesssim \|\frac{1}{\rho}\|_{L^\infty}^{\frac{3}{2}} \| \partial^m v \|_{L^p}^2 \left( \| \rho\|_{H^{3}} \| v \|^2_{H^s_\rho} + \| \rho\|_{H^s} \| v \|_{H^s_\rho}^2 + \| \rho\|_{H^s} \| v \|_{H^3_\rho}^2 \right) \,.
	\end{aligned}
\end{equation}
Under the assumption $s \geq 3$, collecting the estimates \eqref{v-eq11}, \eqref{v-eq120}, \eqref{v-eq13} and \eqref{v-eq140} yields the following lower bound for $V_1$:
\begin{equation}\label{v1-eq}
	\begin{aligned}
V_1 = & \langle \partial^m \left[ \rho \left( \partial_t v + v \cdot \nabla v \right) \right], \partial^m v \rangle \\
\geq & \frac{1}{2} \frac{d}{dt} \| \partial^m v \|_{L^2_\rho}^2 - C\| \frac{1}{\rho} \|_{L^\infty}^{\frac{3}{2}} \| \rho \|_{H^s} \| v \|^2_{H^s_\rho}\| \partial^m v \|_{L_\rho}^2 \\
& - C \| \frac{1}{\rho} \|_{L^\infty} \| \rho \|_{H^s} \| \partial_t v \|_{H^{s-1}} \| \partial^m v \|_{L^2_\rho}
	\end{aligned}
\end{equation}
for some positive constant $C>0$.

Next we estimate the quantity $V_2$. It is straightforward to see that
	\begin{align}\label{v2-eq0}
\no -V_2 & = - \langle \nabla \partial^m (P(\rho) - A |\nabla M|^2 + \mu_0 M \cdot H_{ext}), \partial^m v \rangle \\
\no & = \langle \partial^m (P(\rho) - A |\nabla M|^2 + \mu_0 M \cdot H_{ext}), \nabla \partial^m v \rangle \\
& \lesssim \|\partial^m (P(\rho) - A |\nabla M|^2 + \mu_0 M \cdot H_{ext})\|_{L^2} \|\nabla \partial^m v\|_{L^2} \\
\no & \lesssim \left( \|\partial^m P(\rho)\|_{L^2} + \|\partial^m (|\nabla M|^2)\|_{L^2} + \|\partial^m (M \cdot H_{ext})\|_{L^2} \right) \|\nabla \partial^m v\|_{L^2}\,.
	\end{align}
Recalling $P(\rho)=a\rho^\gamma $ with $ \gamma\geq1$, we have
$$\|\partial^m P(\rho)\|_{L^2} \lesssim \|\rho\|_{H^s}^\gamma$$
for $|m| \leq s$. Moreover, if $s \geq 3$,
$$\|\partial^m |\nabla M|^2\|_{L^2} \lesssim \|\nabla M\|_{H^s}^2 \,,$$
and
\begin{equation*}
	\begin{aligned}
\|\partial^m (M \cdot H_{ext})\|_{L^2} &\lesssim \|M \cdot \partial^m H_{ext}\|_{L^2} + \sum_{0 \neq m^{\prime} \leq m} \|\partial^{m^{\prime} } M \cdot \partial^{m - m'} H_{ext}\|_{L^2} \\
& \lesssim \|\partial^m H_{ext}\|_{L^2} + \|\nabla M\|_{H^{s-1}} \|H_{ext}\|_{H^{s - 1}}\,,
	\end{aligned}
\end{equation*}
where we have used the fact that $|M| = 1$. Consequently, we obtain
\begin{equation}\label{v2-eq}
	\begin{aligned}
-V_2&-\langle \nabla \partial^m (P(\rho) - A |\nabla M|^2 + \mu_0 M \cdot H_{ext}), \partial^m v \rangle \\&
\lesssim \left( \|P\|_{H^s}^2 + \|\nabla M\|_{H^s}^2 + (1+\|\nabla M\|_{H^s} )\|H_{ext}\|_{H^{s}} \right) \|\nabla \partial^m v\|_{L^2}.
	\end{aligned}
\end{equation}

For the quantity $V_3$, integration by parts together with H\"older's inequality and Sobolev embedding yields, for $s \geq 3$,
\begin{equation}\label{v3-eq}
	\begin{aligned}
V_3 & =\langle \partial^m (\rho F: \nabla F^\top), \partial^m v \rangle = \langle \partial^m (\rho F^{jk} \partial_j F^{ik}), \partial^m v^i \rangle \\
& = \langle \nabla \cdot \partial^m (\rho F F^\top), \partial^m v \rangle = \langle \partial_j \partial^m (\rho F^{ik} F^{jk}), \partial^m v^i \rangle \\
& = -\langle \partial^m (\rho F^{ik} F^{jk}), \partial^m \partial_j v^i \rangle \lesssim\|\p^m(\rho FF^\top)\|_{L^2}\|\nabla\p^mv\|_{L^2} \\
& \lesssim\|\rho\|_{H^s}\|F\|^2_{H^s}\|\nabla\p^mv\|_{L^2}.
	\end{aligned}
\end{equation}
Similarly, the quantity $V_4$ can be bounded by
\begin{equation}\label{v4-eq}
	\begin{aligned}
V_4 & =- 2A\langle \nabla \cdot \partial^m (\nabla M \odot \nabla M), \partial^m v \rangle = 2A\langle \partial^m (\nabla M \odot  \nabla M), \nabla \partial^m v \rangle \\
& \lesssim \|\partial^m (\nabla M \odot  \nabla M)\|_{L^2} \|\nabla \partial^m v\|_{L^2} \lesssim \|\nabla M\|_{H^s}^2 \|\nabla \partial^m v\|_{L^2}\,.
	\end{aligned}
\end{equation}
Furthermore, the quantity $V_5$ can be estimated analogously:
\begin{equation}\label{v5-eq}
	\begin{aligned}
V_5& =\mu_0 \langle \partial^m [(\nabla H_{ext})^\top M], \partial^m v \rangle 
= \mu_0 \langle \partial^m (\partial_i H_{ext}^j M_j), \partial^m v^i \rangle \\&
= -\mu_0 \langle \partial^m (H_{ext}^i \partial_i M_j), \partial^m v^i \rangle
-\mu_0 \langle \partial^m (H_{ext}^i M_j), \partial^m \partial_i v^i \rangle \\&
\lesssim \|\partial^m (\nabla M \cdot H_{ext})\|_{L^2} \|\partial^m v\|_{L^2} + \|\partial^m (M \cdot H_{ext})\|_{L^2} \|\nabla \partial^m v\|_{L^2} \\&
\lesssim \|\nabla M\|_{H^s} \|H_{ext}\|_{H^s} \|\partial^m v\|_{L^2}  + (1 + \|\nabla M\|_{H^s}) \|H_{ext}\|_{H^s} \|\nabla \partial^m v\|_{L^2}\,.
	\end{aligned}
\end{equation}

Substituting the estimates \eqref{v1-eq}, \eqref{v2-eq}, \eqref{v3-eq}, \eqref{v4-eq} and \eqref{v5-eq} into \eqref{v-eq0}, we obtain for $s \geq 3$
\begin{equation}\label{v-equation-1}
	\begin{aligned}
&\frac{1}{2} \frac{\d}{\d t} \|\partial^m v\|_{L^2_\rho}^2 + \mu \|\nabla \partial^m v\|_{L^2}^2 + (\mu + \xi) \|\nabla \cdot \partial^m v\|_{L^2}^2 \\&
\lesssim \| \frac{1}{\rho} \|_{L^\infty}^{\frac{3}{2}} \|\rho\|_{H^s} \|v\|_{H^s_\rho}^2 \|\partial^m v\|_{L^2_\rho} + \| \frac{1}{\rho} \|_{L^\infty}^{\frac{1}{2}} \|\rho\|_{H^s} \|\partial_t v\|_{H^{s - 1}} \|\partial^m v\|_{L^2_\rho} \\&
+[ \|\rho\|_{H^s}^\gamma + \|\nabla M\|_{H^s}^2 + (1 + \|\nabla M\|_{H^s}) \|H_{ext}\|_{H^s} ] \|\nabla \partial^m v\|_{L^2}\\&
+ \left[ \|\rho\|_{H^s} \|F\|_{H^s}^2 + \|\nabla M\|_{H^s}^2 + (1 + \|\nabla M\|_{H^s}) \|H_{ext}\|_{H^s} \right] \|\nabla \partial^m v\|_{L^2}\,.
	\end{aligned}
\end{equation}
Together with the definitions of $\mathcal{E}_s(t)$ and $\mathcal{D}_s(t)$ in \eqref{EDs-loc}, inequality \eqref{v-equation-1} implies the desired estimate \eqref{v-equation}. This completes the proof of the lemma.
\end{proof}

\subsubsection{\bf Estimates for $\partial_t v$}

The goal of this subsection is to estimate the norm $\|\partial_t v\|_{H^{s-1}}$ appearing in \eqref{v-equation}. The idea is to bound the time derivative $\partial_t v$ in terms of the spatial derivatives of $(\rho, v, F, M)$ by exploiting the structure of the $v$-equation in \eqref{CMEH-eq}. More precisely, we have the following lemma.

\begin{lemma}\label{Lmm-dtv-loc}
    For $\gamma > 1$ and integer $s \geq 3$, we have
    \begin{equation}\label{dtv-equation}
	\begin{aligned}
\|\partial_t v\|_{H^{s-1}} \lesssim (1 + \| H_{ext} \|_{H^s}) ( \|\frac{1}{\rho}\|^\frac{1}{2}_{L^\infty} + \|\frac{1}{\rho}\|^\frac{s}{2}_{L^\infty} ) \big(1 + \mathcal{E}_s^\frac{s + \gamma}{ 2 } (t) + \mathcal{E}_s^\frac{s + 2}{ 2 } (t) + \mathcal{D}_s^\frac{1}{2} (t) \big) \mathcal{E}_s^\frac{1}{2} (t)  \,,
	\end{aligned}
\end{equation}
where $\mathcal{E}_s (t)$ and $\mathcal{D}_s (t)$ are defined in \eqref{EDs-loc}.
\end{lemma}

\begin{proof}

The $v$-equation in \eqref{CMEH-eq} can be rewritten as
\begin{equation}\label{dtv-eq1}
	\begin{aligned}
&\p_t v +\;  v \cdot \nabla v + \frac{1}{\rho}{\nabla \left( P(\rho) - A |\nabla M|^2 + \mu_0 M \cdot H_{ext} \right)} 
			\\ & 
			= \frac{\mu}{\rho} \Delta v+\frac{\mu+\xi}{\rho}\nabla(\nabla\cdot v)+) \frac{1}{\rho}\nabla\cdot(\rho F F^\top)- \frac{2A}{\rho} \nabla \cdot ( \nabla M \odot \nabla M) + \frac{\mu_0}{\rho} (\nabla H_{ext})^\top M \,.
	\end{aligned}
\end{equation}
We then obtain
\begin{equation}\label{dtv-eq2}
	\begin{aligned}
\|\partial_t v\|_{H^{s-1}}& \lesssim \|v \cdot \nabla v\|_{H^{s-1}} + \|\frac{1}{\rho}\nabla\left(P(\rho) - A|\nabla M|^2 + \mu_0 M \cdot H_{ext}\right)\|_{H^{s-1}} \\&\quad
+ \|\frac{1}{\rho}\Delta v\|_{H^{s-1}} +\|\frac{1}{\rho}\nabla \cdot v\|_{H^{s-1}} + \|\frac{1}{\rho}\nabla \cdot (\rho F F^\top)\|_{H^{s-1}} \\&
\quad+ \|\frac{1}{\rho}\nabla \cdot (\nabla M \odot \nabla M) \|_{H^{s-1}} + \|\frac{1}{\rho}(\nabla H_{ext})^\top M\|_{H^{s-1}}\,.
	\end{aligned}
\end{equation}
By applying Lemma 3.2 of \cite{JLT-M3AS-2020} or Lemma 3.1 of \cite{JLZ-CMS-2021}, we obtain for $s \geq 3$,
\begin{equation}\label{dtv-eq3}
	\begin{aligned}
\|\frac{1}{\rho} \cdot f\|_{H^{s-1}}
&\lesssim ( \|\frac{1}{\rho} \|_{L^\infty}^{\frac{1}{2}} +\|\frac{1}{\rho}\|_{L^\infty}^{\frac{s-1}{2}} ) \|\nabla \rho\|_{H^{s-2}} (1 + \|\nabla \rho\|^{s-2}_{H^{s-2}}) \|f\|_{H^{s-1}} \\&
\lesssim\underset{K_s(\rho)}{\underbrace{ ( \|\frac{1}{\rho}\|_{L^\infty}^{\frac{1}{2}} + \|\frac{1}{\rho}\|_{L^\infty}^{\frac{s-1}{2}} ) \|\rho\|_{H^s}(1 + \|\rho\|_{H^s}^{s-2}) }}\|f\|_{H^{s-1}} \,.
	\end{aligned}
\end{equation}
Using \eqref{dtv-eq3}, we have
\begin{equation}\label{dtv-eq4}
	\begin{aligned}
&\|\frac{1}{\rho} \nabla \left(P(\rho) - A|\nabla M|^2 + \mu_0 M \cdot H_{ext}\right)\|_{H^{s-1}} \\&
\lesssim K_s(\rho) \left( \|\nabla P(\rho)\|_{H^{s-1}} + \|\nabla |\nabla M|^2\|_{H^{s-1}} + \|M \cdot H_{ext}\|_{H^{s-1}} \right) \\&
\lesssim K_s(\rho)( \|\rho\|_{H^s}^\gamma + \|\nabla M\|_{H^s}^2 + (1 + \|\nabla M\|_{H^s}) \|H_{ext}\|_{H^s} )\,.
	\end{aligned}
\end{equation}
Similarly, 
\begin{equation}\label{dtv-eq5}
	\begin{aligned}
\|\frac{1}{\rho}\Delta v\|_{H^{s-1}} \lesssim K_s(\rho) \|\Delta v\|_{H^{s-1}} \lesssim K_s(\rho) \|\nabla v\|_{H^s}\,,
	\end{aligned}
\end{equation}
and
\begin{equation}\label{dtv-eq6}
	\begin{aligned}
\|\frac{1}{\rho}\nabla \cdot v\|_{H^{s-1}} \lesssim K_s(\rho) \|\nabla \cdot v\|_{H^{s-1}} \lesssim K_s(\rho) \|v\|_{H^s}
 \lesssim K_s(\rho) \|\frac{1}{\rho}\|_{L^\infty}^{\frac{1}{2}} \|v\|_{H^s_\rho}^2\,,
	\end{aligned}
\end{equation}
and
\begin{equation}\label{dtv-eq7}
\begin{aligned}
\|\frac{1}{\rho}\nabla \cdot (\rho F F^\top)\|_{H^{s-1}} \lesssim K_s(\rho) \|\nabla \cdot (\rho F F^\top)\|_{H^{s-1}}  \lesssim K_s(\rho) \|\rho\|_{H^s} \|F\|_{H^s}^2\,,
	\end{aligned}
\end{equation}
and
\begin{equation}\label{dtv-eq8}
\begin{aligned}
\|\frac{1}{\rho}\nabla \cdot (\nabla M \odot \nabla M)\|_{H^{s-1}} \lesssim K_s(\rho) \|\nabla \cdot (\nabla M \odot \nabla M)\|_{H^{s-1}} \lesssim K_s(\rho) \|\nabla M\|_{H^s}^2\,,
	\end{aligned}
\end{equation}
and
\begin{equation}\label{dtv-eq9}
\begin{aligned}
\|\frac{1}{\rho}(\nabla H_{ext})^\top M\|_{H^{s-1}} \lesssim K_s(\rho) \|(\nabla H_{ext})^\top M\|_{H^{s-1}}\lesssim K_s(\rho) (1 + \|\nabla M\|_{H^s}) \|H_{ext}\|_{H^s}\,.
	\end{aligned}
\end{equation}
Substituting the estimates \eqref{dtv-eq4}, \eqref{dtv-eq5}, \eqref{dtv-eq6}, \eqref{dtv-eq7}, \eqref{dtv-eq8}, and \eqref{dtv-eq9} into \eqref{dtv-eq2} yields, for $s \geq 3$,
\begin{equation}\label{dtv-equation-1}
	\begin{aligned}
\|\partial_t v\|_{H^{s-1}}\lesssim & \|\frac{1}{\rho}\|_{L^\infty}\|v\|^2_{H^s_\rho}+K_s(\rho)\|\nabla v\|_{H^s} \\
& +K_s(\rho)(\|\rho\|^\gamma_{H_s}+\|\nabla M\|^2_{H^s}+(1+\|\nabla M\|_{H^s} ) \|H_{ext}\|_{H^s})\\
& + K_s(\rho) (\|\frac{1}{\rho}\|^{\frac{1}{2}}_{L^\infty} \|v\|_{H^s_\rho} + \|\rho\|_{H^s} \|F\|^2_{H^s} )  \,.
	\end{aligned}
\end{equation}
Note that 
\begin{equation*}
    \begin{aligned}
        K_s (\rho) \lesssim ( \| \frac{1}{\rho} \|_{L^\infty}^\frac{1}{2} + \| \frac{1}{\rho} \|_{L^\infty}^\frac{s-1}{2} ) ( 1 + \mathcal{E}_s^\frac{s-1}{2} (t) ) \mathcal{E}_s^\frac{1}{2} (t) \,.
    \end{aligned}
\end{equation*}
Inequality \eqref{dtv-equation-1} implies the desired bound \eqref{dtv-equation}. This completes the proof of the lemma.
\end{proof}

\subsubsection{\bf Estimates for the $M$-equation}

The goal of this subsection is to establish energy estimates for the magnetic field $M$. We have the following lemma.

\begin{lemma}\label{Lmm-M-loc}
    For any multi-index $m \in \mathbb{N}^3$ with $|m| \leq s$ and $s \geq 3$, we have
\begin{equation}\label{M-equation}
	\begin{aligned}
& \frac{1}{2}\frac{\d}{\d t}\|\nabla\p^m M\|^2_{L^2}+2A\lambda\|\Delta\p^m M\|^2_{L^2}\\
\lesssim & \|\frac{1}{\rho}\|^{\frac{1}{2}}_{L^\infty} \mathcal{E}_s (t) \mathcal{D}_s^\frac{1}{2} (t) + \left[ ( 1 + \mathcal{E}_s^\frac{1}{2} (t) ) \mathcal{E}_s (t) + ( 1 + \mathcal{E}_s (t) ) \|H_{ext}\|_{H^s} \right] \mathcal{D}_s^\frac{1}{2} (t) \,.
	\end{aligned}
\end{equation}
\end{lemma}

\begin{proof}

We apply the $m$-th order derivative operator $\partial^m$ to the $M$-equation in \eqref{CMEH-eq} for all $|m| \leq s$, take the $L^2$ inner product with $\Delta \partial^{m} M$, and integrate by parts over $x \in \mathbb{T}^3$. This yields
\begin{equation}\label{M-eq1}
	\begin{aligned}
& \langle \partial_t \partial^m M, \Delta \partial^m M \rangle + \langle \partial^m (v \cdot \nabla M), \Delta \partial^m M \rangle  \\
= & 2A\lambda \|\Delta \partial^m M\|_{L^2}^2 + \lambda \mu_0 \langle \partial^m H_{ext}, \Delta \partial^m M \rangle + \langle \partial^m (\Gamma(M) M), \Delta \partial^m M \rangle  \\
& - 2 \gamma\langle \partial^m \left[ M \times (2A\Delta M + \mu_0 H_{ext}) \right], \Delta \partial^m M \rangle  
\,.
	\end{aligned}
\end{equation}
Observe that
\begin{equation}\label{M-eq2}
	\begin{aligned}
&\langle \partial_t \partial^m M, \Delta \partial^m M \rangle = -\frac{1}{2} \frac{\d}{\d t} \|\nabla \partial^m M\|_{L^2}^2 
\,.
	\end{aligned}
\end{equation}
Moreover, H\"older's inequality and Sobolev embedding give
\begin{equation}\label{M-eq3}
	\begin{aligned}
\langle \partial^m (v \cdot \nabla M), \Delta \partial^m M \rangle &\lesssim \|\partial^m (v \cdot \nabla M)\|_{L^2} \|\Delta \partial^m M\|_{L^2} \lesssim \|v\|_{H^s} \|\nabla M\|_{H^s} \|\Delta \partial^m M\|_{L^2}\\
& \lesssim \|\rho\|_{L^\infty}^\frac{1}{2} \|v\|_{H^s_\rho} \|\nabla M\|_{H^s} \|\Delta \partial^m M\|_{L^2} \,,
	\end{aligned}
\end{equation}
and
\begin{equation}\label{M-eq4}
	\begin{aligned}
-\lambda\mu_0 \langle \partial^m H_{ext}, \Delta \partial^m M \rangle \lesssim \|\partial^m H_{ext}\|_{L^2} \|\Delta \partial^m M\|_{L^2} \,,
	\end{aligned}
\end{equation}
and
\begin{equation}\label{M-eq5}
	\begin{aligned}
&-\langle \partial^m (\Gamma(M) M), \Delta \partial^m M \rangle \\& = -\langle \partial^m \left[ (2A\lambda|\nabla M|^2 + \lambda\mu_0 H_{ext} \cdot M) M \right], \Delta \partial^m M \rangle \\
\lesssim & \left( \|\partial^m (|\nabla M|^2 M)\|_{L^2} + \|\partial^m ((M \cdot H_{ext}) M)\|_{L^2} \right) \|\Delta \partial^m M\|_{L^2} \\
\lesssim & \left[ (1 + \|\nabla M\|_{H^s}) \|\nabla M\|_{H^s}^2 + (1 + \|\nabla M\|_{H^s}^2) \|H_{ext}\|_{H^s} \right] \|\Delta \partial^m M\|_{L^2} \,,
	\end{aligned}
\end{equation}
and
\begin{equation}\label{M-eq6}
	\begin{aligned}
&\gamma \langle \partial^m \left[ M \times (2A\Delta M + \mu_0 H_{ext}) \right], \Delta \partial^m M \rangle  \\
= & 2\gamma A\sum_{\substack{0 \neq m^{\prime}\leq m}} C_m^{m^{\prime}} \langle \partial^{m^{\prime}} M \times \Delta \partial^{m-m^{\prime}} M, \Delta \partial^m M \rangle + \gamma\mu_0 \langle \partial^m (M \times H_{ext}), \Delta \partial^m M \rangle \\
\lesssim & \sum_{\substack{0 \neq  m^{\prime}\leq m}} \|\partial^{ m^{\prime}} M \times \Delta \partial^{m- m^{\prime}} M\|_{L^2} \|\Delta \partial^m M\|_{L^2} +\|\partial^m (M \times H_{ext})\|_{L^2} \|\Delta \partial^m M\|_{L^2} \\
\lesssim & \|\nabla M\|_{H^s}^2 \|\Delta \partial^m M\|_{L^2}
+ (1 + \|\nabla M\|_{H^s}) \|H_{ext}\|_{H^s} \|\Delta \partial^m M\|_{L^2}\\
\lesssim & \left[\|\nabla M\|_{H^s}^2+(1 + \|\nabla M\|_{H^s}) \|H_{ext}\|_{H^s}]\right \|\Delta \partial^m M\|_{L^2} \,.
	\end{aligned}
\end{equation}
Substituting the relations \eqref{M-eq2}, \eqref{M-eq3}, \eqref{M-eq4}, \eqref{M-eq5}, and \eqref{M-eq6} into \eqref{M-eq1} yields
\begin{equation}\label{M-equation-1}
	\begin{aligned}
&\frac{1}{2}\frac{\d}{\d t}\|\nabla\p^m M\|^2_{L^2}+2A\lambda\|\Delta\p^m M\|^2_{L^2}\\
\lesssim & \|\frac{1}{\rho}\|^{\frac{1}{2}}_{L^\infty}\|v\|_{H^s_\rho}\|\nabla M\|_{H^s}\|\Delta\p^m M\|_{L^2}\\
& +\left[(1+\|\nabla M\|_{H^s})\|\nabla M\|_{H^s}^2 + (1+\|\nabla M\|_{H^s}^2)\|H_{ext}\|_{H^s}\right]\|\Delta\partial^mM\|_{L^2} \,.
	\end{aligned}
\end{equation}
Recalling the definitions of $\mathcal{E}_s(t)$ and $\mathcal{D}_s(t)$ in \eqref{EDs-loc}, the estimate \eqref{M-equation-1} implies the energy bound \eqref{M-equation}. This completes the proof of the lemma.
\end{proof}

\subsubsection{\bf Closing the estimates: Proof of Proposition \ref{Prop-Loc}}

Using Lemmas \ref{Lmm-rho-loc}, \ref{Lmm-F-loc}, \ref{Lmm-v-loc}, \ref{Lmm-dtv-loc} and \ref{Lmm-M-loc}, we now close the local a priori estimate \eqref{Ener-Bnds-Local}, thereby proving Proposition \ref{Prop-Loc}. From \eqref{rho-eq}, \eqref{F-eq}, \eqref{v-equation} and \eqref{M-equation}, we have
\begin{equation}\label{energy-eq1}
	\begin{aligned}
& \frac{1}{2}\frac{\d}{\d t}\left(\|\partial^m \rho\|_{L^2_\rho}^2 + \|\partial^m F\|_{L^2}^2 + \|\partial^m v\|_{L^2_\rho}^2 + \|\nabla \partial^m M\|_{L^2}^2\right) \\
& \quad + \mu \|\nabla \partial^m v\|_{L^2}^2 + (\mu + \xi) \|\nabla \cdot \partial^m v\|_{L^2}^2 + 2A\lambda \|\Delta \partial^m M\|_{L^2}^2\\
\lesssim & \|\tfrac{1}{\rho}\|^{\tfrac{1}{2}}_{L^\infty} \mathcal{E}_s^\frac{3}{2} (t) + \mathcal{E}_s (t) \mathcal{D}_s^\frac{1}{2} (t) + \|\tfrac{1}{\rho}\|^{\tfrac{1}{2}}_{L^\infty} \mathcal{E}_s^\frac{3}{2} (t) + \mathcal{E}_s (t) \mathcal{D}_s^\frac{1}{2} (t) \\
& + \| \frac{1}{\rho} \|_{L^\infty}^{\frac{3}{2}} \mathcal{E}_s^2 (t) + \| \frac{1}{\rho} \|_{L^\infty}^{\frac{1}{2}} \|\partial_t v\|_{H^{s - 1}} \mathcal{E}_s (t) \\
& + \big[ \mathcal{E}_s^\frac{\gamma}{2} (t) + \mathcal{E}_s (t) + \mathcal{E}_s^\frac{3}{2} (t) + ( 1 + \mathcal{E}_s^\frac{1}{2} (t) ) \|H_{ext}\|_{H^s} \big] \mathcal{D}_s^\frac{1}{2} (t) \\
& + \|\frac{1}{\rho}\|^{\frac{1}{2}}_{L^\infty} \mathcal{E}_s (t) \mathcal{D}_s^\frac{1}{2} (t) + \left[ ( 1 + \mathcal{E}_s^\frac{1}{2} (t) ) \mathcal{E}_s (t) + ( 1 + \mathcal{E}_s (t) ) \|H_{ext}\|_{H^s} \right] \mathcal{D}_s^\frac{1}{2} (t) \,.
	\end{aligned}
\end{equation}

Substituting \eqref{dtv-equation} into \eqref{energy-eq1} yields
\begin{equation}\label{energy-eq2}
	\begin{aligned}
	& \frac{1}{2}\frac{\d}{\d t}\left(\|\partial^m \rho\|_{L^2_\rho}^2 + \|\partial^m F\|_{L^2}^2 + \|\partial^m v\|_{L^2_\rho}^2 + \|\nabla \partial^m M\|_{L^2}^2\right) \\
& \quad + \mu \|\nabla \partial^m v\|_{L^2}^2 + (\mu + \xi) \|\nabla \cdot \partial^m v\|_{L^2}^2 + 2A\lambda \|\Delta \partial^m M\|_{L^2}^2\\
\lesssim & \|\tfrac{1}{\rho}\|^{\tfrac{1}{2}}_{L^\infty} \mathcal{E}_s^\frac{3}{2} (t) + \mathcal{E}_s (t) \mathcal{D}_s^\frac{1}{2} (t) + \|\tfrac{1}{\rho}\|^{\tfrac{1}{2}}_{L^\infty} \mathcal{E}_s^\frac{3}{2} (t) + \mathcal{E}_s (t) \mathcal{D}_s^\frac{1}{2} (t) + \| \frac{1}{\rho} \|_{L^\infty}^{\frac{3}{2}} \mathcal{E}_s^2 (t) \\
& + (1 + \| H_{ext} \|_{H^s}) ( \|\frac{1}{\rho}\|_{L^\infty} + \|\frac{1}{\rho}\|^\frac{s+1}{2}_{L^\infty} ) \big(1 + \mathcal{E}_s^\frac{s + \gamma}{ 2 } (t) + \mathcal{E}_s^\frac{s + 2}{ 2 } (t) + \mathcal{D}_s^\frac{1}{2} (t) \big) \mathcal{E}_s^\frac{3}{2} (t) \\
& + \big[ \mathcal{E}_s^\frac{\gamma}{2} (t) + \mathcal{E}_s (t) + \mathcal{E}_s^\frac{3}{2} (t) + ( 1 + \mathcal{E}_s^\frac{1}{2} (t) ) \|H_{ext}\|_{H^s} \big] \mathcal{D}_s^\frac{1}{2} (t) \\
& + \|\frac{1}{\rho}\|^{\frac{1}{2}}_{L^\infty} \mathcal{E}_s (t) \mathcal{D}_s^\frac{1}{2} (t) + \left[ ( 1 + \mathcal{E}_s^\frac{1}{2} (t) ) \mathcal{E}_s (t) + ( 1 + \mathcal{E}_s (t) ) \|H_{ext}\|_{H^s} \right] \mathcal{D}_s^\frac{1}{2} (t) \,.
	\end{aligned}
\end{equation}
Recall the definitions of the energy functional $\mathcal{E}_s(t)$ and the energy dissipation rate functional $\mathcal{D}_s(t)$ in \eqref{EDs-loc}. Summing over $|m|\leq s$ in \eqref{energy-eq2} therefore implies that
\begin{equation}\label{Ener-Bnds-Local2}
	\begin{aligned}
\frac{1}{2}\frac{\d}{\d t}	\mathcal{E}_s(t)+\mathcal{D}_s(t) \lesssim & (1 + \| H_{ext} \|_{H^s}) ( \| \tfrac{1}{\rho} \|_{L^\infty}^\frac{1}{2} + \| \tfrac{1}{\rho} \|_{L^\infty}^\frac{s+1}{2} ) \big[ \mathcal{E}_s^\frac{3}{2} (t) + \mathcal{E}_s^\frac{s + \gamma + 3}{2} (t) + \mathcal{E}_s^\frac{s+5}{2} (t) \big] \\
& + (1 + \| H_{ext} \|_{H^s}) ( 1 + \| \tfrac{1}{\rho} \|_{L^\infty}^\frac{s+1}{2} ) \big[  1 + \mathcal{E}_s^\frac{3}{2} (t) + \mathcal{E}_s^\frac{\gamma}{2} (t) \big] \mathcal{D}_s^\frac{1}{2} (t) \,.
	\end{aligned}
\end{equation}

It remains to control the quantity $\|\frac{1}{\rho}\|_{L^\infty}$ . Recall that the density $\rho$ satisfies the equation
\begin{equation*}
\begin{aligned}
\p_t\rho+v\cdot\nabla\rho+\rho\nabla\cdot v=0\,,
\end{aligned}
\end{equation*}
with initial data $\rho|_{t=0}=\rho_0$. Define the trajectory map
\[
\begin{cases}
\frac{\d}{\d t} X(t, x) = v (t, X(t, x)) \,, \\
X(0, x) = x \,.
\end{cases}
\]
Then the density $\rho$ along the trajectory map $X (t,x)$ satisfies
\begin{equation*}
\begin{aligned}
\p_t(\rho(t,X(t,x)))+(\nabla\cdot v)(t,X(t,x))\rho(t,X(t,x))=0
\,,
\end{aligned}
\end{equation*}
which implies that
\begin{equation*}
\begin{aligned}
\p_t\Big[e^{\int^t_0(\nabla\cdot v)(\tau,X(\tau,x))d\tau}\rho(t,X(t,x))\Big]=0 \,.
\end{aligned}
\end{equation*}
Integrating the above equation over $[0,t]$ gives
\begin{equation*}
\begin{aligned}
\rho(t,X(t,x))=\rho_0 (x) e^{-\int^t_0(\nabla\cdot v)(\tau,X(\tau,x))d\tau} \,,
\end{aligned}
\end{equation*}
which shows that $\rho(t, X(t,x)) > 0$ provided that the initial data satisfies $\rho_0(x) > 0$. Consequently, we have 
\begin{equation}\label{rho-infty}
\begin{aligned}
\| \frac{1}{\rho} \|_{L^\infty} &\leq \| \frac{1}{\rho_0} \|_{L^\infty} e^{\int_0^t \| \nabla \cdot v \|_{L^\infty} d\tau} \leq \| \frac{1}{\rho_0} \|_{L^\infty} e^{c \int_0^t \| \nabla v \|_{H^2} d\tau} \leq\| \frac{1}{\rho_0} \|_{L^\infty} e^{c t \sup_{\tau \in [0,t]} \mathcal{E}_s^{\frac{1}{2}}(\tau)} 
\end{aligned}
\end{equation}
for some constant $c > 0$. As a result, \eqref{Ener-Bnds-Local2} together with \eqref{rho-infty} yields the inequality \eqref{Ener-Bnds-Local}. This completes the proof of Proposition \ref{Prop-Loc}.

\subsection{Local well-posedness of \eqref{CMEH-eq}}

In this subsection, based on the energy estimate \eqref{Ener-Bnds-Local} in Proposition \ref{Prop-Loc}, we establish the local existence and uniqueness of solutions to the compressible Navier-Stokes-Landau-Lifshitz-Gilbert system \eqref{CMEH-eq} for large initial data.

We first construct the following nonlinear iterative approximation scheme:
\begin{equation}\label{Ite-App-Eq}
\begin{cases}
\partial_t \rho^{n+1} + v^{n} \cdot \nabla \rho^{n+1} + \rho^{n+1} \nabla \cdot v^{n}=0, \\[4pt]
\rho^{n}(\partial_t v^{n+1} + v^{n} \cdot \nabla v^{n+1} ) + P'( \rho^{n} ) \nabla \rho^{n+1} + \nabla ( - A | \nabla M^n |^{2} + \mu_{0} M^{n} \cdot H_{ext} ) \\
=\mu \Delta v^{n+1} + ( \mu + \xi ) \nabla ( \nabla \cdot v^{n+1} ) + \nabla \cdot ( \rho^{n} F^{n} F^{n \top} ) \\
\qquad \qquad -2 A \nabla \cdot ( \nabla M^{n} \odot \nabla M^{n} ) + \mu_{0} ( \nabla H_{ext} )^\top M^{n}, \\[4pt]
\partial_t F^{n+1} + v^{n} \cdot \nabla F^{n+1} = \nabla v^{n} F^{n}, \\[4pt]
\partial_t M^{n+1} + v^{n} \cdot \nabla M^{n+1} = 2 A \lambda \Delta M^{n+1} + \lambda \mu_{0} H_{ext} \\
\qquad \qquad \qquad \qquad \qquad + \Gamma ( M^{n+1} ) M^{n+1} - \gamma M^{n+1} \times ( \Delta M^{n+1} + H_{ext} ) , \\[4pt]
\left| M^{n+1} \right|=1, \Gamma ( M^{n+1} ) = 2 A \lambda | \nabla M^{n+1}|^2 - \lambda \mu_{0} M^{n+1} \cdot H_{ext} .
\end{cases} 
\end{equation}
The initial data are given by
\begin{equation*}
    \begin{aligned}
        ( \rho^{n+1}, v^{n+1}, F^{n+1}, M^{n+1} ) |_{t=0} = ( \rho_0, v_0, F_0, M_0 ) (x) \,.
    \end{aligned}
\end{equation*}
The iteration starts from $n = 0$ with
\begin{equation*}
    \begin{aligned}
        ( \rho^0, v^0, F^0, M^0 ) (t,x) = ( \rho_0, v_0, F_0, M_0 ) (x) \,.
    \end{aligned}
\end{equation*}
The iteration for the compressible variables $(\rho^{n+1}, v^{n+1})$ is inspired by (5.1) in Section 5 of \cite{JLT-M3AS-2020}. The iteration for the deformation gradient tensor $F^{n+1}$ and the magnetization $M^{n+1}$ follows the approach in (4.1) of Section 4 of \cite{JLL-JDE-2023}. The nonlinearity in the iterative scheme arises from the iterated Lagrangian multiplier $\Gamma(M^{n+1})$, which enforces the geometric constraint $|M^{n+1}|=1$. This nonlinear iterative scheme for preserving the spherical constraint was first introduced in \cite{JL-SIMA-2019} for the incompressible hyperbolic Ericksen–Leslie model of liquid crystals. Subsequently, the ideas were successfully applied in \cite{JLL-JDE-2023,JLT-M3AS-2020}.

Following the ideas of \cite{JLL-JDE-2023,JL-SIMA-2019,JLT-M3AS-2020}, the iterative approximation \eqref{Ite-App-Eq} admits a uniform estimate in $n\ge0$ that has essentially the same structure as the local a priori estimate \eqref{Ener-Bnds-Local} in Proposition \ref{Prop-Loc}. Using this uniform estimate, the local well-posedness can be obtained. For simplicity, we will directly work with the local a priori estimate \eqref{Ener-Bnds-Local}.

Under the assumption that $H_{\text{ext}} \in L^\infty(0,T_0;H^s)$ for some $T_0 >0$ as in Theorem \ref{Thm1}, the bound \eqref{Ener-Bnds-Local} implies that
\begin{equation}\label{En-2}
\begin{aligned}
\frac{\d}{\d t} \mathcal{E}_s (t) + \mathcal{D}_s(t) \leq C (1 + e^{c(s+1)t \sup_{ \tau \in [0, t] } \mathcal{E}_s (\tau) }  ) ( 1 + \mathcal{E}_s^{ \frac{s}{2} + \gamma + 3 } (t) )
\end{aligned}
\end{equation}
for all $t \in [0,T_0]$, where the constant $C = C(s, \| H_{ext} \|_{L^\infty(0,T_0;H^s)}, \| \frac{1}{\rho_0} \|_{L^\infty} ) > 0$. Recall that $\mathcal{E}_s(0) = \mathcal{E}_s^{in} < \infty$, where the local initial energy $\mathcal{E}_s^{in}$ is defined in \eqref{E_s-in}. Define
\begin{equation}\label{En-3}
\begin{aligned}
T_1 = \sup\{\tau \in [0, T_0); \sup_{t \in [0,\tau]} \mathcal{E}_s (t) \leq 2 \mathcal{E}_s^{in}\} \geq 0 \,.
\end{aligned}
\end{equation}
By the continuity of the energy functional $\mathcal{E}_s(t)$, we have $T_1 > 0$. Then inequality \eqref{En-2} yields, for all $t \in [0, T_1]$,
\begin{equation*}
\begin{aligned}
\frac{\d}{\d t} \mathcal{E}_s(t) + \mathcal{D}_s(t) \leq \underbrace{ C \left( 1 + e^{2 c (s+1) T_0 \mathcal{E}_s^{in}} \right) \left( 1 + ( 2 \mathcal{E}_s^{in} )^{ \frac{s}{2} + \gamma + 3 } \right) }_{C_1} .
\end{aligned}
\end{equation*}
Integrating this inequality over $[0,t] \subseteq [0,T_1]$ and using the non-negativity of $\mathcal{D}_s(t)$, we obtain
\begin{equation}\label{En-4}
\begin{aligned}
\mathcal{E}_s(t) \leq \mathcal{E}_s^{in} + C_1 t.
\end{aligned}
\end{equation}
Choosing
\[
t_{*} = \frac{\mathcal{E}_s^{in}}{2 C_1} > 0,
\]
which depends only on $s$, the coefficients, $\mathcal{E}_s^{in}$ and $\|H_{ext}\|_{L^{\infty}(0,T_0;H^{s})}$, it follows from \eqref{En-4} that
\[
\mathcal{E}_s(t) \leq \frac{3}{2} \mathcal{E}_s^{in} < 2 \mathcal{E}_s^{in} \quad (\forall t \in [0, t_{*}]).
\]
Consequently, the number $T_1$ defined in \eqref{En-3} admits a uniform positive lower bound $t_{*} > 0$; i.e.,
\[
T_1 \geq t_{*} > 0.
\]
Hence, for all $t \in [0, t_{*}]$,
\[
\frac{\d}{\d t} \mathcal{E}_s(t) + \mathcal{D}_s(t) \leq C_1,
\]
which implies
\[
\sup_{t \in [0,t_{*}]} \mathcal{E}_s (t) + \int_{0}^{t_{*}} \mathcal{D}_s(t) \, \d t \leq \mathcal{E}_s^{in} + C_1 t_{*}.
\]
This establishes the local existence of solutions to the compressible NSLLG system \eqref{CMEH-eq}. The uniqueness result in the classical solution regime is standard; for instance, one may refer to Section 4 of \cite{JLM-CMS-2026}. Thus the proof of Theorem \ref{Thm1} is complete.

\section{Global well-posedness near constant equilibrium: Proof of Theorem \ref{Thm2}}\label{Sec:4}

In this section, we establish the global well-posedness of the compressible NSLLG system \eqref{CMEH-eq} without external magnetic field ($H_{{ext}} = 0$) near the constant equilibrium $(1, 0, I, M_e) \in \mathbb{R} \times \mathbb{R}^3 \times \mathbb{R}^{3 \times 3} \times \mathbb{S}^2$. More precisely, we prove the global existence of solutions to \eqref{CMEH-eqG} for small initial data. We first examine the damping mechanism of the deformation gradient tensor $F$ near the identity matrix $I \in \mathbb{R}^{3 \times 3}$. Then we derive global differential energy estimates to obtain the global existence result.

\subsection{Damping mechanism of the deformation gradient tensor $F$ in \eqref{CMEH-eqG}}

From the initial assumptions $\rho_0 > 0$ and $\rho_0 \det F_0 = 1$, Proposition 1 of Qian-Zhang \cite{QZ-ARMA-2010} implies that for all $(t,x)$,
\begin{equation*}
    \begin{aligned}
        \rho \det F = 1 \,, \quad \rho > 0 \,.
    \end{aligned}
\end{equation*}
Thus $F$ is nondegenerate. Let $G = F^{-1}$. Then 
$$(\p_t F+u\cdot \nabla F)G=-F(\p_tG+u\cdot\nabla G),$$
and 
$$\nabla uFG=\nabla u.$$
Consequently,
$$-F(\p_tG+u\cdot\nabla G)=\nabla u,$$ 
i.e., 
$$\p_tG+u\cdot\nabla G+G\nabla u=0.$$
Moreover, by Lemma 2.1 of \cite{Sideris-Thomases-CPAM-2005}, we have $\partial_i G^{jk} = \partial_k G^{ji}$ for $i,j,k = 1,2,3$. Define $U = G - I$. Then $G$ satisfies
\begin{equation}\label{G}
\begin{aligned}
\p_t U+u\cdot \nabla U+(U+I)\nabla u=0\,,
\end{aligned}
\end{equation}
and $\p_iU^{jk}=\p_kU^{ji}\quad (i,j,k=1,2,3)$, that is, the matrix $U$ is curl-free. More precisely, let $U(t,x)= ( \mathscr{U}^1 (t,x), \mathscr{U}^2 (t,x), \mathscr{U}^3 (t,x))^\top $, where 
$$\mathscr{U}^j(t,x)=(U^{j1}(t,x),U^{j2}(t,x),U^{j3}(t,x))\quad(j=1,2,3).$$ 
Then each vector field $\mathscr{U}^j$ is curl-free: $\nabla \times \mathscr{U}^j = 0$ for $j = 1,2,3$. By the div-curl theorem, there exists an $\R^3$-valued function $\psi(t,x)=(\psi^1(t,x),\psi^2(t,x),\psi^3(t,x))$ such that 
$$(U^{j1},U^{j2},U^{j3})=\nabla \psi^j \quad (j=1,2,3).$$
In other words, $U(t,x)=\nabla\psi(t,x)$, whose components are given by
\begin{equation}\label{U}
\begin{aligned}
 U^{ij}=\p_i\psi^i\,.
\end{aligned}
\end{equation}
Note that $\psi(t,x)=X^{-1}(t,x)-x$, where $X^{-1}(t,x)$ is the inverse flow map $x (X, t)$ defined in \eqref{FM-x}.

Substituting \eqref{U} into \eqref{G} yields
\begin{equation}\label{psi}
\begin{aligned}
\p_j(\p_t\psi^i+u^i+(u\cdot \nabla )\psi^i)=0 \,, \quad 1\leq i,j\leq 3\,.
\end{aligned}
\end{equation}
By (3.3) of \cite{Lin-Zhang-CPAM-2008}, we have
$$ \| U(t) \|_{H^s}=\|F^{-1}(t)-I\|_{H^s}\sim\|F(t)-I\|_{H^s},$$
provided that the fluctuation satisfies $\|F(t)-I\|_{H^s} < 1$ (or equivalently $\|F^{-1}(t)-I\|_{H^s} < 1$). By the smallness of the initial data, there exists $T > 0$ such that $\|F(t)-I\|_{H^s} < 1$ for all $t \in [0,T]$. Then, using Taylor expansion, the Cauchy–Green tensor becomes
\begin{equation}\label{CG-tensor}
\begin{aligned}
\frac{ W_F (F) F^\top }{ \det F} & = \rho F F^\top = ( 1 + \theta ) ( I + U)^{-1} ( I + U)^{-\top} \\
& = ( 1 + \theta ) ( I - U - U^\top + g(U) ) \,,
\end{aligned}
\end{equation}
where $g(U)=O(|U|^2).$

Using the curl-free property $\p_i U^{jk} = \p_k U^{ji} \ (1\leq i,j,k\leq3)$, we obtain $\nabla \cdot U^\top=\nabla \mathrm{tr} U.$ Hence,
$$\nabla \cdot(\rho FF^\top)=\nabla \theta-\nabla\cdot((1+\theta)\nabla \psi)-\nabla\psi\cdot\nabla\theta-(1+\theta)\nabla \mathrm{tr} U.$$
Since $\rho \det F=1$, we have
$$ 1 + \theta = \det ( I + U ) = 1 + \mathrm{tr} U + O ( |U|^2 ) .$$
Thus $\mathrm{tr} U= \theta + \tilde{g} (U), \ \tilde{g}(U) = O(|U|^2) $. Consequently,
\begin{equation}\label{U000}
\begin{aligned}
\nabla \cdot ( \rho F F^\top ) & = \nabla \theta - \nabla \cdot \left((1+\theta)\nabla \psi\right) - \nabla \psi \cdot \nabla \theta - (1+\theta)\nabla\left(\theta + \tilde{g}(U)\right) \\
& = -\nabla \cdot \left((1+\theta)\nabla \psi\right) - \nabla \psi \cdot \nabla \theta - \nabla \tilde{g}(U)  - \theta \nabla\left(\theta + \tilde{g}(U)\right) \\
& = -(1+\theta)\Delta \psi - 2\nabla \psi \cdot \nabla \theta - \nabla \tilde{g}(U)- \theta \nabla\left(\theta + \tilde{g}(U)\right)\,.
\end{aligned}
\end{equation}
From \eqref{CMEH-eqG}, \eqref{psi} and \eqref{U000}, we deduce the following reformulated system:                   
\begin{align}\label{CMEH-eqGd}
	\begin{cases}
		\p_t \theta +u\cdot\nabla\theta+(1+\theta)\nabla\cdot u = 0 \,, \\
		\begin{aligned}
			(1&+\theta)(\p_t u+u\cdot\nabla u) + \nabla \left( P(1+\theta) - A |\nabla d|^2 +\tilde{g}(U)\right) 
			\\[5pt] & 
			=  \mu\Delta u+(\mu+\xi)\nabla(\nabla \cdot u)-(1+\theta)\Delta\psi-2\nabla\psi\cdot\nabla\theta-\theta\nabla(\theta+\tilde{g}(U))-2A \nabla \cdot ( \nabla d \odot \nabla d)  \,, 
		\end{aligned} \\
		\partial_t \psi + u+u\cdot\nabla \psi =0 \,, \\
		\p_t d + u\cdot \nabla d = 2A\lambda \Delta d+2A\lambda|\nabla d|^2(d+M_e)-2A\gamma(d+M_e)\times \Delta d
		\,, \\ 
	|d|^2+2M_e\cdot d=0 \,. 
	\end{cases}
\end{align}

\subsection{Global existence near constant equilibrium: Proof of Theorem \ref{Thm2}}

Recall the fluctuations
\begin{equation}\label{Flc}
    \begin{aligned}
        \rho = 1 + \theta \,, \ M = d + M_e \,, \ v = 0 + u \,, \ U = F^{-1} - I \,, \ U^{ij} = \partial_j \psi^i \,.
    \end{aligned}
\end{equation}
The global energy functional $\mathbf{E}_s (t)$ in \eqref{close-eq5} and the global dissipation rate $\mathbf{D}_s (t)$ in \eqref{close-eq6} can be expressed as
\begin{equation}\label{close-eq5-1}
\begin{aligned}
\mathbf{E}_s (t) = \| \theta \|^2_{H^s} + \| u \|^2_{H^s} + | \bar{d} |^2 + \| d -\bar{d} \|^2_{H^s} + \| \nabla d \|^2_{H^s} + \| \nabla \psi \|^2_{H^s} \,,
\end{aligned}
\end{equation}
and 
\begin{equation}\label{close-eq6-1}
\begin{aligned}
\mathbf{D}_s (t) = \| \nabla u \|^2_{H^s} + \| \nabla \theta \|_{H^{s-1}} + \| \nabla \psi |^2_{H^s} + \| \nabla d \|^2_{H^s} + \| \Delta d \|^2_{H^s} + \| \nabla \cdot u \|^2_{H^s} \,.
\end{aligned}
\end{equation}

For any $\delta,\eta>0$, we define the following {\bf instant energy functional}
\begin{equation}\label{close-eq7}
\begin{aligned}
\mathbb{E}_{s;\delta,\eta, \epsilon}(t)&=\epsilon a \gamma \|\theta\|_{H_{ \w (\theta) }^s}^2 + \epsilon \|u\|_{H_{1+\theta}^s}^2 + \epsilon \eta \|u + \nabla \theta\|_{H^{s-1}}^2 - \epsilon \eta \|u\|_{H^{s-1}}^2 - \epsilon \eta \|\nabla \theta\|_{H^{s-1}}^2 + |\bar{d}|^2 \\
& \quad + \|d - \bar{d}\|_{H^s}^2 + \|\nabla d\|_{H^s}^2 + \delta \|\nabla \psi\|_{H^s}^2 + \delta \|u - \bar{u}\|_{H^s}^2 - \frac{2\delta}{\mu} \sum_{|m| \leq s} \langle \partial^m (\psi - \bar{\psi}), \partial^m u \rangle \,,
\end{aligned}
\end{equation}
and the {\bf instant dissipation rate} 
\begin{equation}\label{close-eq8}
\begin{aligned}
\mathbb{D}_{s;\delta,\eta, \epsilon} (t) & = (\epsilon \mu + \delta \mu  - \delta \frac{c_p}{\mu} ) \|\nabla u\|_{H^s}^2 + (\epsilon + \delta) (\mu + \xi ) \|\nabla \cdot u\|_{H^s}^2 + \epsilon \eta a \gamma \|\nabla \theta\|_{H_{ \w (\theta) }^s }^2 \\
& \quad + \frac{\delta}{\mu} \| \nabla \psi\|_{H^s}^2 - (\epsilon + \delta) \sum_{|m| \leq s} \langle \nabla \partial^m \psi, \nabla \partial^m u \rangle - \frac{\mu + \xi}{\mu} \delta \sum_{|m| \leq s} \langle \nabla \cdot \partial^m \psi, \nabla \cdot \partial^m u \rangle \\
& \quad - c_0 \epsilon \eta [ \mu \| \nabla u \|_{H^s} + (\mu + \xi) \| \nabla \cdot u \|_{H^s} + \| \nabla \psi \|_{H^s} ] \| \nabla \theta \|_{ H^{s-1}_{ \w (\theta) } } \\
& \quad - c_1 \delta a \gamma ( \mu \| \nabla u \|_{H^s} + \| \nabla \psi \|_{H^s} ) \| \nabla \theta \|_{ H^{s-1}_{ \w (\theta) } }  + 2A\lambda \|\nabla d\|_{H^s}^2 + 2A\lambda \|\Delta d\|_{H^s}^2 \,,
\end{aligned}
\end{equation}
where $\w (\theta)$ satisfies $a \gamma \w (\theta) = \frac{P' (1 + \theta)}{1 + \theta}$, hence, 
\begin{equation}\label{w-weight}
  \begin{aligned}
    \w (\theta) = (1 + \theta)^{\gamma - 2} \,.
  \end{aligned}
\end{equation}
Here the notations $\bar{d}$, $\bar{u}$ and $\bar{\psi}$ are given in \eqref{f-bar}. Moreover, the constant $c_0 > 0$ is independent of $\delta, \eta, \epsilon$. 

We then have the following lemma concerning the positivity of the functionals $\mathbb{E}_{s;\delta,\eta, \epsilon}(t)$ and $\mathbb{D}_{s;\delta,\eta, \epsilon}(t)$.

\begin{lemma}\label{Lmm-Glob-ED}
    If the coefficients satisfy
\begin{equation}\label{close-eq19}
\begin{aligned}
a > 0, \ \gamma > 1, \ A > 0, \ \lambda > 0, \ \mu > 0, \ \mu + \xi > 0 \,,
\end{aligned}
\end{equation}
then there exist $\delta,\eta, \epsilon>0$ such that
\begin{equation}\label{close-eq20}
\begin{aligned}
c_1 \mathbf{D}_s (t) \leq \mathbb{D}_{s;\delta,\eta, \epsilon} (t) \leq \hat{c_1} \mathbf{D}_s (t), \ c_2 \mathbf{E}_s (t) \leq \mathbb{E}_{s;\delta,\eta, \epsilon} (t) \leq \hat{c_2} \mathbf{E}_s (t)
\end{aligned}
\end{equation}
for some constants $c_1,c_2,\hat{c_1},\hat{c_2}>0$.
\end{lemma}

\begin{proof}

The upper bounds for the functionals $\mathbb{E}_{s;\delta,\eta,\epsilon}(t)$ and $\mathbb{D}_{s;\delta,\eta,\epsilon}(t)$ in \eqref{close-eq20} follow readily from H\"older's inequality and Young's inequality. We now focus on the lower bounds.

For $\mathbb{D}_{s;\delta,\eta, \epsilon}(t)$, we need to find some $\delta, \eta, \epsilon>0$ such that the quantity 
\begin{equation}\label{close-eq9}
\begin{aligned}
& (\epsilon \mu + \delta \mu  - \delta \frac{c_p}{\mu} ) \|\nabla u\|_{H^s}^2 + (\epsilon + \delta) (\mu + \xi ) \|\nabla \cdot u\|_{H^s}^2 + \epsilon \eta a \gamma \|\nabla \theta\|_{H_{ \w (\theta) }^s }^2 \\
& \quad + \frac{\delta}{\mu} \| \nabla \psi\|_{H^s}^2 - (\epsilon + \delta) \sum_{|m| \leq s} \langle \nabla \partial^m \psi, \nabla \partial^m u \rangle - \frac{\mu + \xi}{\mu} \delta \sum_{|m| \leq s} \langle \nabla \cdot \partial^m \psi, \nabla \cdot \partial^m u \rangle \\
& \quad - c_0 \epsilon \eta [ \mu \| \nabla u \|_{H^s} + (\mu + \xi) \| \nabla \cdot u \|_{H^s} + \| \nabla \psi \|_{H^s} ] \| \nabla \theta \|_{ H^{s-1}_{ \w (\theta) } } \\
& \quad - c_1 \delta a \gamma ( \mu \| \nabla u \|_{H^s} + \| \nabla \psi \|_{H^s} ) \| \nabla \theta \|_{ H^{s-1}_{ \w (\theta) } }
\end{aligned}
\end{equation}
admits a lower bound of the form $c_*(\|\nabla u\|^2_{H^s}+\|\nabla\cdot u\|^2_{H^s}+\|\nabla \psi\|^2_{H^s} + \| \nabla \theta \|_{ H^{s-1}_{ \w (\theta) } }^2)$ for some $c_*>0$. It suffices to show that the quadratic form
\begin{equation*}
\begin{aligned}
F(V) = V
\underbrace{
\begin{pmatrix}
\epsilon \mu + \delta \mu - \delta \frac{c_p}{\mu} & 0 & - \frac{1 + \delta}{2} & - \frac{1}{2} (c_0 \epsilon \eta \mu + c_1 \delta \mu a \gamma) \\
0 & (\epsilon + \delta) (\mu + \xi) & - \delta \frac{\mu + \xi}{2 \mu} & - \frac{1}{2} c_0 \epsilon \eta (\mu + \xi) \\
- \frac{1 + \delta}{2} & - \delta \frac{\mu + \xi}{2 \mu} & \frac{\delta}{\mu} & - \frac{1}{2} (c_0 \epsilon \eta + c_1 \delta a \gamma) \\
- \frac{1}{2} (c_0 \epsilon \eta \mu + c_1 \delta \mu a \gamma) & - \frac{1}{2} c_0 \epsilon \eta (\mu + \xi) & - \frac{1}{2} (c_0 \epsilon \eta + c_1 \delta a \gamma) & \epsilon \eta a \gamma
\end{pmatrix}}_{: = \mathcal{M}}
V^\top
\end{aligned}
\end{equation*}
is positive definite, where
\begin{equation*}
  \begin{aligned}
    V = (X, Y, Z, W) \,, \ X = \| \nabla u \|_{H^s} \,, Y = \| \nabla \cdot u \|_{H^s} \,, Z = \| \nabla \psi \|_{H^s} \,, W = \| \nabla \theta \|_{H^{s-1}_{ \w (\theta) }} \,.
  \end{aligned}
\end{equation*}
We remark that the quadratic form $F(V)$ provides a lower bound for the quantity \eqref{close-eq9}. The matrix $\mathcal{M}$ is positive definite if and only if the following conditions hold:
  \begin{equation}\label{M1234}
    \begin{aligned}
      & \mathcal{M}_1 = \epsilon \mu + \delta \mu - \delta \frac{c_p}{\mu} > 0 \,, \quad \mathcal{M}_2 = \det \begin{pmatrix}
\epsilon \mu + \delta \mu - \delta \frac{c_p}{\mu} & 0 \\
0 & (\epsilon + \delta) (\mu + \xi)
\end{pmatrix} > 0 \,, \\
      & \mathcal{M}_3 = \det \begin{pmatrix}
\epsilon \mu + \delta \mu - \delta \frac{c_p}{\mu} & 0 & - \frac{1 + \delta}{2} \\
0 & (\epsilon + \delta) (\mu + \xi) & - \delta \frac{\mu + \xi}{2 \mu} \\
- \frac{1 + \delta}{2} & - \delta \frac{\mu + \xi}{2 \mu} & \frac{\delta}{\mu} 
\end{pmatrix} > 0 \,, \quad \mathcal{M}_4 = \det \mathcal{M} > 0 \,.
    \end{aligned}
  \end{equation}
Observe that $\mathcal{M}_1 > 0$ and $\mathcal{M}_2 > 0$ imply
\begin{equation}\label{M12}
  \begin{aligned}
    \epsilon \mu + \delta \mu - \delta \frac{c_p}{\mu} > 0 \,, \ \mu + \xi > 0 \,.
  \end{aligned}
\end{equation}
Moreover, $\mathcal{M}_3 > 0$ and $\mathcal{M}_4 > 0$ can be computed as
\begin{equation}\label{M3}
  \begin{aligned}
    \mathcal{M}_3 = (\mu + \xi) \Big\{ [ (1 - \frac{c_p}{\mu^2}) ( 1 - \frac{\mu + \xi}{4} ) - \frac{1}{4} ] \delta^3 + [ \epsilon ( 2 - \frac{c_p}{\mu^2} - \frac{\mu + \xi}{4} ) - \frac{3}{4} ] \delta^2 \\
    + (\epsilon^2 - \frac{3}{4}) \delta - \frac{1}{4} \Big\} > 0 \,,
  \end{aligned}
\end{equation}
and
\begin{equation}\label{M4}
  \begin{aligned}
    \mathcal{M}_4 = \frac{(\epsilon + \delta) ( \mu + \xi )}{\mu} ( \mu^2 (\epsilon + \delta) - \delta c_p ) ( \Xi_2 \eta^2 + \Xi_1 \eta + \Xi_0 ) > 0 \,,
  \end{aligned}
\end{equation}
where
  \begin{align*}
    \Xi_{2} &= c_{0}^{2}\epsilon^{2}\left[-\frac{1}{4}\left(\frac{\mu^{3}}{\mu^{2}(\epsilon+\delta)-\delta c_{p}}+\frac{\mu+\xi}{\epsilon+\delta}\right) U-\left(\frac{1}{2}+\frac{\mu^{2}(1+\delta)}{4\left(\mu^{2}(\epsilon+\delta)-\delta c_{p}\right)}+\frac{\delta(\mu+\xi)}{4\mu(\epsilon+\delta)}\right)^{2}\right], \\
    \Xi_{1} &= \epsilon\left[U\left(a\gamma-\frac{\mu^{3} c_{0}c_{1}\delta a\gamma}{2\left(\mu^{2}(\epsilon+\delta)-\delta c_{p}\right)}\right) \right. \\
    & \quad \left.-2 c_{0}c_{1}\delta a\gamma\left(\frac{1}{2}+\frac{\mu^{2}(1+\delta)}{4\left(\mu^{2}(\epsilon+\delta)-\delta c_{p}\right)}\right)\left(\frac{1}{2}+\frac{\mu^{2}(1+\delta)}{4\left(\mu^{2}(\epsilon+\delta)-\delta c_{p}\right)}+\frac{\delta(\mu+\xi)}{4\mu(\epsilon+\delta)}\right)\right], \\
    \Xi_{0} &=  \left(c_{1}\delta a\gamma\right)^{2}\left[\frac{\mu^{3} U}{4\left(\mu^{2}(\epsilon+\delta)-\delta c_{p}\right)}+\left(\frac{1}{2}+\frac{\mu^{2}(1+\delta)}{4\left(\mu^{2}(\epsilon+\delta)-\delta c_{p}\right)}\right)^{2}\right], \\
    U &= \frac{\delta}{\mu}-\frac{\mu(1+\delta)^{2}}{4\left(\mu^{2}(\epsilon+\delta)-\delta c_{p}\right)}-\frac{\delta^{2}(\mu+\xi)}{4\mu^{2}(\epsilon+\delta)}.
  \end{align*}
  To guarantee \eqref{M4}, it is sufficient to require
  \begin{equation}\label{Xi0}
    \begin{aligned}
      \Xi_0 > 0 \,,
    \end{aligned}
  \end{equation}
  which yields \eqref{M4} by finally taking $\eta > 0$ sufficiently small.

We choose $\epsilon \gg 1$ and $0 < \delta \ll 1$ such that $\epsilon^2 \delta \gg 1$. It is easy to see that \eqref{M12} holds, i.e., $\mathcal{M}_1, \mathcal{M}_2 > 0$. Moreover, to verify \eqref{M3}, it suffices to show
\begin{equation}\label{M3-1}
  \begin{aligned}
    \epsilon^2 \delta > ( \frac{c_p}{\mu^2} + \frac{\mu + \xi}{4} )  \frac{\delta}{\epsilon} \epsilon^2 \delta + \frac{1}{4} + \frac{3}{4} \delta^2 - [ (1 - \frac{c_p}{\mu^2}) ( 1 - \frac{\mu + \xi}{4} ) - \frac{1}{4} ] \delta^3 + \frac{3}{4} \delta \,.
  \end{aligned}
\end{equation}
Observe that under the choices $\epsilon \gg 1$, $0 < \delta \ll 1$ and $\epsilon^2 \delta \gg 1$, we have $0 < \bigl( \frac{c_p}{\mu^2} + \frac{\mu + \xi}{4} \bigr) \frac{\delta}{\epsilon} \ll 1$ and
\[
\Bigl| \frac{3}{4} \delta^2 - \Bigl[ \Bigl(1 - \frac{c_p}{\mu^2}\Bigr) \Bigl(1 - \frac{\mu + \xi}{4}\Bigr) - \frac{1}{4} \Bigr] \delta^3 + \frac{3}{4} \delta \Bigr| \ll 1.
\] 
Consequently, \eqref{M3-1} holds, which implies \eqref{M3}.

We now turn to \eqref{M4}. We first verify \eqref{Xi0}, which can be equivalently written as
\begin{equation*}
  \begin{aligned}
    \mu^2 \epsilon + \mu^2 ( 1 + 3 \delta) > c_p \mu^2 \delta + \frac{\mu (\mu + \xi)}{4} \frac{\delta^2}{\epsilon + \delta} \,.
  \end{aligned}
\end{equation*}
Since $\mu > 0$, the choices $\epsilon \gg 1$, $0 < \delta \ll 1$ and $\epsilon^2 \delta \gg 1$ imply the above inequality, hence $\Xi_0 > 0$. Then there exists a small $\eta_0 = \eta_0(\delta,\epsilon) > 0$ such that
\begin{equation*}
  \begin{aligned}
    \Xi_2 \eta^2 + \Xi_1 \eta + \Xi_0 > 0
  \end{aligned}
\end{equation*}
for all $0 < \eta < \eta_0$. Note that $\frac{(\epsilon + \delta) ( \mu + \xi )}{\mu} ( \mu^2 (\epsilon + \delta) - \delta c_p ) > 0$. Thus \eqref{M4} follows immediately, i.e., $\mathcal{M}_4 > 0$.

Therefore, under the choices $\epsilon \gg 1$, $0 < \delta \ll 1$ and $\epsilon^2 \delta \gg 1$, the quadratic form $F(V)$ is positive definite. In other words, there exists a constant $c_* > 0$ such that
\[
F(V) \geq c_* |V|^2 \quad \text{for all } V = (X,Y,Z,W) \in \mathbb{R}^4.
\]
As a consequence,
\begin{equation}\label{close-eq12}
\begin{aligned}
& (\epsilon \mu + \delta \mu  - \delta \frac{c_p}{\mu} ) \|\nabla u\|_{H^s}^2 + (\epsilon + \delta) (\mu + \xi ) \|\nabla \cdot u\|_{H^s}^2 + \epsilon \eta a \gamma \|\nabla \theta\|_{H_{ \w (\theta) }^s }^2 \\
& \quad + \frac{\delta}{\mu} \| \nabla \psi\|_{H^s}^2 - (\epsilon + \delta) \sum_{|m| \leq s} \langle \nabla \partial^m \psi, \nabla \partial^m u \rangle - \frac{\mu + \xi}{\mu} \delta \sum_{|m| \leq s} \langle \nabla \cdot \partial^m \psi, \nabla \cdot \partial^m u \rangle \\
& \quad - c_0 \epsilon \eta [ \mu \| \nabla u \|_{H^s} + (\mu + \xi) \| \nabla \cdot u \|_{H^s} + \| \nabla \psi \|_{H^s} ] \| \nabla \theta \|_{ H^{s-1}_{ \w (\theta) } } \\
& \quad - c_1 \delta a \gamma ( \mu \| \nabla u \|_{H^s} + \| \nabla \psi \|_{H^s} ) \| \nabla \theta \|_{ H^{s-1}_{ \w (\theta) } } \\
& \geq c_*(\|\nabla u\|^2_{H^s}+\|\nabla\cdot u\|^2_{H^s}+\|\nabla \psi\|^2_{H^s} + \| \nabla \theta \|_{ H^{s-1}_{ \w (\theta) } }^2) \,,
\end{aligned}
\end{equation}
which implies 
\begin{equation*}
\begin{aligned}
\mathbb{D}_{s;\delta,\eta, \epsilon}(t)
&\geq c_*(\|\nabla u\|^2_{H^s}+\|\nabla\cdot u\|^2_{H^s}+\|\nabla \psi\|^2_{H^s} + \| \nabla \theta \|_{ H^{s-1}_{ \w (\theta) } }^2) + 2A\lambda \left( \|\nabla d\|_{H^{s}}^{2} + \|\Delta d\|_{H^{s}}^{2} \right)\\&
\geq c_{1} \left( \|\nabla \theta\|_{H^{s-1}}^{2} + \|\nabla u\|_{H^{s}}^{2} + \|\nabla \cdot u\|_{H^{s}}^{2} + \|\nabla \psi\|_{H^{s}}^{2} + \|\nabla d\|_{H^{s}}^{2} + \|\Delta d\|_{H^{s}}^{2} \right)
\\&=c_1\mathbf{D}_s(t) \,,
\end{aligned}
\end{equation*}
where $c_1= \min\left\{ c_{*} , 2A \lambda, \inf_{(t,x)} \w (\theta) \right\} > 0$.

Next we establish the lower bound of $\mathbb{E}_{s;\delta,\eta, \epsilon}(t)$. Recall the definition of $\mathbb{E}_{s;\delta,\eta}(t)$ in \eqref{close-eq7}. For small $\delta>0$ and large $\epsilon > 0$, we require 
\begin{equation}\label{close-eq14}
\begin{aligned}
\epsilon \| u \|^2_{H^s_{1+ \theta}} + \delta \|\nabla \psi\|_{H^s}^2
- \frac{2\delta}{\mu} \sum_{|m| \leq s} \left\langle \partial^m (\psi - \bar{\psi}), \partial^m u \right\rangle
\geq c_{\#} \left( \|\nabla \psi\|_{H^s}^2 + \|u \|_{H^s_{1 + \theta}}^2 \right).
\end{aligned}
\end{equation}
Note that
\begin{equation*}
\begin{aligned}
\left\langle \partial^m (\psi - \bar{\psi}), \partial^m u \right\rangle
&= \left\langle \partial^m (\psi - \bar{\psi}), \partial^m (u - \bar{u}) + \partial^m \bar{u} \right\rangle = \left\langle \partial^m (\psi - \bar{\psi}), \partial^m (u - \bar{u}) \right\rangle\\&
\leq \left\| \partial^m (\psi - \bar{\psi}) \right\|_{L^2} \left\| \partial^m (u - \bar{u}) \right\|_{L^2} \leq c_{p} \left\| \nabla \partial^m \psi \right\|_{L^2} \left\| \partial^m (u - \bar{u}) \right\|_{L^2} \,,
\end{aligned}
\end{equation*}
where $c_{p} > 0$ is the Poincar\'e constant. Hence,
\begin{equation*}
\begin{aligned}
&\delta \|\nabla \psi\|_{H^s}^2 + \epsilon \| u \|^2_{H^s_{1+ \theta}}
- \frac{2\delta}{\mu} \sum_{|m| \leq s} \left\langle \partial^m (\psi -\bar{\psi}), \partial^m u \right\rangle\\&
\geq \delta \|\nabla \psi\|_{H^s} + \epsilon \| u \|^2_{H^s_{1+ \theta}}
- \frac{2\delta c_{p}}{\mu} \sum_{|m| \leq s} \|\nabla \partial^m \psi\|_{L^2} \|\partial^m (u -\bar{u})\|_{L^2} \\
& \geq \delta \|\nabla \psi\|_{H^s} + \epsilon \| u \|^2_{H^s_{1+ \theta}}
- \frac{4 \delta c_{p}}{\mu} \sup_{t,x} (1 + \theta)^\frac{1}{2} \| \nabla \psi \|_{H^s} \| u \|_{H^s_{1+\theta}} \,.
\end{aligned}
\end{equation*}
Thus we can choose
\begin{equation}\label{close-eq15}
\begin{aligned}
c_\# = \min \{ \frac{\delta}{2}, \epsilon - \frac{4 c_p^2}{\mu^2} \delta  \} > 0
\end{aligned}
\end{equation}
for large $\epsilon \gg 1$ and small $0 < \delta \ll 1$, which ensures that \eqref{close-eq14} holds.

Furthermore, we may select a small $\eta_0' \leq \eta_0$ such that for all $\eta \in (0,\eta_0']$,
\begin{equation*}
\begin{aligned}
&\epsilon a \gamma \|\theta\|_{H_{ \w (\theta) }^s}^2 + c_{\#} \left( \|\nabla \psi\|_{H^s}^2 + \|u \|_{H^s_{1 + \theta}}^2 \right) + \epsilon \eta \|u + \nabla \theta\|_{H^{s-1}}^2 - \epsilon \eta \|u\|_{H^{s-1}}^2 - \epsilon \eta \|\nabla \theta\|_{H^{s-1}}^2\\&
\geq c_{2}^\prime \left( \|\theta\|_{H^{s}}^{2} + \|u\|_{H^{s}}^{2} + \|\nabla \psi\|_{H^s}^2 \right)
\end{aligned}
\end{equation*}
for some constant $c_{2}' > 0$. Consequently,
\begin{equation}\label{close-eq16}
\begin{aligned}
\mathbb{E}_{s;\delta,\eta, \epsilon}(t)
&\geq c_{2}^\prime \left( \|\theta\|_{H^{s}}^{2} + \|u\|_{H^{s}}^{2} + \|\nabla \psi\|_{H^s}^2 \right) + |\bar{d}|^2 + \|d - \bar{d}\|_{H^s}^2 + \|\nabla d\|_{H^s}^2
\\&\geq c_{2} \left( \|\theta\|_{H^s}^2 + \|u\|_{H^s}^2 + \|\nabla \psi\|_{H^s}^2 + |\bar{d}|^2 + \|d - \bar{d}\|_{H^s}^2 + \|\nabla d\|_{H^s}^2 \right) \\
& = c_2 \mathbf{E}_s(t),
\end{aligned}
\end{equation}
where $c_2 = \min\{c_2^\prime, 1 \}>0$. This completes the proof of the lemma.
\end{proof}

Based on Lemma~\ref{Lmm-Glob-ED}, we establish the following global energy estimates associated with the instant energy functional $\mathbb{E}_{s;\delta,\eta, \epsilon}(t)$ in \eqref{close-eq7} and instant dissipation rate $\mathbb{D}_{s;\delta,\eta, \epsilon}(t)$ in \eqref{close-eq8}.

\begin{proposition}\label{Prop-Global}
    Let $s \geq 3$ be an integer and let $(\theta, u, \psi, d)$ be the solution constructed in Theorem~\ref{Thm1} with the fluctuation form \eqref{Flc}. Under the same assumptions as in Lemma~\ref{Lmm-Glob-ED}, we have
\begin{equation}\label{close-eq21}
\begin{aligned}
\frac{1}{2} \frac{\d}{\d t}\mathbb{E}_{s;\delta,\eta, \epsilon}(t)+\mathbb{D}_{s;\delta,\eta, \epsilon}(t)\leq C (1+\mathbb{E}^{\frac{s}{2}}_{s;\delta,\eta, \epsilon}(t))\mathbb{E}^{\frac{1}{2}}_{s;\delta,\eta, \epsilon}(t)\mathbb{D}_{s;\delta,\eta, \epsilon}(t)
\end{aligned}
\end{equation}
for some constant $C > 0$.
\end{proposition}

The proof of Proposition~\ref{Prop-Global} will be given in the next section. We now establish the global existence of solutions to \eqref{CMEH-eqGd} (or equivalently \eqref{CMEH-eqG}) for small initial data by employing the estimate \eqref{close-eq21}.

\begin{proof}[{\bf Global well-posedness: Proof of Theorem \ref{Thm2}}]

    Assume that \( \mathbb{E}_{s;\delta, \eta, \epsilon} (0) \leq \varepsilon_{1} \), where the small \( \varepsilon_{1} > 0 \) is to be determined later. Define
\[
T^{*} = \sup \left\{ t > 0 : \sup_{\tau \in [0,t]} C \left( 1 + \mathbb{E}_{s;\delta, \eta, \epsilon}^{\frac{s}{2}}(\tau) \right) \mathbb{E}_{s;\delta, \eta, \epsilon}^{\frac{1}{2}}(\tau) \leq \frac{1}{2} \right\},
\]
where the constant \( C > 0 \) appearing in Proposition~\ref{Prop-Global}. We now choose \( \varepsilon_{1} > 0 \) so that
\[
C \left( 1 + \varepsilon_{1}^{\frac{s}{2}} \right) \varepsilon_{1}^{\frac{1}{2}} \leq \frac{1}{4}.
\]
For instance, we may take \( \varepsilon_{1} = \min \left\{ 1, \frac{1}{64 C^{2}} \right\} > 0 \), which satisfies the above inequality. Then we have
\[
C \left( 1 + \mathbb{E}_{s;\delta, \eta, \epsilon}^{\frac{s}{2}}(0) \right) \mathbb{E}_{s;\delta, \eta, \epsilon}^{\frac{1}{2}}(0) \leq \frac{1}{4} < \frac{1}{2}.
\]
By the continuity of the instant energy functional $\mathbb{E}_{s;\delta,\eta,\epsilon}(t)$, we obtain $T^{*} > 0$. Consequently, the estimate \eqref{close-eq21} in Proposition~\ref{Prop-Global} yields
\[
\frac{1}{2} \frac{\d}{\d t} \mathbb{E}_{s;\delta, \eta, \epsilon} (t) + \frac{1}{2} \mathbb{D}_{s;\delta, \eta, \epsilon}(t) \leq 0 \quad \text{for } t \in [0, T^{*}],
\]
which implies
\[
\mathbb{E}_{s;\delta, \eta, \epsilon} (t) \leq \mathbb{E}_{s;\delta, \eta, \epsilon} (0) \leq \varepsilon_{1} \quad \forall t \in [0, T^{*}].
\]
Hence,
\[
\sup_{t \in [0,T^{*}]} C \left( 1 + \mathbb{E}_{s;\delta, \eta, \epsilon}^{\frac{s}{2}}(t) \right) \mathbb{E}_{s;\delta, \eta, \epsilon}^{\frac{1}{2}}(t) \leq \frac{1}{4} < \frac{1}{2}.
\]
We claim that \( T^{*} = \infty \). Indeed, if $T^{*} < \infty$, then by continuity there exists a small $\varepsilon_2 > 0$ such that
\[
\sup_{t \in [0, T^{*} + \varepsilon_{2}]} C \left( 1 + \mathbb{E}_{s;\delta, \eta, \epsilon}^{\frac{1}{2}}(t) \right) \mathbb{E}_{s;\delta, \eta, \epsilon}^{\frac{1}{2}}(t) \leq \frac{3}{8} < \frac{1}{2},
\]
which contradicts the definition of \( T^{*} \). Therefore,
\begin{equation}\label{GL-1}
    \begin{aligned}
        \sup_{t \geq 0} \mathbb{E}_{s;\delta, \eta, \epsilon} (t) + \int_{0}^{\infty} \mathbb{D}_{s;\delta, \eta, \epsilon} (t) \, \d t \leq \mathbb{E}_{s;\delta, \eta, \epsilon} (0)
    \end{aligned}
\end{equation}
provided that \( \mathbb{E}_{s;\delta, \eta, \epsilon} (0) \leq \varepsilon_{1} \) with \( \varepsilon_{1} = \min \left\{ 1, \frac{1}{64 C^{2}} \right\} > 0 \).
Applying \eqref{close-eq20} from Lemma~\ref{Lmm-Glob-ED} to \eqref{GL-1} gives
\[
\sup_{t \geq 0} \mathbf{E}_{s}(t) + \int_{0}^{\infty} \mathbf{D}_{s} (t) \, \d t \leq \mathfrak{C} \mathbf{E}_{s}^{in}
\]
provided that \( \mathbf{E}_{s}(0) = \mathbf{E}_{s}^{in} \leq \varepsilon_{0} = \hat{c}_2 \varepsilon_{1} \), where $\mathbf{E}_s (t)$ is defined in \eqref{close-eq5} (or \eqref{close-eq5-1}), $\mathbf{D}_s (t)$ is defined in \eqref{close-eq6} (or \eqref{close-eq6-1}), $\mathbf{E}_{s}^{in}$ is defined in \eqref{Esin-g}, and the constant
\[
\mathfrak{C} = \min \left\{ c_{1}, c_{2} \right\}^{-1} \hat{c}_2 > 0.
\]
This completes the proof of Theorem~\ref{Thm2}.
\end{proof}

\section{Global estimates: Proof of Proposition \ref{Prop-Global}}\label{Sec:Global-es}                         

In this section, our main goal is to establish the global energy estimate for the reformulated system \eqref{CMEH-eqGd}, i.e., to prove Proposition~\ref{Prop-Global}.                                                        

\subsection{Estimates for the $d$-equation in \eqref{CMEH-eqG}}

In this subsection, we estimate separately the zero-frequency part $\bar{d} = \frac{1}{|\mathbb{T}^3|}\int_{\mathbb{T}^3} d \, dx$ and the nonzero-frequency part $d - \bar{d}$ of the fluctuation $d = M - M_e$. We also estimate the gradient $\nabla d$ in $H^s$. First, we have the following lemma.

\begin{lemma}\label{Lmm-G1}
Let $s \geq 3$ be an integer. Then
\begin{equation}\label{d6}
\begin{aligned}  
\frac{1}{2} \frac{\d}{\d t} |\bar{d}|^2 \lesssim ( 1 + \mathbf{E}_s^\frac{1}{2} (t) ) \mathbf{E}_s^\frac{1}{2} (t) \mathbf{D}_s (t) \,,
\end{aligned}
\end{equation}
where the expressions of $\mathbf{E}_s (t)$ and $\mathbf{D}_s (t)$ are defined in \eqref{close-eq5-1} and \eqref{close-eq6-1}, respectively.
\end{lemma}

\begin{proof}

Integrating the forth equation of \eqref{CMEH-eqGd} over $x\in \mathbb{T}^3$ yields
\begin{equation*}
\begin{aligned}
\partial_t \int_{\mathbb{T}^3} d \, \d x + \int_{\mathbb{T}^3} u \cdot \nabla d \, \d x & = 2A \lambda \int_{\mathbb{T}^3} |\nabla d|^2 (d +M_e) \, \d x - 2A \lambda \int_{\mathbb{T}^3} (d +M_e) \times \Delta d \, \d x \,.
\end{aligned}
\end{equation*}
A direct computation gives
\begin{equation*}
\begin{aligned}
\int_{\mathbb{T}^3} u \cdot \nabla d \, \d x & = -\int_{\mathbb{T}^3} (\nabla \cdot u) (d - \bar{d}) \, \d x - \int_{\mathbb{T}^3} \nabla \cdot u \, \d x \, \bar{d} = - \int_{\mathbb{T}^3} (\nabla \cdot u) ( d - \bar{d} ) \, \d x \,,
\end{aligned}
\end{equation*}
where $\bar{d} = \frac{1}{|\mathbb{T}^3|} \int_{\mathbb{T}^3} d \, \d x$. Moreover,
\begin{equation*}
\begin{aligned}
2A\lambda \int_{\mathbb{T}^3} |\nabla d|^2 (d+Me)\,\d x = 2A\lambda \int_{\mathbb{T}^3} |\nabla d|^2 (d-\bar{d}+ Me+\bar{d})\, \d x \,.
\end{aligned}
\end{equation*}
For $1 \leq i \leq 3$, using the Einstein summation convention and the Levi-Civita symbol $\varepsilon_{ijk}$ (with $\varepsilon_{ijk}=1$ for even permutations of $\{123\}$ and $-1$ for odd permutations), we have
\begin{equation*}
\begin{aligned}
 & -2A\gamma \int_{\mathbb{T}^3} [(d+Me)\times \Delta d]^i\, \d x
 = -2A \gamma  \int_{\mathbb{T}^3} \varepsilon_{ijk} (d^j +M_e^k) \Delta d^k\, \d x \\
 & = -2A\gamma \int_{\mathbb{T}^3} \varepsilon_{ijk} (d^j +M_e^j) \partial_l \partial_l d^k\, \d x = 2A \gamma  \int_{\mathbb{T}^3}\varepsilon_{ijk} \partial_l (d^j +M_e^j) \partial_l d^k\, \d x \\
 & = 2A \gamma \int_{\mathbb{T}^3} \varepsilon_{ijk} \partial_l d^j \partial_l d^k\, \d x = 0.
\end{aligned}
\end{equation*}
Thus,
$$ -2A\gamma  \int_{\mathbb{T}^3} (d+M_e) \times \Delta d\, \d x = 0.$$
Consequently,
$$\partial_t \int_{\mathbb{T}^3} d\, \d x = \int_{\mathbb{T}^3} (\nabla \cdot u)(d-\bar{d})\, \d x + 2A \lambda \int_{\mathbb{T}^3} |\nabla d|^2 (d-\bar{d} + M_e + \bar{d} ) \, \d x \,.$$
Recalling $ \bar{d} = \frac{1}{|\mathbb{T}^3|} \int_{\mathbb{T}^3} d\, \d x $. we obtain the equation for the zero-frequency part:
\begin{equation}\label{d1}
\begin{aligned}
\partial_t \bar{d} = \frac{1}{|\mathbb{T}^3|} \int_{\mathbb{T}^3} (\nabla\cdot u)(d-\bar{d})\, \d x +\frac {2A\lambda}{|\mathbb{T}^3|} \int_{\mathbb{T}^3} |\nabla d|^2 (d-\bar{d} +M_e+\bar{d})\, \d x.
\end{aligned}
\end{equation}

Multiplying \eqref{d1} by $\bar{d}$ gives                       \begin{equation}\label{d3}
\begin{aligned}      
\frac{1}{2} \frac{\d}{\d t} |\bar{d}|^2 & = \frac{\bar{d}}{|\mathbb{T}^3|} \cdot \int_{\mathbb{T}^3} (\nabla \cdot u)(d - \bar{d}) \, \d x + \frac{2 A \lambda}{|\mathbb{T}^3|} \bar{d} \cdot \int_{\mathbb{T}^3} |\nabla d|^2 \left( d - \bar{d} + M_e + \bar{d} \right) \d x \,.
\end{aligned}
\end{equation}                              
By H\"older's inequality,
  \begin{equation}\label{d4}
\begin{aligned}      
\frac{\bar{d}}{|\mathbb{T}^3|} \cdot \int_{\mathbb{T}^3} (\nabla \cdot u)(d - \bar{d}) \, \d x &\leq |\bar{d}| \int_{\mathbb{T}^3} |\nabla \cdot u| \, |d - \bar{d}| \, \d x \\
&\lesssim |\bar{d}| \, \|\nabla u\|_{L^2} \, \|d - \bar{d}\|_{L^2} \lesssim |\bar{d}| \, \|\nabla u\|_{L^2} \, \|\nabla d\|_{L^2} \,,
\end{aligned}
\end{equation}
where we used the Poincar\'e inequality $\|d-\bar{d}\|_{L^2} \lesssim\|\nabla d\|_{L^2}$ (since $\int_{\mathbb{T}^3} (d-\bar{d})\, dx = 0$). Furthermore, Hölder's inequality together with the Sobolev embedding $H^1 \hookrightarrow L^\infty$ yields
  \begin{equation}\label{d5}
\begin{aligned}   
\frac{2A\lambda}{|\mathbb{T}^3|} \bar{d} \cdot \int_{\mathbb{T}^3} |\nabla d|^2 \left( d - \bar{d} + M_e + \bar{d} \right) \d x & \lesssim  |\bar{d}| \, \|\nabla d\|_{L^2}^2 \left( 1 + \|d - \bar{d}\|_{L^\infty} + |\bar{d}| \right) \\
&\lesssim \left( 1 + \|d - \bar{d}\|_{H^2} + |\bar{d}| \right) |\bar{d}| \, \|\nabla d\|_{L^2}^2 \,.  
\end{aligned}
\end{equation}
Substituting \eqref{d4} and \eqref{d5} into \eqref{d3}, we obtain
  \begin{equation}\label{d6-1}
\begin{aligned}  
\frac{1}{2} \frac{\d}{\d t} |\bar{d}|^2 & \lesssim |\bar{d}| \, \|\nabla u\|_{L^2} \|\nabla d\|_{L^2} + \left( 1 + \|d - \bar{d}\|_{H^2} + |\bar{d}| \right) |\bar{d}| \, \|\nabla d\|_{L^2}^2 \,.
\end{aligned}
\end{equation}
Recalling the definitions of $\mathbf{E}_s(t)$ and $\mathbf{D}_s(t)$ in \eqref{close-eq5-1} and \eqref{close-eq6-1}, estimate \eqref{d6-1} implies \eqref{d6}. This completes the proof of the lemma.
\end{proof}

We next estimate the nonzero-frequency part $d - \bar{d}$. More precisely, we establish the following lemma.

\begin{lemma}\label{Lmm-G2}
    Let $s \ge 3$ be an integer. Then one has
    \begin{equation}\label{d13}
\begin{aligned}
\frac{1}{2} \frac{\d}{\d t} \| d - \bar{d} \|_{H^s}^2 + 2A \lambda \| \nabla d \|_{H^s}^2 \lesssim ( 1 + \mathbf{E}_s^\frac{1}{2} (t) ) \mathbf{E}_s^\frac{1}{2} (t) \mathbf{D}_s (t) \,,
 \end{aligned}
 \end{equation}
 where the functionals $\mathbf{E}_s (t)$ and $\mathbf{D}_s (t)$ are defined  in \eqref{close-eq5-1} and \eqref{close-eq6-1}, respectively.
\end{lemma}

\begin{proof}
It is derived from the difference $\eqref{CMEH-eqGd}_4$ and \eqref{d1} that
\begin{equation}\label{d2}
\begin{aligned}
\partial_t (d - \bar{d}) + u \cdot \nabla (d - \bar{d}) &= 2A\lambda \Delta d + 2A\lambda |\nabla d|^2 (d - \bar{d} +M_e + \bar{d}) \\
&\quad - 2A\lambda \left( (d - \bar{d}) +M_e + \bar{d} \right) \times \Delta d \\
&\quad - \frac{1}{|\mathbb{T}^3|} \int_{\mathbb{T}^3} (\nabla \cdot u)(d - \bar{d}) \, \d x \\
&\quad - \frac{2A\lambda}{|\mathbb{T}^3|} \int_{\mathbb{T}^3} |\nabla d|^2 (d - \bar{d} +M_e + \bar{d}) \, \d x \,.
\end{aligned}
\end{equation}
Multiplying \eqref{d2} by $d-\bar{d}$ and integrating over $x \in \mathbb{T}^3$, one immediately obtains 
  \begin{equation*}
\begin{aligned}  
\frac{1}{2} \frac{\d}{\d t} \| d - \bar{d} \|_{L^2}^2 + 2A\lambda \| \nabla d \|_{L^2}^2
&= -\left\langle u \cdot \nabla (d - \bar{d}), d - \bar{d} \right\rangle + 2A\lambda \left\langle |\nabla d|^2 (d - \bar{d} +M_e + \bar{d}), d - \bar{d} \right\rangle \\&
\quad - 2A\gamma \left\langle (d - \bar{d} +M_e + \bar{d}) \times \Delta d, d - \bar{d} \right\rangle \\&
\quad- \bar{d} \cdot \left\langle \nabla \cdot u, d - \bar{d} \quad \right\rangle - 2A\lambda \bar{d} \cdot \left\langle |\nabla d|^2, d - \bar{d} +M_e + \bar{d} \right\rangle \,.
\end{aligned}
\end{equation*}
The H\"older inequality, the Poincar\'e inequality $\| d - \bar{d} \|_{L^2} \lesssim \| \nabla d \|_{L^2}$, and the Sobolev embedding $H^2 \hookrightarrow L^\infty$ yield 
\begin{equation*}
\begin{aligned}  
 -\left\langle u \cdot \nabla (d - \bar{d}), d - \bar{d} \right\rangle = \frac{1}{2} \left\langle \nabla \cdot u, |d - \bar{d}|^2 \right\rangle & \lesssim \| \nabla u \|_{L^2} \| d - \bar{d} \|_{L^2} \| d - \bar{d} \|_{L^\infty} \\
 & \lesssim \| d - \bar{d} \|_{H^2} \| \nabla u \|_{L^2} \| \nabla d \|_{L^2} \,,
\end{aligned}
\end{equation*}
and
\begin{equation*}
\begin{aligned}   
2A\lambda \left\langle |\nabla d|^2 (d - \bar{d} +M_e + \bar{d}), d - \bar{d} \right\rangle & \lesssim \| \nabla d \|_{L^2}^2 \| d - \bar{d} \|_{L^\infty} \left( 1 + |\bar{d}| + \| d - \bar{d} \|_{L^\infty} \right) \\
& \lesssim \left( 1 + |\bar{d}| + \| d - \bar{d} \|_{H^2} \right) \| d - \bar{d} \|_{H^2} \| \nabla d \|_{L^2}^2\,,
\end{aligned}
\end{equation*}
and
\begin{align*}
-\bar{d} \cdot \langle \nabla \cdot u, d-\bar{d} \rangle \lesssim | \bar{d} | \|\nabla u\|_{L^2} \|d-\bar{d}\|_{L^2} \lesssim | \bar{d} | \|\nabla u\|_{L^2} \|\nabla d\|_{L^2} \,,
\end{align*}
and
\begin{align*}
-2A \lambda \bar{d} \cdot \langle |\nabla d|^2, d-\bar{d} + Me + t\bar{d} \rangle & \lesssim | \bar{d} | \|\nabla d\|_{L^2}^2 \left( \|d-\bar{d}\|_{L^2} + 1 + | \bar{d} | \right) \\
& \lesssim \left(1+ | \bar{d} | + \|d-\bar{d}\|_{H^2}\right) | \bar{d} | \|\nabla d \|_{L^2}^2 \,.
\end{align*}
Furthermore, a direct calculation shows the following cancellation:
\begin{equation*}
\begin{aligned}  
 & -2A\gamma \left\langle (d - \bar{d} +M_e + \bar{d}) \times \triangle d, d - \bar{d} \right\rangle = -2A\gamma \left\langle (M_e + \bar{d}) \times \triangle d, d - \bar{d} \right\rangle \\
 = & -2A\gamma \left\langle \varepsilon_{ijk} \left(M_e^j + \bar{d}^j\right) \partial_l^2 d^k, d^i - \bar{d}^i \right\rangle = 2A\gamma \left\langle \varepsilon_{ijk} \left(M_e^j + \bar{d}^j\right) \partial_l d^k, \partial_l d^i \right\rangle \\
 = & 2A\gamma \sum_{l=1}^3 \left\langle (M_e + \bar{d}) \times \partial_l d, \partial_l d \right\rangle = 0 \,.
\end{aligned}
\end{equation*}
Collecting all the above relations, one obtains
\begin{equation}\label{d7}
\begin{aligned}
\frac{1}{2} \frac{\d}{\d t} \|d-\bar{d}\|_{L^2}^2 + 2A\lambda \|\nabla d\|_{L^2}^2 &\lesssim \left( | \bar{d} | + \|d-\bar{d}\|_{H^2} \right) \|\nabla u\|_{L^2} \|\nabla d\|_{L^2} \\
&\quad + \left(1+ |\bar{d} |+\|d-\bar{d}\|_{H^2}\right) \left( | \bar{d} | + \|d-\bar{d}\|_{H^2} \right) \|\nabla d\|_{L^2}^2 \,.
\end{aligned}
\end{equation}

We now deal with the higher order derivatives of the nonzero-frequency part $d - \bar{d}$. Applying the derivative operator $\partial^m$ ($1\le |m|\le s$) to \eqref{d2} and taking the $L^2$ inner product of the resulting equation with $\partial^m (d-\bar{d})$, we obtain
\begin{align}\label{d8}
\no & \frac{1}{2} \frac{\d}{\d t} \left\|\partial^m (d - \bar{d})\right\|_{L^2}^2 + 2A\lambda \left\|\nabla \partial^m d\right\|_{L^2}^2 \\
\no = & -\left\langle \partial^m \left[ u \cdot \nabla (d - \bar{d}) \right], \partial^m (d - \bar{d}) \right\rangle \\
& + 2A\lambda \left\langle \partial^m \left[ |\nabla d|^2 \left(d - \bar{d} +M_e+ \bar{d}\right) \right], \partial^m (d - \bar{d}) \right\rangle \\
\no & - 2A\lambda \left\langle \partial^m \left[ \left(d - \bar{d} +M_e +\bar{d}\right) \times \nabla d \right], \partial^m (d - \bar{d}) \right\rangle.
\end{align}
Since $1\leq |m|\leq s$, we easily have
\begin{equation*}
\begin{aligned}
&-\langle \partial^m[u\cdot\nabla(d-\bar{d})], \partial^m(d-\bar{d})\rangle= \frac{1}{2} \langle \nabla \cdot u, |\partial^m(d-\bar{d})|^2 \rangle 
\end{aligned}
\end{equation*}
Hence,
\begin{equation}\label{d9}
\begin{aligned}
&- \sum_{0 \neq m' \leq m} C_m^{m'} \langle \partial^{m'} u \cdot \nabla \partial^{m-m'}(d-\bar{d}), \partial^m(d-\bar{d}) \rangle \\
&\lesssim \|\nabla u\|_{L^{\infty}} \|\nabla d\|_{H^s}\|\p^m(d-\bar{d})\|_{L^2} \\
&\quad + \sum_{0 \neq m' \leq m} \|\partial^{m'} u\|_{L^4} \|\nabla \partial^{m-m'}(d-\bar{d})\|_{L^4} \|\partial^m (d-\bar{d})\|_{L^2}\\&
 \lesssim\|\p^m(d-\bar{d})\|_{L^2}\|\nabla u\|_{H^s}\|\nabla d\|_{H^s}\,.
\end{aligned}
\end{equation}
Moreover,
\begin{equation}\label{d10}
\begin{aligned}
&2A\lambda \left\langle \partial^m \left[ |\nabla d|^2 \left(d - \bar{d} +M_e +\bar{d}\right) \right], \partial^m (d - \bar{d}) \right\rangle \\
&= 2A\lambda \left\langle \left(\partial^m |\nabla d|^2\right) \left(d - \bar{d} +M_e +\bar{d}\right), \partial^m (d - \bar{d}) \right\rangle \\
&\quad + 2A\lambda \sum_{\substack{0 \neq m' \leq m}} C_m^{m'} \left\langle \partial^{m - m'} |\nabla d|^2 \, \partial^{m'} (d - \bar{d}), \partial^m (d - \bar{d}) \right\rangle \\
&\lesssim \left\| \partial^m |\nabla d|^2 \right\|_{L^2} \left( 1 + |\bar{d}| + \|d - \bar{d}\|_{L^\infty} \right) \left\| \partial^m (d - \bar{d}) \right\|_{L^2} \\
&\quad + \sum_{\substack{0 \neq m' \leq m}} \left\| \partial^{m - m'} |\nabla d|^2 \right\|_{L^2} \left\| \partial^{m'} (d - \bar{d}) \right\|_{L^4} \left\| \partial^m (d - \bar{d}) \right\|_{L^4} \\
&\lesssim \left( 1 + |\bar{d}| + \|d - \bar{d}\|_{H^2} \right) \|\nabla d\|_{H^s}^3 + \|\nabla d\|_{H^s}^4 \,.
\end{aligned}
\end{equation}
Furthermore, using the cancellation $\langle \p^m(d-\bar{d})\times\Delta d,\p^m(d-\bar{d})\rangle = 0 $, we have
\begin{align*}
&- 2A\gamma \left\langle \partial^m \left[ \left(d - \bar{d} +M_e +\bar{d}\right) \times \Delta d \right], \partial^m (d - \bar{d}) \right\rangle \\
&= \underset{II_1}{\underbrace{-2A\gamma \left\langle \left(d - \bar{d} +M_e +\bar{d}\right) \times \Delta \partial^m d, \partial^m (d - \bar{d}) \right\rangle}} \\
&\quad\underset{II_2}{\underbrace{ - 2A\gamma \sum_{\substack{0 \neq m' < m}} C_m^{m'} \left\langle \partial^{m'} (d - \bar{d}) \times \Delta \partial^{m - m'} d, \partial^m (d - \bar{d}) \right\rangle}} \,.
\end{align*}
For the quantity $II_1$,
\begin{equation*}
\begin{aligned}
&II_1= -2A\gamma \left\langle \varepsilon_{ijk}\left(d^j - \bar{d}^j +M_e^j \bar{d}^j\right) \Delta \partial^m d^j, \partial^m (d^i - \bar{d}^i) \right\rangle \\
&= \sum_{l=1}^3 2A\gamma \left\langle \varepsilon_{ijk} \partial_l (d^j - \bar{d}^j) \, \partial_l \partial^m d^j, \partial^m (d^i - \bar{d}^i) \right\rangle \\
&\quad + \sum_{l=1}^3 2A\gamma \left\langle \varepsilon_{ijk} \left(d^j - \bar{d}^j +M_e^j+ \bar{d}^j\right) \partial_l \partial^m d^j, \partial_l \partial^m d^i \right\rangle\\&
= 2A\gamma  \sum_{l=1}^3 \langle \p_l d \times \p_l \partial^m d, \partial^m (d - \bar{d}) \rangle + 2A \gamma \sum_{l=1}^3 \langle (d - \bar{d} +M_e + \bar{d}) \times \partial_l \partial^m d, \p_l \partial^m d \rangle \\
&\lesssim \| \nabla d \|_{L^{\infty}} \| \nabla \partial^m d \|_{L^{2}} \| \partial^m (d - \bar{d}) \|_{L^{2}} \lesssim \| \nabla d \|_{H^{s}}^{3}\,.
\end{aligned}
\end{equation*}
The term $II_2$ can be bounded as
\begin{equation*}
\begin{aligned}
 II_2 \lesssim \sum_{0 \neq m' < m} \| \partial^{m'} (d - \bar{d}) \|_{L^{4}} \| \Delta \partial^{m - m'} d \|_{L^{2}} \| \partial^m (d - \bar{d}) \|_{L^{4}} \lesssim\| \nabla d \|_{H^{s}}^{3}\,.
\end{aligned}
\end{equation*}
Consequently,
\begin{equation}\label{d11}
\begin{aligned}
-2A\gamma\langle \p^m\big[(d-\bar{d}+M_e+\bar{d}) \times \Delta d\big], \p^m(d-\bar{d})  \rangle\lesssim\|\nabla d\|_{H^3}^3
 \,.
\end{aligned}
\end{equation}
Substituting the estimates \eqref{d9}, \eqref{d10} and \eqref{d11} into \eqref{d8} yields, for $1\le |m|\le s$,
\begin{equation}\label{d12}
\begin{aligned}
\frac{1}{2} \frac{\d}{\d t} \left\| \partial^m (d - \bar{d}) \right\|_{L^2}^2 + 2A \lambda \left\|\nabla \partial^m d \right\|_{H^s}^2 & \lesssim \left\| \partial^m (d - \bar{d}) \right\|_{L^2} \left\| \nabla u \right\|_{H^s} \left\| \nabla d \right\|_{L^2} \\
& + \left( 1 + |\bar{d}| + \left\| d - \bar{d} \right\|_{H^s} \right) \left\| \nabla d \right\|_{H^s}^3 + \left\| \nabla d \right\|_{H^s}^4
 \,.
\end{aligned}
\end{equation}
Combining \eqref{d7} and \eqref{d12}, we obtain
\begin{equation}\label{d13-1}
\begin{aligned}
&\frac{1}{2} \frac{\d}{\d t} \| d - \bar{d} \|_{H^s}^2 + 2A \lambda \| \nabla d \|_{H^s}^2 \\ 
& \lesssim \left( |\bar{d}| + \| d - \bar{d} \|_{H^s} \right) \| \nabla u \|_{H^s} \| \nabla d \|_{H^s} \\
& \quad+ \left( 1 + \| d - \bar{d} \|_{H^s} + |\bar{d}| + \| \nabla d \|_{H^s} \right) \left( |\bar{d}| + \| d - \bar{d} \|_{H^s} + \| \nabla d \|_{H^s} \right) \| \nabla d \|_{H^s}^2 \,.
\end{aligned}
\end{equation}
Recalling the definitions of $\mathbf{E}_s(t)$ and $\mathbf{D}_s(t)$ in \eqref{close-eq5-1} and \eqref{close-eq6-1}, we conclude estimate \eqref{d13} from the bound \eqref{d13-1}. This completes the proof of the lemma.
\end{proof}

Now we bound the gradient $\nabla d$ in the $H^s$-space. More precisely, we establish the following lemma.

\begin{lemma}\label{Lmm-G3}
    Let $s \ge 3$ be an integer. Then one has
\begin{equation}\label{nd4}
\begin{aligned}
\frac{1}{2} \frac{\d}{\d t} \left\| \nabla d \right\|_{H^s}^2 + 2A \lambda \left\| \Delta d \right\|_{H^s}^2 \lesssim ( 1 + \mathbf{E}_s^\frac{1}{2} (t) ) \mathbf{E}_s^\frac{1}{2} (t) \mathbf{D}_s (t) \,,
\end{aligned}
\end{equation}
where the functionals $\mathbf{E}_s (t)$ and $\mathbf{D}_s (t)$ are defined  in \eqref{close-eq5-1} and \eqref{close-eq6-1}, respectively.
\end{lemma}

\begin{proof}

We now apply the derivative operator $\partial^m$ ($|m|\le s$) to the fourth equation of \eqref{CMEH-eqGd} and take the $L^2$ inner product of the resulting equation with $\Delta\partial^m d$. This yields
\begin{equation}\label{d14}
\begin{aligned}
\frac{1}{2} \frac{\d}{\d t} \Vert \nabla \partial^m d \Vert_{L^2}^2 + 2A \lambda \Vert \Delta \partial^m d \Vert_{L^2}^2 & = \left\langle \partial^m (u \cdot \nabla d), \Delta \partial^m d \right\rangle - 2A \lambda \left\langle \partial^m \left[ |\nabla d|^2 (d + M e) \right], \Delta \partial^m d \right\rangle \\
& \quad + 2A \gamma \left\langle \partial^m \left[ (d + M e) \times \Delta d \right], \Delta \partial^m d \right\rangle \,.
\end{aligned}
\end{equation}
For the first term on the right-hand side of \eqref{d14}, we have
\begin{equation*}
\begin{aligned}
\left\langle \partial^m (u \cdot \nabla d), \Delta \partial^m d \right\rangle =\underset{III_1}{\underbrace{ \left\langle u \cdot \nabla \partial^m d, \Delta \partial^m d \right\rangle}} +\underset{III_2}{\underbrace{ \sum_{0 \neq m' \leq m} C_m^{m'} \left\langle \partial^{m^\prime} u \cdot \nabla \partial^{m-m'} d, \Delta \partial^m d \right\rangle}} \,.
\end{aligned}
\end{equation*}
The term $III_1$ can be estimated as follows:
\begin{equation*}
\begin{aligned}
III_1 & = -\left\langle u \cdot \nabla \nabla \partial^m d,\, \nabla \partial^m d \right\rangle - \left\langle \nabla u \cdot \nabla \partial^m d,\, \nabla \partial^m d \right\rangle \\
&= \frac{1}{2}\left\langle \nabla \cdot u,\, |\nabla \partial^m d|^2 \right\rangle - \left\langle \nabla u \cdot \nabla \partial^m d,\, \nabla \partial^m d \right\rangle \\
&\lesssim \| \nabla u \|_{L^\infty} \| \nabla \partial^m d \|_{L^2} \lesssim \| \nabla u \|_{H^s} \| \nabla \partial^m d \|_{L^2}^2. 
\end{aligned}
\end{equation*}
The term $III_2$ satisfies the bound
\begin{equation*}
\begin{aligned}
 III_2&\lesssim \sum_{0 \neq m' \leq m} \| \partial^{m'} u \|_{L^4} \| \nabla \partial^{m - m'} d \|_{L^4} \| \nabla \partial^m d \|_{L^2} \lesssim \| \nabla u \|_{H^s} \| \nabla d \|_{H^s} \| \nabla \partial^m d \|_{L^2} \,.
\end{aligned}
\end{equation*}
Consequently,
\begin{equation}\label{nd1}
\begin{aligned}
\left\langle \partial^m (u \cdot \nabla d),\, \nabla \partial^m d \right\rangle &\lesssim \| \nabla \partial^m d \|_{L^2}^2 \| \nabla u \|_{H^s} + \| \nabla d \|_{H^s} \| \nabla u \|_{H^s} \| \Delta \partial^m d \|_{L^2}\,.
\end{aligned}
\end{equation}

The second term on the right-hand side of \eqref{d14} can be controlled by
\begin{equation}\label{nd2}
\begin{aligned}
&- 2A\lambda \left\langle \partial^m \left[ |\nabla d|^2 (d +M_e) \right],\, \Delta\partial^m d \right\rangle \\
& = -2A\lambda  \left\langle \partial^m (|\nabla d|^2) (d +M_e),\, \Delta \partial^m d \right\rangle \\
&- 2A\lambda  \mathbf{1}_{|m|\geq 1}\sum_{0 \neq m' \leq m} C^{m^\prime}_m\left\langle \partial^{m - m'} (|\nabla d|^2) \partial^{m'} d,\, \Delta \partial^m d \right\rangle \\
&\lesssim \left[ 1 + |\bar{d}| + \| d - \bar{d} \|_{L^\infty} \right] \| \partial^m |\nabla d|^2 \|_{L^2} \| \Delta \partial^m d \|_{L^2} \\
&\quad + \sum_{0 \neq m' \leq m} \| \partial^{m - m'} |\nabla d|^2 \|_{L^4} \| \partial^{m'} d \|_{L^4} \| \Delta\partial^m d \|_{L^2}\\
&\lesssim \left( 1 + |\bar{d}| + \| d - \bar{d} \|_{H^s} \right) \| \nabla d \|_{H^s}^2 \| \Delta \partial^m d \|_{L^2} + \| \nabla d \|_{H^s}^3 \| \Delta\partial^m d \|_{L^2}. \\
&\lesssim \left( 1 + |\bar{d}| + \| d - \bar{d} \|_{H^s} + \| \nabla d \|_{H^s} \right) \| \nabla d \|_{H^s}^2 \| \Delta \partial^m d \|_{L^2}.
\end{aligned}
\end{equation}
Moreover, using the cancellation $\langle(d+M_e)\times\Delta\p^m d ,\Delta \p^m d \rangle=0$, the last term in \eqref{d14} can be bounded by
\begin{equation}\label{nd3}
\begin{aligned}
& 2A \gamma \langle\p^m(d + M_e) \times \Delta d], \Delta \p^m d) 
= 2A \gamma \sum_{0 \neq m' \leq m} C_m^{m'} \langle \p^{m'} d \times \Delta \p^{m-m'} d, \Delta \p^m d \rangle \\
&\lesssim \sum_{0 \neq m' \leq m} \| \p^{m^\prime} d \times \Delta \p^{m-m'} d \|_{L^2} \, \| \Delta \p^m d \|_{L^2} \lesssim \| \nabla d \|_{H^s} \| \Delta d \|_{H^s} \| \Delta \p^m d \|_{L^2} \,.
\end{aligned}
\end{equation}
Substituting the estimates \eqref{nd1}, \eqref{nd2} and \eqref{nd3} into \eqref{d14} and summing over $|m|\le s$, we obtain
\begin{equation}\label{nd4-1}
\begin{aligned}
\frac{1}{2}\frac{\d}{\d t}\left\| \nabla d \right\|_{H^s}^2 + 2A\lambda \left\| \Delta d \right\|_{H^s}^2 \lesssim \left\| \nabla d \right\|_{H^s}^2 \left\| \nabla u \right\|_{H^s} + \left\| \nabla d \right\|_{H^s} \left\| \nabla u \right\|_{H^s} \left\| \Delta d \right\|_{H^s} \\
+ \left( 1 + |\bar{d}| + \left\| d - \bar{d} \right\|_{H^s} + \left\| \nabla d \right\|_{H^s} \right) \left\| \nabla d \right\|_{H^s}^2 \left\| \Delta d \right\|_{H^s} + \left\| \nabla d \right\|_{H^s} \left\| \Delta d \right\|_{H^s}^2\,.
\end{aligned}
\end{equation}
Recalling the definitions of $\mathbf{E}_s(t)$ and $\mathbf{D}_s(t)$ in \eqref{close-eq5-1} and \eqref{close-eq6-1}, we deduce estimate \eqref{nd4} from \eqref{nd4-1}. This completes the proof of the lemma.
\end{proof}

\subsection{Estimates for the $\theta$-equation in \eqref{CMEH-eqGd}}

The main goal of this subsection is to derive the estimate for $\theta$ in the weighted space $H^s_{\w (\theta)}$, where the weighted function $\w (\theta)$ is defined in \eqref{w-weight}. More precisely, we establish the following lemma.

\begin{lemma}\label{Lmm-G4}
    Let $s \ge 3$ be an integer. Then one has
\begin{equation}\label{nd16}
\begin{aligned}
\frac{a \gamma}{2} \frac{\d}{\d t} \|\theta\|^2_{H^s_{\w (\theta)}} - \sum_{|m| \leq s} \langle P^\prime(1+\theta) \nabla \partial^m \theta, \partial^m u \rangle \lesssim \mathbf{E}_s^\frac{1}{2} (t) \mathbf{D}_s (t) \,,
\end{aligned}
\end{equation}
where the functionals $\mathbf{E}_s (t)$ and $\mathbf{D}_s (t)$ are defined  in \eqref{close-eq5-1} and \eqref{close-eq6-1}, respectively.
\end{lemma}

\begin{proof}

Taking the $L^2$ inner product of the first equation in \eqref{CMEH-eqGd} with $\frac{P^\prime(1+\theta)}{1+\theta}\theta$ yields
\begin{equation*}
\begin{aligned}
\langle\p_t \theta,\frac{P^\prime(1+\theta)}{1+\theta}\theta \rangle +\langle u\cdot\nabla\theta,\frac{P^\prime(1+\theta)}{1+\theta}\theta \rangle +\langle(1+\theta)\nabla\cdot u,\frac{P^\prime(1+\theta)}{1+\theta}\theta \rangle  = 0 \,.
\end{aligned}
\end{equation*}
The first term on the left-hand side can be computed as
\begin{equation*}
\begin{aligned}
\left\langle \partial_t \theta,\ \frac{P'(1+\theta)}{1+\theta} \theta \right\rangle & = \left\langle \frac{P'(1+\theta)}{1+\theta} ,\partial_t \frac{1}{2}\theta^2 \right\rangle \\
& = \frac{1}{2}\frac{\d}{\d t} \int_{\mathbb{T}^3} \frac{P'(1+\theta)}{1+\theta} |\theta|^2 \d x - \frac{1}{2} \int_{\mathbb{T}^3} \partial_t \left( \frac{P'(1+\theta)}{1+\theta} \right) \theta^2 \d x \\
& = \frac{1}{2}\frac{\d}{\d t} \|\theta\|_{L^2_{\frac{P'(1+\theta)}{1+\theta}} }- \frac{1}{2} \left\langle \partial_t \left( \frac{P'(1+\theta)}{1+\theta} \right),\ \theta^2 \right\rangle \,.
\end{aligned}
\end{equation*}
The second term on the left-hand side can be rewritten as
\begin{equation*}
\begin{aligned}
\left\langle u \cdot \nabla \theta,\ \frac{P'(1+\theta)}{1+\theta} \theta \right\rangle& = \frac{1}{2} \left\langle u \cdot \nabla \theta^2,\ \frac{P'(1+\theta)}{1+\theta} \right\rangle \\
& = -\frac{1}{2}\left\langle u \cdot \nabla,\ \theta^2 \frac{P'(1+\theta)}{1+\theta} \right\rangle - \frac{1}{2}\left\langle u \cdot \nabla \left( \frac{P'(1+\theta)}{1+\theta} \right),\ \theta^2 \right\rangle \,.
\end{aligned}
\end{equation*}
Consequently,
\begin{equation*}
\begin{aligned}
\frac{1}{2}\frac{\d}{\d t} \|\theta\|^2_{L^2_{\frac{P'(1+\theta)}{1+\theta}}} & = \frac{1}{2} \left\langle \partial_t \left( \frac{P'(1+\theta)}{1+\theta} \right) + u \cdot \nabla \left( \frac{P'(1+\theta)}{1+\theta} \right),\ \theta^2 \right\rangle \\
& \quad + \frac{1}{2}\left\langle  \nabla \cdot u,\ \theta^2 \frac{P'(1+\theta)}{1+\theta} \right\rangle - \left\langle\nabla\cdot u,\ P^\prime(1+\theta)\theta \right\rangle \\
& = \underset{IV_1}{\underbrace{\frac{1}{2} \left\langle \partial_t \left( \frac{P'(1+\theta)}{1+\theta} \right) + \nabla \cdot \left( \frac{P'(1+\theta)}{1+\theta} u \right),\ \theta^2 \right\rangle }} \ \underset{IV_2}{\underbrace{ - \left\langle  \nabla\cdot u,\ P^\prime ( 1 + \theta ) \theta \right\rangle}} \,.
\end{aligned}
\end{equation*}

The term $IV_1$ can be estimated as
\begin{equation*}
\begin{aligned}
IV_1&= \frac{1}{2}\left\langle\left( \frac{P'(1+\theta)}{1+\theta}\right)^\prime (\partial_t \theta + u \cdot \nabla \theta),\ \frac{P'(1+\theta)}{1+\theta} \nabla \cdot u,\ \theta^2 \right\rangle \\&
= \frac{1}{2}\left\langle -(1+\theta) \nabla \cdot u \left( \frac{P'(1+\theta)}{1+\theta} \right)^\prime + \frac{P'(1+\theta)}{1+\theta} \nabla \cdot u,\ \theta^2 \right\rangle \\&
= \frac{1}{2}\left\langle \left( \frac{P^\prime(1+\theta)}{1+\theta} - \left( \frac{P'(1+\theta)}{1+\theta} \right)' \right) \nabla \cdot u,\ \theta^2 \right\rangle \\&
\lesssim \| \nabla\cdot u \|_{L^2} \|\theta\|^2_{L^4} \,.
\end{aligned}
\end{equation*}
where we have used the fact $0<c_0\leq 1+\theta\leq c_1$ established in Theorem \ref{Thm1}. By the Sobolev embedding theorem $\|\theta\|_{L^4}\lesssim\|\theta\|^{\frac{1}{4}}_{L^2}\|\nabla\theta\|^{\frac{3}  {4}}_{L^2}\lesssim\|\nabla\theta\|^{\frac{1}{2}}_{L^2}\|\theta\|^{\frac{1}{2}}_{H^1}$, and thus
\begin{equation}\label{nd5}
\begin{aligned}
IV_1&\lesssim\|\theta\|_{H^1}\|\nabla u\|_{L^2}\|\nabla\theta\|_{L^2}\lesssim\|\theta\|_{H^1}(\|\nabla u\|^2_{L^2}\|\nabla\theta\|^2_{L^2})
 \,.
\end{aligned}
\end{equation}
Furthermore, the term $IV_2$ can be bounded by
\begin{equation}\label{nd52}
\begin{aligned}
IV_2&= -\langle P'(1+\theta)\nabla\cdot u, \theta\rangle - \langle P'(1+\theta)\nabla\theta, u\rangle + \langle P'(1+\theta)\nabla\theta, u\rangle \\
&= \langle P''(1+\theta)\theta\nabla\theta, u\rangle + \langle P'(1+\theta)\nabla\theta, u\rangle \\
&\lesssim \|\theta\|_{L^4} \|\nabla\theta\|_{L^2} \|u\|_{L^4} + \langle P'(1+\theta)\nabla\theta, u\rangle \\
&\lesssim \|\theta\|_{L^2}^{\frac{1}{4}} \|\nabla\theta\|_{L^2}^{\frac{3}{4}} \|\nabla\theta\|_{L^2} \|u\|_{L^2}^{\frac{1}{4}} \|\nabla u\|_{L^2}^{\frac{3}{4}} + \langle P'(1+\theta)\nabla\theta, u\rangle  \\
&\lesssim(\|\theta\|_{H^1}+ \|u\|_{H^1}) \left(\|\nabla u\|^2_{L^2} + \|\nabla\theta\|^2_{L^2}\right) + \langle P'(1+\theta)\nabla\theta, u\rangle \,.
\end{aligned}
\end{equation}
Consequently, one obtains
\begin{equation}\label{nd6}
\begin{aligned}
\frac{1}{2} \frac{\d}{\d t} \|\theta\|^2_{L^2_{\frac{P'(1+\theta)}{1+\theta}}} - \langle P'(1+\theta) \nabla \theta, u \rangle \lesssim (\|\theta\|_{H^1} + \|u\|_{H^1})(\|\nabla u\|^2_{L^2} + \|\nabla \theta\|^2_{L^2}) \,.
\end{aligned}
\end{equation}

Next we deal with the higher order derivatives. Applying the derivative operator $\partial^m$ ($1\le |m|\le s$) to the first equation of \eqref{CMEH-eqGd} and taking the $L^2$ inner product of the resulting equation with $\frac{P'(1+\theta)}{1+\theta}\partial^m\theta$, we obtain
\begin{equation}\label{nd7}
\begin{aligned}
&\left\langle \partial_t \partial^m \theta, \frac{P'(1+\theta)}{1+\theta} \partial^m \theta \right\rangle + \left\langle \partial^m (u \cdot \nabla \theta), \frac{P'(1+\theta)}{1+\theta} \partial^m \theta \right\rangle\\&\quad+ \left\langle \partial^m [(1+\theta) \nabla \cdot u], \frac{P'(1+\theta)}{1+\theta} \partial^m \theta \right\rangle = 0 
 \,.
\end{aligned}
\end{equation}
It is easy to see that
\begin{equation}\label{nd8}
\begin{aligned}
\left\langle \partial_t \partial^m \theta, \frac{P'(1+\theta)}{1+\theta} \partial^m \theta \right\rangle& = \left\langle \frac{P'(1+\theta)}{1+\theta}, \partial_t \frac{1}{2} |\partial^m \theta|^2 \right\rangle\\& = \frac{1}{2} \frac{\d}{\d t} \|\partial^m \theta\|^2_{L^2_{\frac{P'(1+\theta)}{1+\theta}}}  - \frac{1}{2} \left\langle \partial_t \left(\frac{P'(1+\theta)}{1+\theta}\right), |\partial^m \theta|^2 \right\rangle
 \,,
\end{aligned}
\end{equation}
and
\begin{equation}\label{nd9}
\begin{aligned}
 & \quad \left\langle \partial^m (u \cdot \nabla \theta),\ \frac{P'(1+\theta)}{1+\theta} \partial^m \theta \right\rangle \\
 & =  \left\langle u \cdot \nabla \partial^m \theta,\ \frac{P'(1+\theta)}{1+\theta} \partial^m \theta \right\rangle +\sum_{\substack{0 \neq m' \leq m}} C^{m'}_m \left\langle \partial^{m'} u \nabla \partial^{m-m'} \theta,\ \frac{P'(1+\theta)}{1+\theta} \partial^m \theta \right\rangle \\
 & = -\frac{1}{2}\left\langle \nabla \cdot u,\ |\partial^m \theta|^2 \frac{P'(1+\theta)}{1+\theta} \right\rangle - \frac{1}{2}\left\langle u \cdot \nabla \left( \frac{P'(1+\theta)}{1+\theta} \right),\ |\partial^m \theta|^2 \right\rangle \\
 & \quad + \sum_{\substack{0 \neq m' \leq m}} C^{m'}_m \left\langle \partial^{m'} u \cdot \nabla \partial^{m-m'} \theta,\ \frac{P'(1+\theta)}{1+\theta} \partial^m \theta \right\rangle \,,
\end{aligned}
\end{equation}
and
\begin{equation}\label{nd10}
\begin{aligned}
 & \quad \langle \partial^m [(1+\theta) \nabla \cdot u], \frac{P'(1+\theta) }{1+\theta}\partial^m \theta \rangle \\
 & = \langle\nabla\cdot \p^m u,P^\prime(1+\theta)\p^m\theta\rangle +\sum_{0 \neq m' \leq m} C_m^{m'} \langle \partial^{m^\prime} \theta \nabla \cdot \partial^{m-m'} u, \frac{P'(1+\theta)} {1+\theta}\partial^m \theta \rangle \\
& = -\langle P'(1+\theta) \nabla \partial^m \theta, \partial^m u \rangle + \langle \nabla \cdot \partial^m u, P'(1+\theta) \partial^m \theta \rangle \\
&\quad + \langle P'(1+\theta) \nabla \partial^m \theta, \partial^m u \rangle\\&\quad + \sum_{0 \neq m' \leq m} C_m^{m'} \langle \partial^{m^\prime} \theta \nabla \cdot \partial^{m-m'} u,\frac{ P'(1+\theta)}{1+\theta}\partial^m \theta \rangle \,.
\end{aligned}
\end{equation}
Substituting \eqref{nd8}, \eqref{nd9} and \eqref{nd10} into \eqref{nd7} yields
\begin{equation}\label{nd11}
\begin{aligned}
&\frac{1}{2} \frac{\d}{\d t} \|\partial^m \theta\|^2_{L^2_{\frac{P'(1+\theta)}{1+\theta}}} - \langle P'(1+\theta) \nabla \partial^m \theta, \partial^m u \rangle \\&= \underset{VI_1}{\underbrace{\frac{1}{2} \left\langle \partial_t (\frac{P'(1+\theta)}{1+\theta}) + \nabla\cdot \left( \frac{P'(1+\theta)}{1+\theta} \right) u, |\partial^m \theta|^2 \right\rangle}} \\
&\quad \underset{VI_2}{\underbrace{- \langle \nabla \cdot \partial^m u, P'(1+\theta) \partial^m \theta \rangle  - \langle P'(1+\theta) \nabla \partial^m \theta, \partial^m u \rangle}}\\
&\quad \underset{VI_3}{\underbrace{- \sum_{0 \neq m' \leq m} C_m^{m'} \langle \partial^{m^\prime} u \cdot \nabla \partial^{m-m'} \theta, \frac{P'(1+\theta)}{1+\theta} \partial^m \theta \rangle}} \\
&\quad \underset{VI_4}{\underbrace{- \sum_{0 \neq m' \leq m} C_m^{m'} \langle \partial^{m^\prime} \theta \nabla \cdot \partial^{m-m'} u,\frac{P'(1+\theta)}{1+\theta} \partial^m \theta \rangle}} \,,
\end{aligned}
\end{equation}

By the same arguments as in \eqref{nd5}, we have
\begin{align}\label{nd11}
\no VI_1 & = \frac{1}{2} \left\langle \left( \frac{P^\prime(1+\theta)}{1+\theta} \right)' (\p_t+u\cdot\nabla)\theta + \nabla \cdot u \frac{P'(1+\theta)}{1+\theta}, | \partial^m \theta |^2 \right\rangle\\
& = \frac{1}{2} \left\langle \left( \frac{P'(1+\theta)}{1+\theta} - (1+\theta) \left( \frac{P'(1+\theta)}{1+\theta} \right)' \right) \nabla \cdot u, | \partial^m \theta |^2 \right\rangle\\
\no & \lesssim \| \nabla \cdot u \|_{L^\infty} \| \partial^m \theta \|^2_{L^2} \lesssim \| \nabla u \|_{H^s} \| \nabla \theta \|_{H^{s-1}} \| \partial^m \theta \|_{L^2} \,,
\end{align}
Analogous to the estimates in \eqref{nd52}, it follows that
\begin{equation}\label{nd12}
\begin{aligned}
VI_2 & =\langle \partial^m u, P''(1+\theta) \nabla \theta \partial^m \theta \rangle \lesssim \|P''(1+\theta)\nabla \theta\|_{L^\infty} \| \partial^m u \|_{L^2}  \| \partial^m \theta \|_{L^2} \\
& \lesssim\|\nabla \theta\|_{H^{s-1}}\|\nabla u\|_{H^s}\|\p^m \theta\|_{L^2} \,,
\end{aligned}
\end{equation}
Using the fact $|\frac{P^\prime(1+\theta)}{1+\theta}|\lesssim1$, we obtain
\begin{equation}\label{nd13}
\begin{aligned}
VI_3& = \sum_{\substack{0 \neq m' \leq m}} C^{m'}_m \langle|\p^{m^{\prime}}u|\cdot|\nabla\p^{m-m^\prime}\theta|,|\p^m\theta|\rangle\\&
\lesssim \| \partial^m u \|_{L^2} \cdot \| \nabla \theta \|_{L^4} \cdot \| \partial^m \theta \|_{L^2}
+ \sum_{0 \neq m' < m} \| \partial^{m'} u \|_{L^4} \cdot \| \nabla \partial^{m-m'} \theta \|_{L^4} \cdot \| \partial^m \theta \|_{L^2}\\&
\lesssim \| \nabla u \|_{H^s} \cdot \| \nabla \theta \|_{H^{s-1}} \cdot \| \partial^m \theta \|_{L^2}
 \,,
\end{aligned}
\end{equation}
Applying similar estimates as in \eqref{nd12}, we have
\begin{equation}\label{nd14}
\begin{aligned}
VI_4& \lesssim \|\partial^m\theta\|_{L^2} \|\nabla u\|_{L^\infty} \|\partial^m\theta\|_{L^2} + \sum_{0 \neq m' < m} \|\partial^{m'}\theta\|_{L^4} \|\nabla \cdot \partial^{m - m'} u\|_{L^4} \|\partial^m\theta\|_{L^2} \\&
\lesssim \|\nabla u\|_{H^s} \|\nabla \theta\|_{H^{s-1}} \|\partial^m\theta\|_{L^2} \,,
\end{aligned}
\end{equation}
Therefore, inserting the estimates \eqref{nd11}, \eqref{nd12}, \eqref{nd13} and \eqref{nd14} into \eqref{nd10} yields, for all $1\le |m|\le s$,
\begin{equation}\label{nd15}
\begin{aligned}
\frac{1}{2} \frac{\d}{\d t} \|\partial^m \theta\|^2_{L^2_{\frac{P'(1+\theta)}{1+\theta}}}- \langle P'(1+\theta) \nabla \partial^m \theta, \partial^m u \rangle 
\lesssim \|\nabla u\|_{H^s} \|\nabla \theta\|_{H^{s-1}} \|\partial^m \theta\|_{L^2} \,,
\end{aligned}
\end{equation}
Combining \eqref{nd6} and \eqref{nd15}, we obtain
\begin{equation}\label{nd16-1}
\begin{aligned}
\frac{1}{2} \frac{\d}{\d t} \|\theta\|^2_{H^s_{\frac{P'(1+\theta)}{1+\theta}}} - \sum_{|m| \leq s} \langle P^\prime(1+\theta) \nabla \partial^m \theta, \partial^m u \rangle \\
\lesssim \left( \|\theta\|_{H^s} + \|u\|_{H^s} \right) \left( \|\nabla u\|_{H^s}^2 + \|\nabla \theta\|_{H^{s-1}}^2 \right) \,.
\end{aligned}
\end{equation}
Notice that $a \gamma \w (\theta) = \frac{P'(1+\theta)}{1+\theta}$. Recalling the definitions of $\mathbf{E}_s (t)$ and $\mathbf{D}_s (t)$ in \eqref{close-eq5-1} and \eqref{close-eq6-1}, we deduce estimate \eqref{nd16} from \eqref{nd16-1}. This completes the proof of the lemma.
\end{proof}

\subsection{Estimates for the $u$-equation in \eqref{CMEH-eqGd}}

In this subsection, we aim to establish an estimate for the velocity $u$ in the $H^s_{1+\theta}$ norm. More precisely, we have the following lemma.

\begin{lemma}\label{Lmm-G5}
    Let $s \ge 3$ be an integer. Then one has
\begin{equation}\label{u-eq-estimate}
\begin{aligned}
& \frac{1}{2} \frac{\d}{\d t} \| u \|_{H_{1+\theta}^s}^2 + \mu \|\nabla u\|_{H^s}^2 + (\mu + \xi) \|\nabla \cdot u\|_{H^s}^2 \\
& \quad + \sum_{|m| \leq s} \left\langle P^\prime(1+\theta) \nabla \partial^m \theta, \partial^m u \right\rangle - \sum_{|m| \leq s} \left\langle \nabla \partial^m \psi, \nabla \partial^m u \right\rangle \\
& \lesssim ( 1 + \mathbf{E}_s^\frac{s}{2} (t) ) \mathbf{E}_s^\frac{1}{2} (t) \mathbf{D}_s (t) \,,
\end{aligned}
\end{equation}
where the functionals $\mathbf{E}_s (t)$ and $\mathbf{D}_s (t)$ are defined in \eqref{close-eq5-1} and \eqref{close-eq6-1}, respectively.
\end{lemma}

\begin{proof}

We first rewrite the second $u$-equation in \eqref{CMEH-eqGd} as follows:
\begin{equation}\label{u-eq0}
\begin{aligned}
& \partial_t u + u \cdot \nabla u + \frac{1}{1+\theta} \nabla \left( P(1+\theta) - A |\nabla d|^2 + \tilde{g}(U) \right) \\
& = \frac{\mu}{1+\theta} \Delta u + \frac{\mu + \xi}{1+\theta} \nabla (\nabla \cdot u) - \Delta \psi - \frac{2}{1+\theta} \nabla \psi \cdot \nabla \theta \\
& \quad - \frac{\theta}{1+\theta} \nabla \left( \theta + \tilde{g}(U) \right) - \frac{2A}{1+\theta} \nabla \cdot \left( \nabla d \odot \nabla d \right) \,.
\end{aligned}
\end{equation}
Applying the derivative operator $\p^m \ ( |m|\leq s)$ to \eqref{u-eq0} and taking the $L^2$ inner product with $\nabla\p^m \theta$, we obtain
\begin{equation}\label{u-eq1}
\begin{aligned}
&\frac{1}{2} \frac{\d}{\d t} \left\| \partial^m u \right\|_{L^2_{1+\theta}}^2 + \underset{A_1}{\underbrace{\left\langle \partial^m (u \cdot \nabla u),\ (1+\theta) \partial^m u \right\rangle}} \\
& \quad + \underset{A_2}{\underbrace{\left\langle \partial^m \left[ \frac{\nabla \mathcal{P}(1+\theta)}{1+\theta} \right],\ (1+\theta) \partial^m u \right\rangle}}
+\underset{A_3}{\underbrace{ \left\langle \partial^m \left[ \frac{\tilde{g}(U) - A|\nabla d|^2}{1+\theta} \right],\ (1+\theta) \partial^m u \right\rangle}}\\
& = \underset{A_4}{\underbrace{\left\langle \partial^m \left[ \frac{\mu}{1+\theta} \Delta u \right],\ (1+\theta) \partial^m u \right\rangle}}
+ \underset{A_5}{ \underbrace{ \left\langle \partial^m \left[ \frac{\mu+\xi}{1+\theta} \nabla (\nabla \cdot u) \right],\ (1+\theta) \partial^m u \right\rangle}} \\
& \quad \underset{A_6}{\underbrace{- \left\langle \partial^m \left[ \frac{2}{1+\theta} \nabla \psi \cdot \nabla \theta \right],\ (1+\theta) \partial^m u \right\rangle}} \underset{A_7}{\underbrace{- \left\langle \partial^m \left[ \frac{\theta}{1+\theta} \nabla \left(\theta+\tilde{g}(U) \right) \right],\ (1+\theta) \partial^m u \right\rangle}} \\
& \quad \underset{A_8}{\underbrace{- \left\langle \partial^m \left[ \frac{2A}{1+\theta} \nabla \cdot (\nabla d \odot \nabla d) \right],\ (1+\theta) \partial^m u \right\rangle}}
\underset{A_9}{\underbrace{- \left\langle \Delta \partial^m \psi,\ (1+\theta) \partial^m u \right\rangle}} \,.
\end{aligned}
\end{equation}
The term $A_1$ can be rewritten as
\begin{equation*}
\begin{aligned}
-A_1&=-\left\langle \partial^m(u\cdot\nabla u),\ (1+\theta)\partial^m u\right\rangle \\
&= -\left\langle u\cdot\nabla\partial^m u,\ (1+\theta)\partial^m u\right\rangle -\sum_{\substack{0\neq m'\leq m}} C_m^{m'}\left\langle \partial^{m'}u\cdot\nabla\partial^{m-m'}u,\ (1+\theta)\partial^m u\right\rangle \\
&= -\frac{1}{2}\left\langle u\cdot\nabla|\partial^m u|^2,\ 1+\theta\right\rangle -\sum_{\substack{0\neq m'\leq m}} C_m^{m'}\left\langle \partial^{m'}u\cdot\nabla\partial^{m-m'}u,\ (1+\theta)\partial^m u\right\rangle \\
&= \underset{A_{11}}{\underbrace{\frac{1}{2}\left\langle \nabla\cdot u\ |\partial^m u|^2,\ 1+\theta\right\rangle }}+\underset{A_{12}}{\underbrace{\frac{1}{2}\left\langle u\cdot\nabla\theta,\ |\partial^m u|^2\right\rangle}} \\
&\quad \underset{A_{13}}{\underbrace{-\sum_{\substack{0\neq m'\leq m}} C_m^{m'}\left\langle \partial^{m'}u\cdot\nabla\partial^{m-m'}u,\ (1+\theta)\partial^m u\right\rangle}} \,.
\end{aligned}
\end{equation*}
For $A_{11}$, we have
\begin{equation*}
\begin{aligned}
{A_{11}} &\lesssim\|\nabla u\|_{L^3}\|\p^m u\|^2_{L^3} \lesssim\|\nabla u\|^{\frac{1}{2}}_{L^2}\|\nabla^2 u\|^{\frac{1}{2}}_{L^2}\|\p^m u\|_{L^2}\|\nabla\p^m u\|_{L^2} \lesssim\|\nabla u\|^2_{H^s}\|\p^m u\|_{L^2} \,,
\end{aligned}
\end{equation*}
where we used the Sobolev embedding $ \|u\|_{L^3} \lesssim \| u \|^{ \frac{1}{2} }_{L^2} \| \nabla u \|^{ \frac{1}{2} }_{L^2}$. Similarly, $A_{12}$ can be bounded by
\begin{equation*}
\begin{aligned}
{A_{12}}&\lesssim\|u\|_{L^4} \|\nabla\theta\|_{L^4} \|\partial^m u\|_{L^4}^2 \\
&\lesssim \|u\|_{L^2}^{\frac{1}{4}} \|\nabla u\|_{L^2}^{\frac{3}{4}} \|\nabla\theta\|_{L^2}^{\frac{1}{4}} \|\nabla^2\theta\|_{L^2}^{\frac{3}{4}} \|\partial^m u\|_{L^2}^{\frac{1}{2}}\|\nabla\partial^m u\|_{L^2}^{\frac{3}{2}} \\
&\lesssim(\|u\|_{H^s}^2+\|\theta\|_{H^s}^2)(\|\nabla u\|_{H^s}^2+\|\nabla \theta\|_{H^{s-1}}^2) \,,
\end{aligned}
\end{equation*}
where we used the inequality $\|u\|_{L^4}\lesssim\|u\|^{\frac{1}{4}}_{L^2}\|\nabla u\|^{\frac{3}{4}}_{L^2}$. Moreover, $A_{13}$ can be estimated analogously:
\begin{equation*}
\begin{aligned}
{A_{13}} \lesssim\sum_{0\neq m^\prime \leq m } \langle | \p^{m^\prime}u||\nabla\p^{m-m^\prime}u|,(1+\theta)|\p^m u|    \rangle \lesssim\|\p^m u\|_{L^2}\|\nabla u\|^2_{H^s} \,.
\end{aligned}
\end{equation*}
Consequently, $A_1$ satisfies
\begin{equation}\label{u-eq2}
\begin{aligned}
-A_1&=-\left\langle \partial^m(u\cdot\nabla u),\ (1+\theta)\partial^m u\right\rangle \lesssim(\|\p^m u\|_{L^2}+\| u\|^2_{H^s}+\| \theta\|^2_{H^s})(\|\nabla u\|^2_{H^s}+\| \nabla \theta\|^2_{H^s})
 \,.
\end{aligned}
\end{equation}

For the term $A_2$, we compute 
\begin{equation*}
\begin{aligned}
-A_2&=-\left\langle \partial^m \left[ \frac{\nabla P(1+\theta)}{1+\theta} \right], (1+\theta) \partial^m u \right\rangle = -\left\langle \partial^m \left[ \frac{P'(1+\theta)}{1+\theta} \nabla \theta \right], (1+\theta) \partial^m u \right\rangle \\&
= -\left\langle P'(1+\theta) \nabla \partial^m \theta, \partial^m u \right\rangle 
 \underset{A_{21}}{\underbrace{ - \sum_{\substack{0 \neq m' \leq m}} C_m^{m'} \left\langle \partial^{m'} \left[ \frac{P'(1+\theta)}{1+\theta} \right] \nabla \partial^{m-m'} \theta, (1+\theta) \partial^m u \right\rangle}}
 \,.
\end{aligned}
\end{equation*}
By Lemma 3.2 of \cite{JLT-M3AS-2020}, we have
\begin{equation*}
\begin{aligned}
A_{21}& \lesssim \sum_{\substack{0 \neq m' \leq m}} \left\langle |\partial^{m'} \left[ \frac{P'(1+\theta)}{1+\theta} \right]|| \nabla \partial^{m-m'} \theta|, |\partial^m u |\right\rangle \\&
\lesssim \left\| \nabla \left( \frac{P'(1+\theta)}{1+\theta} \right) \right\|_{H^{s-1}} \left\| \nabla \theta \right\|_{H^{s-1}} \left\| \nabla u \right\|_{H^s} \\&
\lesssim \left( \| \nabla \theta \|_{H^{s-1}} +  \| \nabla \theta \|_{H^{s-1}}^s \right) \left\| \nabla \theta \right\|_{H^{s-1}} \left\| \nabla u \right\|_{H^s} \\&
\lesssim \left( \| \theta \|_{H^s} +  \| \theta \|_{H^s} ^s \right) \left\| \nabla \theta \right\|_{H^{s-1}} \left\| \nabla u \right\|_{H^s} \,.
\end{aligned}
\end{equation*}
Hence,
\begin{equation}\label{u-eq3}
\begin{aligned}
 -A_2 + \left\langle P'(1+\theta) \nabla \partial^m \theta, \partial^m u \right\rangle \lesssim\left( \| \theta \|_{H^s} +  \| \theta \|_{H^s} ^s \right) \left\| \nabla \theta \right\|_{H^{s-1}} \left\| \nabla u \right\|_{H^s} \,.
\end{aligned}
\end{equation}
For the term $A_3$, we note that
\begin{equation*}
\begin{aligned}
-A_3&= -\left\langle \partial^m \left[ \frac{\tilde{g}(U) - A|\nabla d|^2}{1+\theta} \right], (1+\theta)\partial^m u \right\rangle \\
&= \underbrace{ -\left\langle \partial^m \left[ \tilde{g}(U) - A|\nabla d|^2 \right], \partial^m u \right\rangle}_{A_{31}} \\
&\quad \underbrace{ -\sum_{\substack{0 \neq m' \leq m}} C_m^{m'} \left\langle \partial^{m'} \left( \frac{1}{1+\theta} \right) \partial^{m - m'} \left[ \tilde{g}(U) - A|\nabla d|^2 \right], \partial^m u \right\rangle }_{A_{32}} \,.
\end{aligned}
\end{equation*}
Since $\tilde{g}(U) = O(|U|^2)$ and $U = \nabla \psi$, we have
\begin{equation*}
\begin{aligned}
A_{31} \lesssim\left( \|\tilde{g}(U)\|_{H^s} + \| |\nabla d|^2 \|_{H^s} \right) \|\partial^m u\|_{L^2} \lesssim \left( \|\nabla \psi\|_{H^s}^2 + \| |\nabla d|^2 \|^2_{H^s} \right) \|\partial^m u\|_{L^2} \,,
\end{aligned}
\end{equation*}
and
\begin{equation*}
\begin{aligned}
A_{32}&\lesssim\sum_{\substack{0 \neq m' \leq m}} \left\langle \left| \partial^{m'} \left( \frac{1}{1+\theta} \right) \right| \left| \partial^{m} \left[ \tilde{g}(U) - A|\nabla d|^2 \right] \right|, |\partial^m u| \right\rangle\\&
\lesssim \left\| \nabla \left( \frac{1}{1+\theta} \right) \right\|_{H^{s-1}} \left( \|\tilde{g}(U)\|_{H^s} + \| |\nabla d|^2 \|_{H^s} \right) \|\nabla u\|_{H^{s-1}} \,.
\end{aligned}
\end{equation*}
Applying Lemma 3.2 of \cite{JLT-M3AS-2020} and the fact $\|\frac{1}{1+\theta}\|_{L^\infty}\lesssim 1$, we obtain 
\begin{equation*}
\begin{aligned}
\|\nabla (\frac{1}{1+\theta})\|_{H^s}\lesssim\|\nabla\theta\|_{H^{s-1}}+\|\nabla \theta\|^s
_{H^{s-1}} \,.
\end{aligned}
\end{equation*}
Then 
\begin{equation*}
\begin{aligned}
A_{32}&\lesssim(\|\nabla\theta\|_{H^{s-1}}+\|\nabla \theta\|^s_{H^{s-1}})(\|\nabla\psi\|^2_{H^s}+\|\nabla d\|^2_{H^s})\|u\|_{H^s}\\&
\lesssim(\|\theta\|_{H^{s}}+\| \theta\|^s_{H^{s}})(\|\nabla\psi\|^2_{H^s}+\|\nabla d\|^2_{H^s})\|u\|_{H^s} \,.
\end{aligned}
\end{equation*}
Therefore,
\begin{equation}\label{u-eq4}
\begin{aligned}
 - A_3 \lesssim(1+\| \theta\|^s_{H^{s}})(\|\nabla\psi\|^2_{H^s}+\|\nabla d\|^2_{H^s})\|u\|_{H^s} \,.
\end{aligned}
\end{equation}

We now estimate the term $A_4$. Observe that
\begin{equation*}
\begin{aligned}
A_4&= \left\langle \partial^m \left[ \frac{\mu}{1+\theta} \Delta u \right], (1+\theta)\partial^m u \right\rangle \\
&= \mu \left\langle \Delta \partial^m u, \partial^m u \right\rangle +\underset{A_{41}}{\underbrace{  \mu \sum_{\substack{0 \neq m' \leq m}} C_m^{m'} \left\langle \partial^{m'} \left( \frac{1}{1+\theta} \right) \Delta \partial^{m - m'} u, (1+\theta)\partial^m u \right\rangle}}
 \,.
\end{aligned}
\end{equation*}
It is easy to see that
$$ \mu \left\langle \Delta \partial^m u, \partial^m u \right\rangle = -\mu \left\| \nabla \partial^m u \right\|_{L^2}^2 \,, $$
and
\begin{equation}\label{u-eq5}
\begin{aligned}
A_{41}&\lesssim\sum_{\substack{0 \neq m' \leq m}} \left\langle \left| \partial^{m'} \left( \frac{1}{1+\theta} \right) \right| \left| \Delta \partial^{m - m'} u \right|, \left| \partial^m u \right| \right\rangle \\
&\lesssim \left\| \nabla \left( \frac{1}{1+\theta} \right) \right\|_{H^{s-1}} \left\| \nabla u \right\|_{H^s}^2 \lesssim \left( 1 + \left\| \theta \right\|_{H^s}^{s-1} \right) \left\| \theta \right\|_{H^s} \left\| \nabla u \right\|_{H^s}^2 \,.
\end{aligned}
\end{equation}
Consequently,
\begin{equation}\label{u-eq6}
\begin{aligned}
 \mu \left\| \nabla \partial^m u \right\|_{L^2}^2 + A_4 \lesssim \left( 1 + \left\| \theta \right\|_{H^s}^{s-1} \right) \left\| \theta \right\|_{H^s} \left\| \nabla u \right\|_{H^s}^2 \,.
\end{aligned}
\end{equation}

For the term $A_5$,
\begin{equation*}
\begin{aligned}
A_5&= \left\langle \partial^m \left[ \frac{\mu + \xi}{1+\theta} \nabla (\nabla \cdot u) \right], (1+\theta)\partial^m u \right\rangle\\&
=(u+\xi)\langle \nabla(\nabla\cdot \p^m u),\p^m u \rangle +\underset{A_{51}}{\underbrace{ \sum_{\substack{0 \neq m' \leq m}} C_m^{m'} \left\langle \partial^{m} \left( \frac{1}{1+\theta} \right) \nabla \left( \nabla \cdot \partial^{m - m'} u \right), (1+\theta)\partial^m u \right\rangle}}
\,.
\end{aligned}
\end{equation*}
Note that
$$(u+\xi)\langle \nabla(\nabla\cdot \p^m u),\p^m u   \rangle=-(u+\xi)\|\nabla\cdot\p^mu\|^2_{L^2}.$$
Using similar arguments as in \eqref{u-eq5}, we obtain
\begin{equation*}
\begin{aligned}
A_{51}&\lesssim(1+\|\theta\|^{s-1}_{H^s})\|\theta\|_{H^s}\|\nabla u\|^2_{H^s}
\,.
\end{aligned}
\end{equation*}
Thus, 
\begin{equation}\label{u-eq7}
\begin{aligned}
 (\mu + \xi) \left\| \nabla \cdot \partial^m u \right\|_{L^2}^2 + A_5 \lesssim \left( 1 + \left\| \theta \right\|_{H^s}^{s-1} \right) \left\| \theta \right\|_{H^s} \left\| \nabla u \right\|_{H^s}^2 \,.
\end{aligned}
\end{equation}
For the term $A_6$,
\begin{equation*}
\begin{aligned}
A_6&=- \left\langle \partial^m \left[ \frac{2}{1+\theta} \nabla \psi \cdot \nabla \theta \right],\ (1+\theta) \partial^m u \right\rangle\\&
=\underset{A_{61}}{\underbrace{-2\langle\p^m(\nabla\psi\cdot\nabla\theta),\p^m u \rangle }} \ \underset{A_{62}}{ \underbrace{ -2 \sum_{0\neq m^\prime\leq m} \langle     \p^{m^\prime}(\frac{1}{1+\theta})\p^{m-m^\prime}(\nabla \psi\cdot\nabla\theta),(1+\theta)\p^m u\rangle}}
 \,.
\end{aligned}
\end{equation*}
The term $A_{61}$ can be estimated by
\begin{equation*}
\begin{aligned}
A_{61}&\lesssim \sum_{\substack{0 \neq m' \leq m}} \left\| \partial^{m'} \nabla \psi \cdot \p^{m-m^\prime}\nabla \theta \right\|_{L^2} \left\| \partial^m u \right\|_{L^2} + \left\| \Delta \psi \right\|_{L^4} \left\| \p^m \theta \right\|_{L^2} \left\| \partial^m u \right\|_{L^4} \\
&\quad + \left\| \nabla \psi \right\|_{L^\infty} \left\| \p^m \theta \right\|_{L^2} \left\| \nabla \cdot \partial^m u \right\|_{L^2} \\
&\lesssim \left\| u \right\|_{H^s} \left\| \nabla \psi \right\|_{H^s} \left\| \nabla \theta \right\|_{H^{s-1}} + \left\| \theta \right\|_{H^s} \left\| \nabla \psi \right\|_{H^s} \left\|\nabla u \right\|_{H^s} \,,
\end{aligned}
\end{equation*}
and $A_{62}$ satisfies
\begin{equation*}
\begin{aligned}
A_{62}&\lesssim \sum_{\substack{0 \neq m' \leq m}} \left\langle \left| \partial^{m'} \left( \frac{1}{1+\theta} \right) \right| \left| \partial^{m - m'} \left( \nabla \psi \cdot \nabla \theta \right) \right|, \left| \partial^m u \right| \right\rangle \\
&\lesssim \left\| \nabla \left( \frac{1}{1+\theta} \right) \right\|_{H^{s-1}} \left\| \nabla \psi \cdot \nabla \theta \right\|_{H^{s-1}} \left\| u \right\|_{H^s} \\
&\lesssim \left( \left\| \theta \right\|_{H^s} + \left\| \theta \right\|_{H^s}^2 \right) \left\| u \right\|_{H^s} \left\| \nabla \psi \right\|_{H^s} \left\| \nabla \theta \right\|_{H^{s-1}} \,.
\end{aligned}
\end{equation*}
Hence,
\begin{equation}\label{u-eq8}
\begin{aligned}
A_6 \lesssim \left( 1 + \left\| \theta \right\|_{H^s}^2 \right) \left( \left\| u \right\|_{H^s} \left\| \nabla \psi \right\|_{H^s} \left\| \nabla \theta \right\|_{H^s} + \left\| \theta \right\|_{H^s} \left\| \nabla \psi \right\|_{H^s} \left\|\nabla u \right\|_{H^s} \right) \,.
\end{aligned}
\end{equation}

For the quantity $A_7$, we decompose it as
\begin{equation*}
\begin{aligned}
A_{7}&= -\left\langle \partial^m \left[ \frac{\theta}{1+\theta} \nabla \left( \theta + \tilde{g}(U) \right) \right], (1+\theta)\partial^m u \right\rangle \\
&=\underset{A_{71}}{\underbrace{- \left\langle \partial^m \left[ \theta \nabla \theta + \theta \nabla \tilde{g}(U) \right], \partial^m u \right\rangle}} \\
&\quad \underset{A_{72}}{\underbrace{- \sum_{\substack{0 \neq m' \leq m}} C_m^{m'} \left\langle \partial^{m'} \left( \frac{1}{1+\theta} \right) \partial^{m - m'} \left[ \theta \nabla \theta + \theta \nabla \tilde{g}(U) \right],(1+\theta) \partial^m u \right\rangle}}  \,.
\end{aligned}
\end{equation*}
The term $A_{71}$ can be further split as
\begin{equation*}
\begin{aligned}
A_{71}&=  \frac{1}{2} \left\langle \partial^m \theta^2, \nabla \cdot \partial^m u \right\rangle - \left\langle  \theta \nabla \partial^m\tilde{g}(U), \partial^m u \right\rangle - \sum_{\substack{0 \neq m' \leq m}} C_m^{m'} \left\langle \partial^{m'} \theta \nabla \partial^{m - m'} \tilde{g}(U), \partial^m u \right\rangle \\
&=\underset{A_{711}}{\underbrace{ \frac{1}{2} \left\langle \partial^m \theta^2, \nabla \cdot \partial^m u \right\rangle}} + \underset{A_{712}}{\underbrace{\left\langle \nabla \theta \partial^m \tilde{g}(U), \partial^m u \right\rangle}} \underset{A_{713}}{\underbrace{  + \left\langle \theta \partial^m \tilde{g}(U), \nabla \cdot \partial^m u \right\rangle}} \\
&\quad \underset{A_{714}}{\underbrace{ - \sum_{\substack{0 \neq m' \leq m}} C_m^{m'} \left\langle \partial^{m'} \theta \nabla \partial^{m - m'} \tilde{g}(U), \partial^m u \right\rangle}}
 \,.
\end{aligned}
\end{equation*}
The subterms $A_{71i}$ ($1 \leq i \leq 4$) can be estimated as follows:
\begin{equation*}
\begin{aligned}
A_{711}&\lesssim \left( \left\| \theta \right\|_{L^4}^2 + \left\| \nabla \theta^2  \right\|_{H^{s-1}}\right) \left\|\nabla u \right\|_{H^s} \lesssim \left( \left\| \theta \right\|^{\frac{1}{2}}_{L^2} \left\| \nabla \theta \right\|^{\frac{3}{2}}_{L^2} + \left\| \theta \right\|_{H^s} \left\| \nabla \theta \right\|_{H^s} \right) \left\|\nabla u \right\|_{H^s} \\
&\lesssim \left\| \theta \right\|_{H^s} \left\| \nabla \theta \right\|_{H^{s-1}} \left\|\nabla u \right\|_{H^s} \,,
\end{aligned}
\end{equation*}
and
\begin{equation*}
\begin{aligned}
A_{712} \lesssim\left\| \nabla \theta \right\|_{L^\infty} \left\| \partial^m \tilde{g}(U) \right\|_{L^2} \left\| \partial^m u \right\|_{L^2} \lesssim \left\| \theta \right\|_{H^s} \left\| \nabla \psi \right\|_{H^s}^2 \left\| u \right\|_{H^s} \,,
\end{aligned}
\end{equation*}
and
\begin{equation*}
\begin{aligned}
A_{713} \lesssim\left\| \theta \right\|_{L^\infty} \left\| \partial^m \tilde{g}(U) \right\|_{L^2} \left\| \nabla \cdot \partial^m u \right\|_{L^2} \lesssim \left\| \theta \right\|_{H^s} \left\| \nabla \psi \right\|_{H^s}^2 \left\| \nabla u \right\|_{H^s} \,,
\end{aligned}
\end{equation*}
and
\begin{equation*}
\begin{aligned}
A_{714}& \lesssim\sum_{\substack{0 \neq m' \leq m}} \left\langle \left| \partial^{m'} \theta \right| \left| \nabla \partial^{m - m'} \tilde{g}(U) \right|, \left| \partial^m u \right| \right\rangle \lesssim \| \theta \|_{H^s} \|\tilde{g}(U)\|_{H^s}\|\p^m u\|_{L^2} \\
& \lesssim \| \theta \|_{H^s} \| u \|_{H^s} \| \nabla \psi\|^2_{H^s} \,.
\end{aligned}
\end{equation*}
Thus,
\begin{equation*}
\begin{aligned}
A_{71}& \lesssim(1+\|\theta\|_{H^s})(\|\theta\|_{H^s}+\|\nabla \psi\|_{H^s}+\|u\|_{H^s})
\times(\|\nabla u\|^2_{H^s}+\|\nabla \psi\|^2_{H^s}+\|\nabla \theta\|^2_{H^{s - 1}})
\,.
\end{aligned}
\end{equation*}
For $A_{72}$, we have
\begin{equation*}
\begin{aligned}
A_{72}& \lesssim\sum_{0 \neq m' \leq m}\langle|\partial^{m'}(\frac{1}{1 + \theta})|,|\partial^{m - m'}[\theta\nabla \theta+\theta\nabla\tilde{g}(U)]|,|\partial^m u|\rangle\\&
\lesssim\|\nabla(\frac{1}{1 + \theta})\|_{H^{s - 1}}(\|\theta\nabla \theta\|_{H^{s - 1}}+\|\theta\nabla \tilde{g}(U)\|_{H^{s - 1}})\|\Delta u\|_{H^{s - 1}}\\&
\lesssim(\|\nabla\theta\|_{H^{s - 1}}+\|\nabla\theta\|^s_{H^{s - 1}})(\|\theta\|_{H^s}\|\nabla\theta\|_{H^{s - 1}}+\|\theta\|_{H^s}\|\tilde{g}(U)\|_{H^s})\|\nabla u\|_{H^{s-1}}\\&
\lesssim(1+\|\theta\|_{H^s}^{s - 1})(\|\theta\|^2_{H^s}+\|\theta\|_{H^s}\|\nabla\psi\|^2_{H^s})\|\nabla\theta\|_{H^{s-1}}\|\nabla u\|_{H^s}
\,.
\end{aligned}
\end{equation*}
Consequently,
\begin{equation}\label{u-eq9}
\begin{aligned}
A_7 \lesssim(1+\|\theta\|^s_{H^s}+\|\nabla\psi\|_{H^s})(\|\theta\|_{H^s}+\|\nabla\psi\|_{H^s}+\|u\|_{H^s}) (\|\nabla u\|_{H^s}^2+\|\nabla\psi\|_{H^s}^2+\|\nabla\theta\|_{H^{s - 1}}^2)
 \,.
\end{aligned}
\end{equation}

Observe that the quantity $A_8$ can be rewritten as
\begin{equation*}
\begin{aligned}
A_{8}&= -\langle\partial^m[\frac{2A}{1 + \theta}\nabla\cdot(\triangledown d\odot\triangledown d)],(1 + \theta)\partial^m u\rangle\\&
=\underset{A_{81}}{\underbrace{ -2A\langle\nabla\cdot\partial^m(\triangledown d\odot\triangledown d),\partial^m u\rangle}}\\&\quad
\underset{A_{82  }}{\underbrace{ -2A\sum_{0\neq m'\leq m}C_m^{m'}\langle\partial^{m^\prime} (\frac{1}{1 + \theta})\nabla\cdot\partial^{m - m'}(\triangledown d\odot\triangledown d),(1 + \theta)\partial^m u\rangle}} \,.
\end{aligned}
\end{equation*}
The term $A_{81}$ can be dominated by
\begin{equation*}
\begin{aligned}
A_{81}&=2A\langle\partial^m(\nabla d\odot \nabla d),\nabla \cdot\partial^m u\rangle \lesssim \bigl\|\partial^m(\nabla d \odot\nabla d)\bigr\|_{L^2}\bigl\|\nabla\cdot\partial^m u\bigr\|_{L^2} \lesssim\|\nabla d\|^2_{H^s} \|\nabla u\|_{H^s}
\,.
\end{aligned}
\end{equation*}
Moreover, we can control $A_{82}$ as follows:
\begin{equation*}
\begin{aligned}
A_{82}&\lesssim\sum_{0\neq m'\leq m} \langle|\partial^{m'}\Bigl(\frac{1}{1+\theta}\Bigr)|
|\nabla \cdot\partial^{m-m'}(\nabla d\odot\nabla d)|,
|\partial^m u|\rangle\\&
\lesssim\|\nabla\Bigl(\frac{1}{1+\theta}\Bigr)\bigr\|_{H^{s-1}}
\|\nabla d \odot\nabla d\|_{H^{s}}
\|\nabla u\|_{H^{s-1}} \lesssim(\|\nabla\theta\|_{H^{s-1}}+\|\nabla \theta\|^s_{H^{s-1}})
\|\nabla d\|^2_{H^{s}}
\|\nabla u\|_{H^s}\\
& \lesssim ( \| \theta \|_{H^{s}} + \| \theta \|^s_{H^s} )
\|\nabla d\|^2_{H^s}\|\nabla u\|_{H^s} \,.
\end{aligned}
\end{equation*}
Hence,
\begin{equation}\label{u-eq10}
\begin{aligned}
A_8 \lesssim(1+\|\theta\|^s_{H^s})
\|\nabla d\|^2_{H^s}\|\nabla u\|_{H^s} \,.
\end{aligned}
\end{equation}
Finally, the term $A_9$ can be estimated by
\begin{equation}\label{u-eq11}
\begin{aligned}
A_9&=- \langle\Delta \partial^m \psi, (1+\theta)\partial^m u \rangle\\& = \langle \nabla\partial^m \psi,  \nabla\theta\p^m u+ (1+\theta)\nabla\partial^m u \rangle\\&
\leq c\|\nabla \partial^m \psi\|_{L^2} \left( \|\nabla \theta\|_{H^{s-1}}\|u\|_{H^s} + \|\theta\|_{H^s}\|\nabla u\|_{H^s} \right)
+ \langle \nabla \partial^m \psi, \nabla \partial^m u \rangle
\end{aligned}
\end{equation}
for some constant $c > 0$. Substituting the estimates \eqref{u-eq3}, \eqref{u-eq4}, \eqref{u-eq6}, \eqref{u-eq7}, \eqref{u-eq8}, \eqref{u-eq9}, \eqref{u-eq10} and \eqref{u-eq11} into \eqref{u-eq1}, and summing over $|m|\le s$ in the resulting inequality, we obtain
\begin{equation}\label{u-eq-estimate-1}
\begin{aligned}
&\frac{1}{2} \frac{\d}{\d t} \|u\|_{H_{1+\theta}^s}^2 + \mu \|\nabla u\|_{H^s}^2 + (\mu + \xi) \|\nabla \cdot u\|_{H^s}^2\\&\quad
+ \sum_{|m| \leq s} \left\langle P^\prime(1+\theta) \nabla \partial^m \theta, \partial^m u \right\rangle - \sum_{|m| \leq s} \left\langle \nabla \partial^m \psi, \nabla \partial^m u \right\rangle \\&
\lesssim \left( 1 + \|\theta\|^s_{H^s} + \|u\|^s_{H^s} + \|\nabla\psi\|^s_{H^s} \right) \left( \| u\|_{H^s} + \| \theta\|_{H^s} + \|\nabla\psi\|_{H^s} + \|\nabla d\|_{H^s} \right) \\&
\quad\times \left( \|\nabla u\|_{H^s}^2 + \|\nabla \theta\|_{H^{s-1}}^2 + \|\nabla\psi\|_{H^s}^2 + \|\nabla d\|_{H^s}^2 \right)
 \,.
\end{aligned}
\end{equation}
Together with the definitions of $\mathbf{E}_s(t)$ and $\mathbf{D}_s(t)$ in \eqref{close-eq5-1} and \eqref{close-eq6-1}, we deduce estimate \eqref{u-eq-estimate} from \eqref{u-eq-estimate-1}. This completes the proof of the lemma.
\end{proof}

\subsection{$\theta$-dissipative mechanism}

In this subsection, we mainly seek the damping effect of the density $\theta$ arising from the $u$-equation in \eqref{CMEH-eqGd}. More precisely, we establish the following lemma.

\begin{lemma}\label{Lmm-G6}
    Let $s \ge 3$ be an integer. Then one has
\begin{equation}\label{theta-eq-estimate}
\begin{aligned}
&\frac{1}{2} \frac{\d}{\d t} \left( \|u + \nabla \theta\|_{H^{s-1}}^2 - \|u\|_{H^{s-1}}^2 - \|\nabla \theta\|_{H^{s-1}}^2 \right) + a \gamma \|\nabla \theta\|^2_{H^{s-1}_{\w (\theta)}} \\
& \quad - c_0 \big[ \mu \| \nabla u \|_{H^s} \, \| \nabla \theta \|_{H^{s-1}_{\w (\theta)}} - (\mu + \xi ) \| \nabla \cdot u \|_{H^s} \, \| \nabla \theta \|_{H^{s-1}_{\w (\theta)}} \big] - c_1 \|\nabla \psi\|_{H^s} \|\nabla \theta\|_{H^{s-1}_{\w (\theta)}} \\
& \lesssim \left( 1 + \mathbf{E}_s (t) \right) \mathbf{E}_s^\frac{1}{2} (t) \mathbf{D}_s (t) \,,
\end{aligned}
\end{equation}
where the functionals $\mathbf{E}_s (t)$ and $\mathbf{D}_s (t)$ are defined  in \eqref{close-eq5-1} and \eqref{close-eq6-1}, respectively, the weighted function $\w (\theta)$ is given in \eqref{w-weight}, and the constants are $c_0 = \sup_{t,x} (1 + \theta)^{- \frac{\gamma}{2}} > 0$, $c_1 = \sup_{t,x} (1 + \theta)^\frac{2 - \gamma}{2} > 0 $.
\end{lemma}

\begin{proof}

Applying the derivative operator $\partial^m$ ($|m|\le s-1$) to equation \eqref{u-eq0} and taking the $L^2$ inner product of the resulting equation with $\nabla\partial^m\theta$, we obtain
\begin{equation}\label{theta-eq1} 
\begin{aligned}
&\underset{B_1}{\underbrace{\left\langle \partial_t \partial^m u, \nabla \partial^m \theta \right\rangle}}
+ \underset{B_2}{\underbrace{\left\langle \partial^m (u \cdot \nabla u), \nabla \partial^m \theta \right\rangle}}
+\underset{B_3}{\underbrace{ \left\langle \partial^m \left[ \frac{\mathcal{P}'(1+\theta)}{1+\theta} \nabla \theta \right], \nabla \partial^m \theta \right\rangle}}\\&
\quad\underset{B_4}{\underbrace{+ \left\langle \partial^m \left[ \frac{1}{1+\theta} \left( -A|\nabla d|^2 + \tilde{g}(U) \right) \right], \nabla \partial^m \theta \right\rangle}}\\&
= \underset{B_5}{\underbrace{\left\langle \partial^m \left[ \frac{\mu}{1+\theta} \Delta u \right], \nabla \partial^m \theta \right\rangle }}+ \underset{B_6}{\underbrace{\left\langle \partial^m \left[ \frac{\mu+\xi}{1+\theta} \nabla (\nabla \cdot u) \right], \nabla \partial^m \theta \right\rangle}}\\&
\quad\underset{B_7}{\underbrace{- \left\langle \Delta \partial^m \psi, \nabla \partial^m \theta \right\rangle}} - \underset{B_8}{\underbrace{\left\langle \partial^m \left[ \frac{2}{1+\theta} \nabla \psi \cdot \nabla \theta \right], \nabla \partial^m \theta \right\rangle}}\\&
\quad\underset{B_9}{\underbrace{- \left\langle \partial^m \left[ \frac{\theta}{1+\theta} \nabla \left( (1+\theta) \tilde{g}(U) \right) \right], \nabla \partial^m \theta \right\rangle}} \ \underset{B_{10}}{\underbrace{- \left\langle \partial^m \left[ \frac{2A}{1+\theta} \nabla \cdot (\nabla d \odot \nabla d) \right], \nabla \partial^m \theta \right\rangle}} \,.
\end{aligned}
\end{equation}
Using the first equation in \eqref{CMEH-eqGd}, we have
\begin{align*}
B_1 & = \frac{\d}{\d t} \langle \partial^m u, \nabla \partial^m \theta \rangle - \langle \partial^{m} u, \nabla \p_t\partial^m \theta \rangle \\
& = \underset{B_{11}}{ \underbrace{ \frac{\d}{\d t} \langle \partial^m u, \nabla \partial^m \theta \rangle } } + \underset{ B_{12} }{ \underbrace{ \langle \partial^m u, \nabla \partial^m (u \cdot \nabla \theta) \rangle } } + \underset{ B_{13} }{ \underbrace{ \langle \partial^m u, \nabla \partial^m [ ( 1 + \theta ) \nabla \cdot u ] \rangle } } \,.
\end{align*}
Observe that
\begin{equation*}
\begin{aligned}
B_{11}=\frac{1}{2} \frac{\d}{\d t} \left( \| \p^m u + \nabla \partial^m \theta \|_{L^2}^2 - \| \p^m u \|_{L^2}^2 - \| \nabla \partial^m \theta \|_{L^2}^2 \right) \,,
\end{aligned}
\end{equation*}
and
\begin{equation*}
\begin{aligned}
B_{12}&= \underset{B_{121}}{\underbrace{\langle \p^m u, u \cdot \nabla (\nabla \partial^m \theta) \rangle}} +  \underset{B_{122}}{\underbrace{\sum_{0 \neq m' \leq m} C_m^{m'} \langle \p^m u, \partial^{m'} u \cdot \nabla (\nabla \partial^{m - m'} \theta) \rangle }} +  \underset{B_{123}}{\underbrace{\langle \p^m u, \partial^m (\nabla u \cdot \nabla \theta) \rangle}} \,.
\end{aligned}
\end{equation*}
The terms $B_{12i}$ ($1 \leq i \leq 3$) can be estimated as follows:
\begin{equation*}
\begin{aligned}
B_{121} = - \langle \nabla \partial^m \theta, \nabla \cdot (u \otimes \p^m u) \rangle \lesssim \| \nabla \partial^m \theta \|_{L^2}  \| \nabla \cdot (u \otimes \p^m u) \|_{L^2} \lesssim \| u \|_{H^s} \| \nabla u \|_{H^s} \| \nabla \theta \|_{H^{s-1}} ,
\end{aligned}
\end{equation*}
and
\begin{equation*}
\begin{aligned}
B_{122} \lesssim \sum_{0 \neq m' \leq m} \langle |\p^mu|, |\partial^{m'} u| \, |\nabla \nabla\partial^{m - m'} \theta| \rangle \lesssim \| u \|_{H^s} \| \nabla u \|_{H^s} \| \nabla \theta \|_{H^{s-1}} \,,
\end{aligned}
\end{equation*}
and
\begin{equation*}
\begin{aligned}
B_{123} \lesssim \|\p^m u\|_{L^2} \|\p^m(\nabla u\cdot\nabla\theta)\|_{L^2} \lesssim \| u \|_{H^s} \| \nabla u \|_{H^{s-1}} \| \nabla \theta \|_{H^{s-1}} \lesssim\| u \|_{H^s} \| \nabla u \|_{H^{s}} \| \nabla \theta \|_{H^{s-1}} \,.
\end{aligned}
\end{equation*}
Hence, $B_{12}$ satisfies the bound
\begin{equation*}
\begin{aligned}
B_{12} \lesssim\| u \|_{H^s} \| \nabla u \|_{H^{s}} \| \nabla \theta \|_{H^{s-1}} \,.
\end{aligned}
\end{equation*}
For the term $B_{13}$,
\begin{equation*}
\begin{aligned}
B_{13} & = -\langle \nabla \cdot \partial^m u, (1+\theta) \nabla \cdot \partial^m u \rangle - \sum_{0 \neq m' \leq m} C_m^{m'} \langle \nabla \cdot \partial^m u, \partial^{m'} \theta  \nabla \cdot\partial^{m - m'} u \rangle \\
&\lesssim \| \nabla \partial^m u \|_{L^2}^2 + \sum_{0 \neq m' \leq m} \langle |\nabla \partial^m u|, |\partial^{m'} \theta| \, |\nabla \partial^{m - m'} u| \rangle \\
&\lesssim \| \nabla \partial^m u \|_{L^2}^2 + \| \nabla \partial^m u \|_{L^2} \| \theta \|_{H^s} \| \nabla u \|_{H^s} \lesssim (1 + \| \theta \|_{H^s}) \| \nabla u \|_{H^s}^2 \,.
\end{aligned}
\end{equation*}
Consequently, 
\begin{equation}\label{theta-eq2}
\begin{aligned}
& - B_1 + \frac{1}{2} \frac{\d}{\d t} \left( \| \p^m u + \nabla \partial^m \theta \|_{L^2}^2 - \| \p^m u \|_{L^2}^2 - \| \nabla \partial^m \theta \|_{L^2}^2 \right) \\
\lesssim & ( 1 + \| \theta \|_{H^s}) \| u \|_{H^s} \left( \| \nabla u \|_{H^s}^2 + \| \nabla \theta \|_{H^{s-1}}^2 \right) \lesssim\| u \|_{H^s} \| \nabla u \|_{H^{s}} \| \nabla \theta \|_{H^{s-1}} \,.
\end{aligned}
\end{equation}

For the term $B_2$, we have
\begin{equation}\label{theta-eq3}
\begin{aligned}
 B_2 & = - \langle \partial^m (u \cdot \nabla u), \nabla \partial^m \theta \rangle \lesssim \| \partial^m (u \cdot \nabla u) \|_{L^2} \| \nabla \partial^m \theta \|_{L^2} \\
  & \lesssim \| u \|_{H^{s-1}} \| \nabla u \|_{H^{s-1}} \| \nabla\p^m \theta \|_{L^{2}} \lesssim \| u \|_{H^{s}} \| \nabla u \|_{H^s} \| \nabla \theta \|_{H^{s-1}} \,.
\end{aligned}
\end{equation}
The term $B_3$ can be decomposed as
\begin{equation*}
\begin{aligned}
B_{3} =\|\nabla\p^m\theta\|^2_{L^2_{ \frac{P' (1+\theta) }{ 1 + \theta } }} + \underbrace{ \sum_{0\neq m^\prime\leq m } C^{m^\prime}_m \langle \p^{m^\prime} [ \frac{ {P}' (1+\theta) }{ 1+\theta } ] \nabla \p^{m-m^\prime} \theta, \nabla \p^m \theta \rangle }_{B_{31}} \,.
\end{aligned}
\end{equation*}
The term $B_{31}$ can be estimated as
\begin{equation*}
\begin{aligned}
-B_{31}& \lesssim \sum_{0 \neq m' \leq m}\left\langle \left| \p^{m^\prime} ( \frac{ {P}' (1+\theta)}{1+\theta} ) \right| \, | \nabla\partial^{m-m'} \theta| ,  | \nabla \partial^{m} \theta | \right\rangle \lesssim \left\| \nabla \left[ \frac{ {P}' (1+\theta) }{ 1+\theta} \right] \right\|_{H^{s-1}} \left\| \nabla\theta \right\|_{H^{s-1}}^2 \\ 
& \lesssim \left( \left\| \nabla\theta \right\|_{H^{s-1}} + \left\| \nabla\theta \right\|^s_{H^{s-1}} \right) \left\| \nabla\theta \right\|_{H^{s-1}}^2 \lesssim \left( \left\|\theta \right\|_{H^s} + \left\| \theta \right\|^s_{H^s} \right) \left\| \nabla\theta \right\|_{H^{s-1}}^2 \,,
\end{aligned}
\end{equation*}
where the penultimate inequality follows from Lemma 3.2 of \cite{JLT-M3AS-2020}. Therefore,
\begin{equation}\label{theta-eq4}
\begin{aligned}
 \left\|\nabla \partial^m \theta\right\|_{L^2_\frac{ {P}' (1+\theta)}{1+\theta}}^2 - B_3 \lesssim (1 + \| \theta \|_{H^s}^{s-1}) \| \theta \|_{H^s} \| \nabla \theta \|^2_{H^{s-1}} \,.
\end{aligned}
\end{equation}

For the term $B_4$, we have
\begin{equation}\label{theta-eq5}
\begin{aligned}
B_4 & = \left\langle \frac{1}{1+\theta} \nabla \partial^m \left( - A |\nabla d|^2 + \tilde{g}(U) \right), \nabla \partial^m \theta \right\rangle \\
&\quad  + \sum_{0 \neq m' \leq m} C_m^{m'} \left\langle \partial^{m'} \left( \frac{1}{1+\theta} \right) \nabla \partial^{m - m'} \left( - A |\nabla d|^2 + \tilde{g}(U) \right), \nabla \partial^m \theta \right\rangle\\&
\lesssim \left\|\nabla \partial^m \big(-A|\nabla d|^2 + \tilde{g}(U)\big)\right\|_{L^2} \left\|\nabla \partial^m \theta\right\|_{L^2} \\
&\quad + \sum_{0 \neq m' \leq m} \left\langle \left|\partial^{m'} \left(\frac{1}{1+\theta}\right)\right| |\nabla \partial^{m - m'} \big(-A|\nabla d|^2 + \tilde{g}(U)\big)|, \nabla \partial^m \theta \right\rangle \\
&\lesssim \left( \left\| |\nabla d|^2\right\|_{H^s} + \left\|\tilde{g}(U)\right\|_{H^s} \right) \left\|\nabla \partial^m \theta\right\|_{L^2} \\
&\quad + \left\| \nabla \left(\frac{1}{1+\theta}\right) \right\|_{H^{s-1}} \left( \left\||\nabla d|^2\right\|_{H^s} + \left\|\tilde{g}(U)\right\|_{H^s} \right) \left\|\nabla \partial^m \theta\right\|_{L^2} \\
&\lesssim \left( \left\|\nabla d\right\|_{H^s}^2 + \left\|\nabla \psi\right\|_{H^s}^2 \right) \left\|\nabla \theta\right\|_{H^{s-1}} \\
&\quad + \left( \left\|\nabla \theta\right\|_{H^{s-1}} + \left\|\nabla \theta\right\|_{H^{s-1}}^s \right) \left( \left\|\nabla d\right\|_{H^s}^2 + \left\|\nabla \psi\right\|_{H^s}^2 \right) \left\|\nabla \theta\right\|_{H^{s-1}} \\
&\lesssim \left( 1 + \left\|\theta\right\|_{H^s}^s \right) \left\|\theta\right\|_{H^s} \left( \left\|\nabla d\right\|_{H^s}^2 + \left\|\nabla \psi\right\|_{H^s}^2 \right) \,.
\end{aligned}
\end{equation}
For the term $B_5$, we obtain
\begin{equation}\label{theta-eq6}
\begin{aligned}
B_5 & = \left\langle \frac{\mu}{1+\theta} \Delta \partial^m u, \nabla \partial^m \theta\right\rangle + \mu \sum_{0 \neq m' \leq m} C_m^{m'} \left\langle \partial^{m^\prime} \left( \frac{1}{1+\theta} \right) \Delta \partial^{m-m'} u, \nabla \partial^m \theta\right\rangle \\
&\leq \mu \sup_{t,x} (1 + \theta)^{- \frac{\gamma}{2}} \| \Delta \partial^m u \|_{L^2} \, \| \nabla \partial^m \theta\|_{L^2_{\w (\theta)}} \\
& \quad + C \sum_{0 \neq m' \leq m} \left\langle \left| \partial^{m^\prime} \left( \frac{1}{1+\theta} \right) \right| \, \left| \Delta \partial^{m-m'} u \right| , \, | \nabla \partial^m \theta| \right\rangle \\
&\leq \mu \sup_{t,x} (1 + \theta)^{- \frac{\gamma}{2}} \| \nabla u \|_{H^s} \, \| \nabla \theta \|_{H^{s-1}_{\w (\theta)}} + C \left\| \nabla \left( \frac{1}{1+\theta} \right) \right\|_{H^{s-1}} \, \| \nabla u \|_{H^{s}} \, \| \nabla \theta \|_{H^{s-1}} \\
&\leq \mu \sup_{t,x} (1 + \theta)^{- \frac{\gamma}{2}} \| \nabla u \|_{H^s} \, \| \nabla \theta \|_{H^{s-1}_{\w (\theta)}} + C \left( \| \theta \|_{H^s} + \| \theta \|^s_{H^s} \right) \| \nabla u \|_{H^s} \, \| \nabla \theta \|_{H^{s-1}}
\end{aligned}
\end{equation}
for some harmless constant $C>0$. Similarly, 
\begin{equation}\label{theta-eq7}
\begin{aligned}
B_6 & =\left\langle \partial^m \left[ \frac{\mu + \xi}{1+\theta} \nabla (\nabla \cdot u) \right], \nabla \partial^m \theta \right\rangle \\&
\leq (\mu + \xi ) \sup_{t,x} (1 + \theta)^{- \frac{\gamma}{2}} \| \nabla \cdot u \|_{H^s} \, \| \nabla \theta \|_{H^{s-1}_{\w (\theta)}} + C \left( \|\theta\|_{H^s} + \|\theta\|_{H^s}^s \right) \|\nabla u\|_{H^s} \|\nabla \theta\|_{H^{s-1}} \,,
\end{aligned}
\end{equation}
and
\begin{equation}\label{theta-eq8}
\begin{aligned}
B_7 \leq \sup_{t,x} (1 + \theta)^\frac{2 - \gamma}{2} \|\Delta \partial^m \psi\|_{L^2} \|\nabla \partial^m \theta\|_{L^2_{\w (\theta)}} \leq \sup_{t,x} (1 + \theta)^\frac{2 - \gamma}{2} \|\nabla \psi\|_{H^s} \|\nabla \theta\|_{H^{s-1}_{\w (\theta)}} \,.
\end{aligned}
\end{equation}

For the term $B_8$, we have the bound
\begin{align}\label{theta-eq9}
\no B_8 & = -2 \left\langle \frac{1}{1+\theta} \partial^m ( \nabla \psi \cdot \nabla \theta ), \nabla \partial^m \theta \right\rangle - 2 \sum_{ 0 \neq m^\prime \leq m } C^{ m^\prime }_m \langle \p^{ m^\prime } ( \frac{1}{ 1 + \theta } ) \p^{ m - m^\prime } ( \nabla \psi \cdot \nabla \theta ), \nabla \p^m \theta \rangle \\
\no & \lesssim \| \partial^m ( \nabla \psi \cdot \nabla \theta ) \|_{L^2} \| \nabla \p^m \theta \|_{L^2} + \sum_{ 0 \neq m' \leq m } \left\langle | \partial^{m'} \left( \frac{1}{ 1 + \theta } \right) || \partial^{ m - m' } ( \nabla \psi \cdot \nabla \theta ) |, | \nabla \partial^m \theta | \right\rangle \\
\no & \lesssim \| \nabla \psi \cdot \nabla \theta \|_{ H^{s-1} } \| \nabla \theta \|_{ H^{s-1} } + \left\| \nabla \left( \frac{1}{ 1 + \theta } \right) \right\|_{ H^{s-1} } \| \nabla \psi \cdot \nabla \theta \|_{ H^{s-1} } \| \nabla \theta \|_{ H^{s-1} } \\
& \lesssim \| \nabla \psi\|_{ H^s } \| \nabla \theta \|_{ H^{s-1} }^2 + \left( \| \theta \|_{ H^s } + \| \theta \|_{ H^s }^s \right) \| \nabla \psi \|_{ H^s } \| \nabla \theta \|_{ H^{s-1} }^2 \\
\no & \lesssim ( 1 + \| \theta \|^s_{ H^s } ) \| \nabla \psi \|_{ H^s } \| \nabla \theta \|^2_{ H^{s-1} } \,.
\end{align}
The term $B_9$ can be estimated as follows:
\begin{align}\label{theta-eq10}
\no B_9 & - \left\langle \frac{1}{1+\theta} \partial^m \left[ \theta\nabla \left( \theta+\tilde{g}(U) \right) \right], \nabla \partial^m \theta \right\rangle\\
\no & \quad-\sum_{0\neq m^\prime\leq m}C^{m^\prime}_m\left\langle \p^{m^\prime}(\frac{1}{1+\theta})\p^{m-m^\prime}[\theta\nabla(\theta+\tilde{g}(U))] ,\nabla\p^m\theta \right\rangle\\
\no & \lesssim \left\| \partial^m \left[ \theta \nabla (\theta + \tilde{g}(U) )\right] \right\|_{L^2} \left\| \nabla \partial^m \theta \right\|_{L^2}\\
\no & \quad+ \sum_{\substack{0 \neq m' \leq m}} \left\langle| \partial^{m'} \left( \frac{1}{1+\theta} \right)| |\partial^{m - m'} \left[ \theta \nabla (\theta + \tilde{g}(U) )\right]|, |\nabla \partial^m \theta| \right\rangle\\
\no & \lesssim \|\theta\|_{H^{s-1}} \left( \|\nabla \theta\|_{H^{s-1}} + \|\nabla \tilde{g}(U)\|_{H^{s-1}} \right) \|\nabla \theta\|_{H^{s-1}}\\
& \quad+ \left\| \nabla \left( \frac{1}{1+\theta} \right) \right\|_{H^{s-1}} \|\theta\|_{H^{s-1}} \left( \|\nabla \theta\|_{H^{s-1}} + \|\nabla \tilde{g}(U)\|_{H^{s-1}} \right) \|\nabla \theta\|_{H^{s-1}}\\
\no & \lesssim\|\theta\|_{H^{s-1}} \left( \|\nabla \theta\|_{H^{s-1}} + \|\nabla \psi\|_{H^s}^2 \right) \|\nabla \theta\|_{H^{s-1}} \\
\no & \quad+ \left( \|\nabla \theta\|_{H^{s-1}} + \|\nabla \theta\|_{H^{s-1}}^s \right) \|\theta\|_{H^{s-1}} \left( \|\nabla \theta\|_{H^{s-1}} + \|\nabla \psi\|_{H^s}^2 \right) \|\nabla \theta\|_{H^{s-1}}\\
\no & \lesssim \left( 1 + \|\theta\|_{H^s}^s + \|\nabla \psi\|_{H^s} \right) \|\theta\|_{H^s} \left( \|\nabla \theta\|_{H^{s-1}}^2 + \|\nabla \psi\|_{H^s}^2 \right) \,.
\end{align}
Finally, the term $B_{10}$ can be bounded by
\begin{equation}\label{theta-eq11}
\begin{aligned}
B_{10} & = - \left\langle \frac{2A}{1+\theta}\partial^m \left[  \nabla \cdot (\nabla d \odot \nabla d) \right], \nabla \partial^m \theta \right\rangle\\&
\quad -2A\sum_{0\neq m^\prime\leq m}C^{m^\prime}_m\langle \p^{m^\prime}(\frac{1}{1+\theta})\p^{m-m^\prime}\nabla\cdot(\nabla d\odot\nabla d),\nabla\p^m \theta  \rangle\\&
\lesssim \left\| \partial^m \nabla \cdot (\nabla d \odot \nabla d) \right\|_{L^2} \left\| \nabla \partial^m \theta \right\|_{L^2}\\&\quad
+ \sum_{\substack{0 \neq m' \leq m}} \left\langle| \partial^{m'} \left( \frac{1}{1+\theta} \right) ||\partial^{m - m'} \nabla \cdot (\nabla d \odot \nabla d)|, |\nabla \partial^m \theta| \right\rangle\\&
\lesssim \left\| \nabla d \odot \nabla d \right\|_{H^s} \left\| \nabla \theta \right\|_{H^{s-1}}
+ \left\| \nabla \left( \frac{1}{1+\theta} \right) \right\|_{H^{s-1}} \left\| \nabla d \odot \nabla d \right\|_{H^s} \left\| \nabla \theta \right\|_{H^{s-1}}\\&
\lesssim \left\| \nabla d \right\|_{H^s}^2 \left\| \nabla \theta \right\|_{H^{s-1}} + \left( \left\| \nabla \theta \right\|_{H^{s-1}} + \left\| \nabla \theta \right\|_{H^{s-1}}^s \right) \left\| \nabla d \right\|_{H^s}^2 \left\| \nabla \theta \right\|_{H^{s-1}} \\&
\lesssim \left( 1 + \left\| \theta \right\|_{H^s}^s \right) \left\| \theta \right\|_{H^s} \left\| \nabla d \right\|_{H^s}^2 \,.
\end{aligned}
\end{equation}
Substituting the estimates \eqref{theta-eq2}, \eqref{theta-eq3}, \eqref{theta-eq4}, \eqref{theta-eq5}, \eqref{theta-eq6}, \eqref{theta-eq7}, \eqref{theta-eq8}, \eqref{theta-eq9}, \eqref{theta-eq10} and \eqref{theta-eq11} into \eqref{theta-eq1} and summing over $|m| \leq s-1$, we obtain
\begin{align}\label{theta-eq-estimate-1}
\no &\frac{1}{2} \frac{\d}{\d t} \left( \|u + \nabla \theta\|_{H^{s-1}}^2 - \|u\|_{H^{s-1}}^2 - \|\nabla \theta\|_{H^{s-1}}^2 \right) + \|\nabla \theta\|^2_{H^{s-1}_\frac{P^\prime(1+\theta)}{1+\theta}}\\
\no & \quad - \sup_{t,x} (1 + \theta)^{- \frac{\gamma}{2}} \big[ \mu \| \nabla u \|_{H^s} \, \| \nabla \theta \|_{H^{s-1}_{\w (\theta)}} - (\mu + \xi ) \| \nabla \cdot u \|_{H^s} \, \| \nabla \theta \|_{H^{s-1}_{\w (\theta)}} \big] \\
& \quad - \sup_{t,x} (1 + \theta)^\frac{2 - \gamma}{2} \|\nabla \psi\|_{H^s} \|\nabla \theta\|_{H^{s-1}_{\w (\theta)}} \\
\no & \lesssim \left( 1 + \|\theta\|_{H^s}^2 + \|\nabla \psi\|_{H^s}^2 \right) \left( \|\theta\|_{H^s} + \|u\|_{H^s} + \|\nabla \psi\|_{H^s} \right) \\
\no & \quad\times \left( \|\nabla u\|_{H^s}^2 + \|\nabla d\|_{H^s}^2 + \|\nabla \psi\|_{H^s}^2 + \|\nabla \theta\|_{H^{s-1}_\frac{P^\prime(1+\theta)}{1+\theta}}^2 \right) \,.
\end{align}
Recalling the expression of $\omega(\theta)$ in \eqref{w-weight} together with the definitions of $\mathbf{E}_s(t)$ and $\mathbf{D}_s(t)$ in \eqref{close-eq5-1} and \eqref{close-eq6-1}, we deduce estimate \eqref{theta-eq-estimate} from \eqref{theta-eq-estimate-1}. This completes the proof of the lemma.
\end{proof}

\subsection{Estimates for the deformation gradient tensor $\nabla \psi$ in \eqref{CMEH-eqGd}}

In this subsection, we aim to establish an estimate for the $H^s$ norm of the deformation gradient tensor $\nabla \psi$. From the third equation of \eqref{CMEH-eqGd}, we have
\begin{equation}\label{gpsi-eq1}
\begin{aligned}
\partial_t \Delta \psi + \frac{1}{\mu} \Delta \psi + \frac{1}{\mu} \Delta w = -\Delta \left( u \cdot \nabla \psi \right) \,,
\end{aligned}
\end{equation}
where
\begin{equation*}
  \begin{aligned}
    w = \mu  ( u - \bar{u}) - ( \psi - \bar{\psi} ) \,, \ \bar{u} = \frac{1}{|\mathbb{T}^3|} \int_{\mathbb{T}^3} u \, \d x \,, \ \bar{\psi} = \frac{1}{|\mathbb{\mathbb{T}}^3|} \int_{\mathbb{T}^3} \psi \d x \,.
  \end{aligned}
\end{equation*}
Using the second equation of \eqref{CMEH-eqGd}, the unknown $w$ satisfies 
\begin{align}\label{gpsi-eq2}
-\Delta w = (\mu + \xi) \nabla (\nabla \cdot u) - \nabla P(1+\theta) + F, 
\end{align}
where
\begin{equation*}
  \begin{aligned}
    F = & -\theta \Delta \psi - 2 \nabla \psi \cdot \nabla \theta - \theta \nabla \left( \theta + \tilde{g}(U) \right) - 2A \nabla \cdot \left( \nabla d \odot \nabla d \right) \\
    & - (1+\theta) \left( \partial_t u + u \cdot \nabla u \right) - \nabla \left( -A |\nabla d|^2 + \tilde{g}(U) \right) \,.
  \end{aligned}
\end{equation*}
We note that the quantity $w=\mu(u-\bar{u})-(\psi-\bar{\psi})$ satisfies the Stokes system
\begin{align}\label{gpsi-eq3}
\begin{cases}
-\Delta w +\nabla q=F_{incom} \\
\qquad\quad\nabla\cdot w=-\nabla\cdot \psi
\end{cases}
\end{align}
 in the incompressible model (see (1.18) of \cite{JLL-JDE-2023}), where the source term $F_{incom}$ is controllable and $\nabla\cdot \psi = \mathrm{tr} U$ is actually {\bf nonlinear} due to the incompressibility. Unfortunately, the previous mechanism fails in the present work. As shown in \eqref{gpsi-eq2}, the unknown $w$ satisfies the elliptic equation with the source term 
 $$(\mu+\xi)\nabla(\nabla\cdot u )-\nabla P(1+\theta)+F \,,$$
 which contains the linear effect $(\mu+\xi)\nabla(\nabla\cdot u )-\nabla\theta$. This linear effect is uncontrollable if we treat $w$ using standard elliptic theory. Consequently, we should exploit the relation
 $$ - \Delta \psi \sim \p_t u + \nabla \theta - \mu \Delta u + ( u + \xi ) \nabla ( \nabla \cdot u ) + \cdots $$ 
derived from the second $u$-equation in \eqref{CMEH-eqGd} to control the quantity $ w = \mu ( u - \bar{u} ) - ( \psi - \bar{\psi} )$. More precisely, we have the following lemma.

\begin{lemma}\label{Lmm-G7}
  Let $s \ge 3$ be an integer. Then one has
\begin{equation}\label{gpsi-eq-estimate}
\begin{aligned}
& \frac{1}{2} \frac{\d}{\d t} \left( \| \nabla \psi\|_{H^s}^2 + \|u - \bar{u}\|_{H^s}^2 - \frac{2}{\mu} \sum_{|m| \leq s} \langle \partial^m (\psi - \bar{\psi}), \partial^m u \rangle \right) + \frac{1}{\mu} \|\nabla \psi\|_{H^s}^2 \\
& + ( \mu - \frac{c_p}{\mu} ) \|\nabla u\|_{H^s}^2 + (\mu + \xi) \|\nabla \cdot u\|_{H^s}^2 - \sum_{|m| \leq s} \langle \nabla \partial^m \psi, \nabla \partial^m u \rangle \\
& - \frac{\mu + \xi}{\mu} \sum_{|m| \leq s} \langle \nabla \cdot \partial^m \psi, \nabla \cdot \partial^m u \rangle - c_1 a \gamma ( \mu \| \nabla u \|_{H^s} + \| \nabla \psi \|_{H^s} ) \| \nabla \theta \|_{H^{s-1}_{ \w (\theta) }} \\
& \lesssim \left( 1 + \mathbf{E}_s^\frac{s}{2} (t) \right) \mathbf{E}_s^\frac{1}{2} (t) \mathbf{D}_s (t) \,,
\end{aligned}
\end{equation}
where the functionals $\mathbf{E}_s (t)$ and $\mathbf{D}_s (t)$ are defined  in \eqref{close-eq5-1} and \eqref{close-eq6-1}, respectively, the weighted function $\w (\theta)$ is given in \eqref{w-weight}, the constant $c_1$ is from Lemma \ref{Lmm-G6}, and $c_p > 0$ is the Poincar\'e constant.
\end{lemma}

\begin{proof}

Applying the derivative operator $\partial^m$ ($|m| \le s$) to equation \eqref{gpsi-eq1} and taking the $L^2$ inner product of the resulting equation with $\partial^m \psi$, we obtain
\begin{equation}\label{gpsi-eq4}
\begin{aligned}
\frac{1}{2} \frac{\d}{\d t} \|\nabla \partial^m \psi\|_{L^2}^2 + \frac{1}{\mu} \|\nabla \partial^m \psi\|_{L^2}^2
=  \underset{C_{1}}{\underbrace{\left\langle \Delta \partial^m (u \cdot \nabla \psi), \partial^m \psi \right\rangle}}+\underset{C_{2}}{\underbrace{\frac{1}{\mu} \left\langle \Delta \partial^m w, \partial^m \psi \right\rangle}} \,.
\end{aligned}
\end{equation}
The term $C_1$ can be decomposed as
\begin{equation*}
\begin{aligned}
C_1 & =\left\langle \Delta \partial^m (u \cdot \nabla \psi), \partial^m \psi \right\rangle = -\left\langle \nabla \partial^m (u \cdot \nabla \psi), \nabla\partial^m \psi \right\rangle\\&
=-\langle \p^m(u\cdot\nabla\nabla\psi),\nabla\p^m\psi\rangle-\langle\p^m(\nabla u \cdot\nabla \psi),\nabla\p^m\psi  \rangle\\&
=\underset{C_{11}}{\underbrace{-\left\langle u \cdot \nabla \nabla \partial^m \psi, \nabla \partial^m \psi \right\rangle}}
\underset{C_{12}}{\underbrace{- \sum_{\substack{m' \leq m \\ |m'| \geq 2}} C_m^{m'} \left\langle \partial^{m'} u \cdot \nabla \nabla \partial^{m - m'} \psi, \nabla \partial^m \psi \right\rangle}}\\&
\quad\underset{C_{13}}{\underbrace{- \sum_{\substack{m' \leq m \\ |m'| = 1}} C_m^{m'} \left\langle \partial^{m'} u \cdot \nabla \nabla \partial^{m - m'} \psi, \nabla \partial^m \psi \right\rangle}}
\underset{C_{14}}{\underbrace{- \left\langle \partial^m (\nabla u \cdot \nabla \psi), \nabla \partial^m \psi \right\rangle}}  \,.
\end{aligned}
\end{equation*}
The terms $C_{1i}$ ($1 \le i \le 4$) are estimated as follows:
\begin{equation*}
\begin{aligned}
C_{11} = \frac{1}{2} \langle \nabla \cdot u, | \nabla \p^m \psi |^2  \rangle \lesssim \|\nabla \cdot u \|_{L^\infty} \| \nabla \p^m \psi \|_{L^2}^2 \lesssim \| \nabla \psi \|^2_{H^s} \| \nabla u \|_{H^s} \,,
\end{aligned}
\end{equation*}
and
\begin{equation*}
\begin{aligned}
C_{12}&\lesssim\sum_{\substack{m' \leq m \\ |m'| \geq 2}} \left\langle |\partial^{m'} u| |\nabla \nabla \partial^{m - m'} \psi|, |\nabla \partial^m \psi| \right\rangle\\&
\lesssim \sum_{\substack{m' \leq m \\ |m'| \geq 2}} \|\partial^{m'} u\|_{L^4} \|\nabla \nabla \partial^{m - m'} \psi\|_{L^4} \|\nabla \partial^m \psi\|_{L^2} \lesssim \|\nabla \psi\|_{H^s}^2 \|\nabla u\|_{H^s} \,,
\end{aligned}
\end{equation*}
and
\begin{equation*}
\begin{aligned}
C_{13}&=\sum_{\substack{m' \leq m \\ |m'| = 1}} \left\langle |\partial^{m'} u| |\nabla \nabla \partial^{m - m'} \psi|, |\nabla \partial^m \psi| \right\rangle\\&
\lesssim \sum_{\substack{m' \leq m \\ |m'| = 1}}\|\partial^{m'} u\|_{L^\infty}\|\nabla \nabla \partial^{m - m'} \psi\|_{L^2}\|\nabla \partial^m \psi\|_{L^2} \lesssim\|\nabla \psi\|^2_{H^s}\|\nabla u\|_{H^s} \,,
\end{aligned}
\end{equation*}
and
\begin{equation*}
\begin{aligned}
C_{14} \lesssim \| \p^m ( \nabla u \cdot \nabla \psi ) \|_{L^2} \| \nabla \p^m \psi \|_{L^2} \lesssim \| \nabla u \cdot \nabla \psi \|_{H^s} \| \nabla \psi \|_{H^s} \lesssim \| \nabla \psi \|^2_{H^s} \| \nabla u \|_{H^s} \,.
\end{aligned}
\end{equation*}
Consequently,
\begin{equation}\label{gpsi-eq5}
\begin{aligned}
C_1=\langle \Delta\p^m(u\cdot\nabla\psi),\p^m\psi   \rangle\lesssim\|\nabla\psi\|^2_{H^s}\|\nabla u\|_{H^s} \,.
\end{aligned}
\end{equation}

We now estimate the term $C_2 = \frac{1}{\mu} \langle \Delta \p^m w, \p^m \psi \rangle$, which involves the unknown $w$ in \eqref{gpsi-eq4}. Integrating by parts over $x\in \mathbb{T}^3$, we obtain
\begin{equation}\label{gpsi-eq6}
\begin{aligned}
C_2=\frac{1}{\mu} \left\langle \Delta \partial^m w, \partial^m \psi \right\rangle=\frac{1}{\mu} \left\langle  \partial^m w, \Delta\partial^m \psi \right\rangle  \,.
\end{aligned}
\end{equation}
From the second $u$-equation in \eqref{CMEH-eqGd}, it follows that
\begin{align}\label{gpsi-eq7}
\no -\Delta \psi &= \partial_t u + u \cdot \nabla u + \frac{P'(1+\theta)}{1+\theta} \nabla \theta + \frac{1}{1+\theta} \nabla \left(-A |\nabla d|^2 + \tilde{g}(U)\right) \\
\no & \quad -\frac{\mu}{1+\theta} \Delta u - \frac{\mu+\xi}{1+\theta} \nabla (\nabla \cdot u) + \frac{2}{1+\theta} \nabla {\psi} \cdot \nabla \theta \\
& \quad + \frac{\theta}{1+\theta} \nabla \left(\theta + \tilde{g}(U)\right) + \frac{2A}{1+\theta} \nabla \cdot \left( \nabla d \odot \nabla d \right)
 \,.
\end{align}
 Then \eqref{gpsi-eq6} and \eqref{gpsi-eq7} yield
 \begin{equation}\label{gpsi-eq8}
\begin{aligned}
C_2 & = \underset{ C_{21} }{ \underbrace{ - \frac{1}{\mu} \langle \p^m w, \p_t \p^m u \rangle }} - \underset{ C_{22} }{ \underbrace{ \frac{1}{\mu} \langle \p^m w, \p^m ( u \cdot \nabla u ) \rangle}} \\
& \quad \underset{ C_{23} }{ \underbrace{ - \frac{1}{\mu} \left\langle \partial^m w, \partial^m \left[ \frac{ P' ( 1 + \theta ) }{ 1 + \theta } \nabla \theta \right] \right\rangle }} + \underset{ C_{24} }{ \underbrace{ \frac{1}{\mu} \left\langle \partial^m w, \partial^m \left( \frac{\mu}{ 1+\theta } \Delta u \right) \right\rangle }} \\
& \quad \underset{ C_{25} }{ \underbrace{ - \frac{1}{\mu} \left\langle \partial^m w, \partial^m \left[ \frac{1}{ 1+\theta } \nabla \left( - A | \nabla d |^2 + \tilde{g} (U) \right) \right] \right\rangle }} \\
& \quad + \underset{ C_{26} }{ \underbrace{ \frac{1}{\mu} \left\langle \partial^m w, \partial^m \left( \frac{ \mu + \xi }{ 1 + \theta } \nabla ( \nabla \cdot u ) \right) \right\rangle }} \ \underset{ C_{27} }{ \underbrace{ - \frac{1}{\mu} \left\langle \partial^m w, \partial^m \left[ \frac{2}{ 1 + \theta } \nabla \psi \cdot \nabla \theta \right] \right\rangle }} \\
& \quad \underset{ C_{28} }{ \underbrace{ - \frac{1}{\mu} \left\langle \partial^m w, \partial^m \left[ \frac{\theta}{1+\theta} \nabla \left( \theta + \tilde{g} (U) \right) \right] \right\rangle }} \ \underset{ C_{29} }{ \underbrace{ - \frac{1}{\mu} \left\langle \partial^m w, \partial^m \left[ \frac{2A}{ 1 + \theta } \nabla \cdot \left( \nabla d \odot \nabla d \right) \right] \right\rangle }} \,.
\end{aligned}
\end{equation}

We now handle the terms $C_{2i}$ ($1 \le i \le 9$) one by one. The term $C_{21}$ can be computed as
\begin{align*}
C_{21} &  = -\frac{1}{\mu} \left\langle \partial^m [\mu(u-\bar{u}) -(\psi-\bar{\psi}), \partial_t \partial^m u \right\rangle \\
& = -\left\langle \partial^m (u-\bar{u}), \partial_t \partial^m (u-\bar{u}) \right\rangle + \left\langle u-\bar{u}, \partial_t \bar{u} \right\rangle + \frac{1}{\mu} \left\langle \partial^m (\psi-\bar{\psi}), \partial_t \partial^m u \right\rangle \\
&= -\frac{1}{2} \frac{\d}{\d t} \left\| \partial^m (u-\bar{u}) \right\|_{L^2}^2 + \frac{1}{\mu} \frac{\d}{\d t} \left\langle \partial^m (\psi-\bar{\psi}), \partial^m u \right\rangle \\
&\quad -\frac{1}{\mu} \left\langle \partial^m \partial_t (\psi-\bar{\psi}), \partial^m u \right\rangle + \left\langle u-\bar{u}, \partial_t \bar{u} \right\rangle \,.
\end{align*}
Note the cancellation
\begin{equation*}
\begin{aligned}
\langle u-\bar{u},\p_t \bar{u} \rangle=\p_t\bar{u}\cdot \int_{\mathbb{T}^3}{(u-\bar{u})}\d x=0 \,.
\end{aligned}
\end{equation*}
Thus,
\begin{equation*}
\begin{aligned}
C_{21}&=-\frac{1}{2} \frac{\d}{\d t}\big( \left\| \partial^m (u-\bar{u}) \right\|_{L^2}^2 - \frac{2}{\mu}\langle \partial^m (\psi-\bar{\psi}), \partial^m u \rangle\big) -\frac{1}{\mu} \left\langle \partial^m \partial_t (\psi-\bar{\psi}), \partial^m u \right\rangle \,.
\end{aligned}
\end{equation*}
For the inner product $-\frac{1}{\mu} \left\langle \partial^m \partial_t (\psi-\bar{\psi}), \partial^m u \right\rangle$, we use the third $\psi$-equation of \eqref{CMEH-eqGd}: 
$$
\partial_t \left( \psi - \bar{\psi} \right) = - \left( u - \bar{u} \right) - u \cdot \nabla \psi + \bar{u \cdot \nabla \psi}$$
Then, together with the fact $\langle u-\bar{u},\bar{u}\rangle=0$, we obtain
\begin{align*}
&-\frac{1}{\mu} \left\langle \partial^m \partial_t \left( \psi - \bar{\psi} \right), \partial^m u \right\rangle \\
&= \frac{1}{\mu} \left\langle \partial^m (u - \bar{u}), \partial^m u \right\rangle + \frac{1}{\mu} \left\langle \partial^m \left( u \cdot \nabla \psi -\overline{u \cdot \nabla \psi} \right), \partial^m u \right\rangle \\
&= \frac{1}{\mu} \left\langle \partial^m (u - \bar{u}), \partial^m (u - \bar{u}) \right\rangle + \frac{1}{\mu} \left\langle u - \bar{u}, \bar{u} \right\rangle + \frac{1}{\mu} \left\langle \partial^m \left( u \cdot \nabla \psi - \overline{u \cdot \nabla \psi} \right), \partial^m u \right\rangle \\
&= \underbrace{\frac{1}{\mu} \left\| \partial^m (u - \bar{u}) \right\|^2_{L^2}}_{{C}_{211}} + \underbrace{\frac{1}{\mu} \left\langle \partial^m (u \cdot \nabla \psi), \partial^m u \right\rangle}_{{C}_{212}} \ \underbrace{ - \frac{1}{\mu}\overline{u \cdot \nabla \psi} \left\langle \mathbf{1}, u \right\rangle}_{{C}_{213}} \,.
\end{align*}
The term $C_{211}$ satisfies
\begin{equation*}
\begin{aligned}
C_{211}= \frac{1}{\mu}\| \partial^m ( u-\bar{u}) \|^2_{L^2}\leq \frac{c_p}{\mu}\|\nabla \partial^m u\|^2_{L^2} 
\end{aligned}
\end{equation*}
by the Poincar\'e inequality. For $C_{212}$,
\begin{equation*}
\begin{aligned}
C_{212}&= \frac{1}{\mu} \sum_{0 \neq m' \leq m} \mathrm{C}_m^{m'} \left\langle \partial^{m'} u \cdot \nabla \partial^{m - m'} \psi, \partial^m u \right\rangle + \frac{1}{\mu} \left\langle u \cdot \nabla \partial^m \psi, \partial^m u \right\rangle \\
&\lesssim \sum_{0 \neq m' \leq m} \|\partial^{m'} u\|_{L^4} \|\nabla \partial^{m - m'} \psi\|_{L^4} \|\partial^m u\|_{L^2} + \|u\|_{L^4} \|\nabla \partial^m \psi\|_{L^2} \|\partial^m u\|_{L^4} \\
&\lesssim \|u\|_{H^s} \|\nabla u\|_{H^s} \|\nabla \psi\|_{H^s} + \|u\|_{L^2}^{\frac{1}{4}} \|\nabla u\|_{L^2}^{\frac{3}{4}} \|\nabla \partial^m \psi\|_{L^2} \|\partial^m u\|_{L^2}^{\frac{1}{4}} \|\nabla\partial^m u\|_{L^2}^{\frac{3}{4}} \\
&\lesssim \|u\|_{H^s} \|\nabla u\|_{H^s} \|\nabla \psi\|_{H^s} \,,
\end{aligned}
\end{equation*}
For $C_{213}$, it is easy to see that
\begin{equation*}
\begin{aligned}
C_{213}&=-\frac{1}{\mu} \frac{1}{|\mathbb{T}|^3} \int_{\mathbb{T}^3} u \cdot \nabla \psi \, \d x \int_{\mathbb{T}^3} u \, \d x = -\frac{1}{\mu |\mathbb{T}|^3} \int_{\mathbb{T}^3} (u - \overline{u})\cdot \nabla \psi \, \d x \int_{\mathbb{T}^3} u \, \d x \\
& \lesssim \| u - \overline{u} \|_{L^2} \, \| \nabla \psi \|_{L^2} \, \| u \|_{L^2} \lesssim \| u \|_{L^2} \, \| \nabla u \|_{L^2} \, \| \nabla \psi \|_{L^2} \,,
\end{aligned}
\end{equation*}
where we used $ \int_{\mathbb{T}^3} \bar{u} \cdot \nabla \psi \d x=0$ and the Poincar\'e inequality $ \| u - \bar{u} \|_{L^2} \lesssim \| \nabla u \|_{L^2}$. As a result,
\begin{equation*}
\begin{aligned}
-\frac{1}{\mu}\langle \p^m\p_t(\psi-\bar{\psi}),\p^m u\rangle\leq C\|\nabla u\|^2_{H^s}+C\|u\|_{H^s}\|\nabla u\|_{H^s}\|\nabla \psi\|_{H^s} \,,
\end{aligned}
\end{equation*}
which implies
\begin{equation}\label{gpsi-eq9}
\begin{aligned}
C_{21} & \leq - \frac{1}{2} \frac{\d}{\d t} \left( \| \partial^m (u-\bar{u})\|_{L^2}^2 - \frac{2}{\mu} \langle \partial^m(\psi - \bar{\psi}), \partial^m u \rangle \right) \\
& \quad + \frac{c_p}{\mu}\|\nabla \partial^m u\|^2_{L^2} + C \| u \|_{H^s} \| \nabla u \|_{H^s} \| \nabla \psi \|_{H^s} 
\end{aligned}
\end{equation}
for some constant $C > 0 $. Moreover, the term $C_{22}$ admits the estimate
\begin{equation}\label{gpsi-eq10}
\begin{aligned}
C_{22}&= -\frac{1}{\mu} \langle \partial^m w, \partial^m (u\cdot \nabla u) \rangle \\&= -\langle \partial^m (u-\overline{u}), \partial^m (u\cdot \nabla u) \rangle + \frac{1}{\mu} \langle \partial^m (\psi - \overline{\psi}), \partial^m (u\cdot \nabla u) \rangle\\&
\lesssim \|\partial^m (u-\overline{u})\|_{L^2} \|\partial^m (u\cdot \nabla u)\|_{L^2} + \|\partial^m (\psi - \overline{\psi})\|_{L^2} \|\partial^m (u\cdot \nabla u)\|_{L^2}\\&
\lesssim (\|u-\overline{u}\|_{H^s} + \|\psi - \overline{\psi}\|_{H^s}) \|u\cdot \nabla u\|_{H^s}\\&
\lesssim (\|\nabla u\|_{H^s} + \|\nabla \psi\|_{H^s}) \|u\|_{H^s} \|\nabla u\|_{H^s} \,,
\end{aligned}
\end{equation}
where we used the Poincar\'e inequalities $\|u-\bar{u}\|_{L^2}\lesssim\|\nabla u\|_{L^2}$ and $\|\psi-\bar{\psi}\|_{L^2}\lesssim\|\nabla \psi\|_{L^2}$.

We now turn to the term $C_{23}$. It can be split as
\begin{align}\label{gpsi-eq11}
\no C_{23}&=-\frac{1}{\mu} \left\langle \partial^m w, \partial^m \left[ \frac{P'(1+\theta)}{1+\theta} \nabla \theta \right] \right\rangle\\
\no & = -\frac{1}{\mu} \left\langle \partial^m w, \frac{P^\prime(1+\theta)}{1+\theta} \nabla \partial^m \theta \right\rangle -\frac{1}{\mu} \sum_{\substack{0\neq m^\prime \leq m}} C_m^{m'} \left\langle \partial^m w, \partial^{m'} \left[\frac{P'(1+\theta)}{1+\theta}\right] \nabla \partial^{m-m'} \theta \right\rangle \\
\no &= \underset{C_{231}}{\underbrace{\frac{1}{\mu} \left\langle \nabla \cdot \partial^m w, \frac{P'(1+\theta)}{1+\theta} \partial^m \theta \right\rangle {\mathbf{1}}_{m\neq 0} }} \ \underset{C_{232}}{\underbrace{ -\frac{1}{\mu} \left\langle w, \frac{P'(1+\theta)}{1+\theta} \nabla \theta \right\rangle}} \quad  \\
\no & \quad \underset{C_{233}}{\underbrace{+\frac{1}{\mu} \left\langle \partial^m w, \nabla \left[\frac{P'(1+\theta)}{1+\theta}\right] \partial^m \theta \right\rangle }} \\
\no &\quad \underset{C_{234}}{\underbrace{-\frac{1}{\mu} \sum_{\substack{m' \leq m \\ |m'| = 1}} C_m^{m'} \left\langle \partial^m w, \partial^{m'} \left[\frac{P'(1+\theta)}{1+\theta}\right] \nabla \partial^{m-m'} \theta \right\rangle}} \\
& \quad \underset{C_{235}}{\underbrace{ -\frac{1}{\mu} \sum_{\substack{m' \leq m \\ |m'| \geq 2}} C_m^{m'} \left\langle \partial^m w, \partial^{m'} \left[\frac{P'(1+\theta)}{1+\theta}\right] \nabla \partial^{m-m'} \theta \right\rangle}} \,.
\end{align}
The terms $C_{23i}$ ($1 \le i \le 5$) are estimated term by term. For $C_{231}$,
\begin{equation}\label{gpsi-eq12}
\begin{aligned}
C_{231} & \lesssim \left\| \nabla \cdot \partial^m w \right\|_{L^2} \left\| \partial^m \theta \right\|_{L^2_{\w (\theta)}}\mathbf{1}_{m\neq 0} \left\| \frac{P'(1+\theta)}{1+\theta} \right\|_{L^\infty} \\
& \lesssim a \gamma \left( \mu \left\| \nabla u \right\|_{H^s} + \left\| \nabla \psi \right\|_{H^s} \right) \left\| \nabla \theta \right\|_{H^{s-1}_{\w (\theta)}} \,.
\end{aligned}
\end{equation}
Similarly,
\begin{equation}\label{gpsi-eq13}
\begin{aligned}
C_{232}&\lesssim \| w \|_{L^2} \| \nabla \theta \|_{L^2_{\w (\theta)}} \left\| \frac{P'(1+\theta)}{1+\theta} \right\|_{L^\infty} \lesssim a \gamma \left( \mu \| u - \overline{u} \|_{L^2} + \| \psi - \overline{\psi} \|_{L^2} \right) \| \nabla \theta \|_{L^2_{\w (\theta)}} \\&
\lesssim a \gamma \left( \mu \| \nabla u \|_{L^2} + \| \nabla \psi \|_{L^2} \right) \| \nabla \theta \|_{L^2_{\w (\theta)}} \,.
\end{aligned}
\end{equation}
The remaining terms are estimated as
\begin{equation}\label{gpsi-eq14}
\begin{aligned}
C_{233}&\lesssim\mathbf{1}_{m \neq 0} \|\partial^m w\|_{L^2} \left\| \left( \frac{P'(1+\theta)}{1+\theta} \right)' \right\|_{L^\infty}\|\nabla \theta\|_{L^\infty}\|\partial^m \theta\|_{L^2} \\
&\lesssim \left( \|\nabla u\|_{H^s} + \|\nabla \psi\|_{H^s} \right) \|\theta\|_{H^s} \|\nabla \theta\|_{H^{s-1}}
 \,,
\end{aligned}
\end{equation}
and
\begin{equation}\label{gpsi-eq15}
\begin{aligned}
C_{234}&\lesssim \sum_{\substack{m' \leq m \\ |m'|=1}} \|\partial^m w\|_{L^2} \left\| \partial^{m'} \left[ \frac{P'(1+\theta)}{1+\theta} \right] \right\|_{L^\infty} \|\nabla \partial^{m-m'} \theta\|_{L^2} \\
&\lesssim \left( \|\nabla u\|_{H^s} + \|\nabla \psi\|_{H^s} \right) \|\theta\|_{H^s} \|\nabla \theta\|_{H^{s-1}}
 \,,
\end{aligned}
\end{equation}
and
\begin{equation}\label{gpsi-eq16}
\begin{aligned}
C_{235}&\lesssim \sum_{\substack{m' \leq m \\ |m'|\geq 2}} \|\partial^m w\|_{L^4} \left\| \partial^{m'} \left[ \frac{P'(1+\theta)}{1+\theta} \right] \right\|_{L^2} \|\nabla \partial^{m-m'} \theta\|_{L^4} \\
&\lesssim \left( \|\nabla u\|_{H^s} + \|\nabla \psi\|_{H^s} \right) \left\| \nabla \left[ \frac{P'(1+\theta)}{1+\theta} \right] \right\|_{H^{s-1}} \|\nabla \theta\|_{H^{s-1}} \\
&\lesssim \left( \|\theta\|_{H^s} + \|\theta\|_{H^s}^s \right) \left( \|\nabla u\|_{H^s} + \|\nabla \psi\|_{H^s} \right) \|\nabla \theta\|_{H^{s-1}} \,.
\end{aligned}
\end{equation}
Substituting the estimates \eqref{gpsi-eq12}, \eqref{gpsi-eq13}, \eqref{gpsi-eq14}, \eqref{gpsi-eq15} and \eqref{gpsi-eq16} into \eqref{gpsi-eq11} yields
\begin{equation}\label{gpsi-eq17}
\begin{aligned}
C_{23} &\lesssim a \gamma \left( \mu \left\| \nabla u \right\|_{H^s} + \left\| \nabla \psi \right\|_{H^s} \right) \left\| \nabla \theta \right\|_{H^{s-1}_{\w (\theta)}} \\
&\quad + \left( \|\theta\|_{H^s} + \|\theta\|_{H^s}^s \right) \left( \|\nabla u\|_{H^s} + \|\nabla \psi\|_{H^s} \right) \|\nabla \theta\|_{H^{s-1}} \,.
\end{aligned}
\end{equation}

The quantity $C_{24}$ can be decomposed as
\begin{align*}
C_{24}&=\frac{1}{\mu} \left\langle \partial^m w, \partial^m \left( \frac{\mu}{1+\theta} \Delta u \right) \right\rangle \\&
= \left\langle \partial^m w, \frac{1}{1+\theta} \Delta \partial^m u \right\rangle + \sum_{\substack{0 \neq m' \leq m}} C_m^{m'} \left\langle \partial^m w, \partial^{m'} \left( \frac{1}{1+\theta} \right) \Delta \partial^{m-m'} u \right\rangle \\&
= \underset{C_{241}}{\underbrace{-\left\langle \nabla \partial^m w, \frac{1}{1+\theta} \nabla \partial^m u \right\rangle }} - \underset{C_{242}}{\underbrace{\left\langle \partial^m w, \nabla \left( \frac{1}{1+\theta} \right) \cdot \nabla \partial^m u \right\rangle }} \\&
\quad+ \underset{C_{243}}{\underbrace{\sum_{\substack{m' \leq m \\ |m'|=1}} C_m^{m'} \left\langle \partial^m w, \partial^{m'} \left( \frac{1}{1+\theta} \right) \Delta \partial^{m-m'} u \right\rangle}} \\
& \quad+ \underset{C_{244}}{\underbrace{\sum_{\substack{m' \leq m \\ |m'| \geq 2}} C_m^{m'} \left\langle \partial^m w, \partial^{m'} \left( \frac{1}{1+\theta} \right) \Delta \partial^{m-m'} u \right\rangle}} \,.
\end{align*}
By H\"older's inequality and Sobolev embedding, we obtain
\begin{equation*}
\begin{aligned}
C_{241}&=-\langle\nabla\p^m w,\nabla \p^m u\rangle+\langle\nabla\p^m w,\frac{\theta}{1+\theta}\nabla \p^m u\rangle
\\&\lesssim \left\| \frac{1}{1+\theta} \right\|_{L^\infty} \left\| \nabla \partial^m w \right\|_{L^2} \left\| \nabla \partial^m u \right\|_{L^2} - \mu \left\| \nabla \partial^m u \right\|_{L^2}^2 + \left\langle \nabla\partial^m \psi, \nabla \partial^m u \right\rangle \\&
\lesssim C\|\theta\|_{H^s}\left( \left\| \nabla u \right\|_{H^s} + \left\| \nabla \psi \right\|_{H^s} \right) \left\| \nabla u \right\|_{H^s} - \mu \left\| \nabla \partial^m u \right\|_{L^2}^2 + \left\langle \nabla\partial^m \psi, \nabla \partial^m u \right\rangle \,,
\end{aligned}
\end{equation*}
and
\begin{equation*}
\begin{aligned}
C_{242}&\lesssim \left\| \partial^m w \right\|_{L^4} \left\| \nabla \left( \frac{1}{1+\theta} \right) \right\|_{L^4} \left\| \nabla \partial^m u \right\|_{L^2} \\&
\lesssim \left( \left\| \nabla \psi \right\|_{H^s} + \left\| \nabla u \right\|_{H^s} \right) \left\| \theta \right\|_{H^s} \left\| \nabla u \right\|_{H^s} \,,
\end{aligned}
\end{equation*}
and
\begin{equation*}
\begin{aligned}
C_{243}&\lesssim\sum_{\substack{m' \leq m \\ |m'|=1}} \left\| \partial^m w \right\|_{L^2} \left\| \partial^{m'} \left( \frac{1}{1+\theta} \right) \right\|_{L^\infty} \left\| \Delta \partial^{m-m'} u \right\|_{L^2} \\&
\lesssim \left( \left\| \nabla \psi \right\|_{H^s} + \left\| \nabla u \right\|_{H^s} \right) \left\| \theta \right\|_{H^s} \left\| \nabla u \right\|_{H^s} \,,
\end{aligned}
\end{equation*}
and
\begin{equation*}
\begin{aligned}
C_{244}&\lesssim \sum_{\substack{m' \leq m \\ |m'| \geq 2}} \left\| \partial^m w \right\|_{L^4} \left\| \partial^{m'} \left( \frac{1}{1+\theta} \right) \right\|_{L^2} \left\| \Delta \partial^{m-m'} u \right\|_{L^4} \\&
\lesssim \left\| w \right\|_{H^{s+1}} \left\| \nabla \left( \frac{1}{1+\theta} \right) \right\|_{H^{s-1}} \left\| \nabla u \right\|_{H^s} \\&
\lesssim \left( \left\| \nabla \psi \right\|_{H^s} + \left\| \nabla u \right\|_{H^s} \right) \left( \left\| \theta \right\|_{H^s} + \left\| \theta \right\|_{H^s}^s \right) \left\| \nabla u \right\|_{H^s} \,.
\end{aligned}
\end{equation*}
Consequently,
\begin{equation}\label{gpsi-eq18}
\begin{aligned}
C_{24} \leq -\mu \| \nabla \partial^m u \|_{L^2}^2 + \langle \nabla \partial^m \psi, \nabla \partial^m u \rangle + ( \|\theta\|_{H^s} + \|\theta\|_{H^s}^s ) ( \|\nabla \psi\|_{H^s} + \|\nabla u\|_{H^s} ) \| \nabla u \|_{H^s} \,.
\end{aligned}
\end{equation}

For the term $C_{25}$, we split it as
\begin{align*}
C_{25} & = \frac{1}{\mu} \left\langle \nabla \cdot \partial^m w, \partial^m \left[ \frac{1}{1+\theta} (-A | \nabla d |^2 + \tilde{g}(U)) \right] \right\rangle \\
&\quad + \frac{1}{\mu} \left\langle \partial^m w, \partial^m \left[ \nabla \left(\frac{1}{1+\theta}\right) (-A | \nabla d |^2 + \tilde{g}(U)) \right] \right\rangle \\
&= \underset{C_{251}}{\underbrace{\frac{1}{\mu} \left\langle \nabla \cdot \partial^m w, \partial^m \left[ \frac{1}{1+\theta} (-A | \nabla d |^2 + \tilde{g}(U)) \right] \right\rangle}} \\
&\quad +\underset{C_{252}}{\underbrace{ \frac{1}{\mu} \sum_{0 \neq m' \leq m} C_m^{m'} \left\langle \partial^m w, \nabla \partial^{m-m'} \left(\frac{1}{1+\theta}\right) \partial^{m'} (-A | \nabla d |^2 + \tilde{g}(U)) \right\rangle}}{} \\
&\quad \underset{C_{253}}{\underbrace{- \frac{1}{\mu} \left\langle \partial^m w, \partial^m \left(\frac{1}{1+\theta}\right) \nabla (-A | \nabla d |^2 + \tilde{g}(U)) \right\rangle}} \,.
\end{align*}
We then obtain 
\begin{equation*}
\begin{aligned}
C_{251}&\lesssim\| \nabla w \|_{H^s} \left\| \frac{1}{1+\theta} (-A | \nabla d |^2 + \tilde{g}(U)) \right\|_{H^s} \\&
\lesssim ( \| \nabla u \|_{H^s} + \| \nabla \psi \|_{H^s} ) ( 1 + \left\| \nabla \left(\frac{1}{1+\theta}\right) \right\|_{H^{s-1}} ) ( \| |\nabla d|^2 \|_{H^s} + \| \tilde{g}(U) \|_{H^s} ) \\&
\lesssim( 1 + \| \theta \|_{H^s}^s ) ( \| \nabla u \|_{H^s} + \| \nabla \psi \|_{H^s} ) ( \| \nabla d \|_{H^s}^2 + \| \nabla \psi \|_{H^s}^2 ) \\&
\lesssim ( 1 + \| \theta \|_{H^s}^s ) ( \| \nabla d \|_{H^s} + \| \nabla \psi \|_{H^s} ) ( \| \nabla u \|_{H^s}^2 + \| \nabla d \|_{H^s}^2 + \| \nabla \psi \|_{H^s}^2 )
\,,
\end{aligned}
\end{equation*}
and
\begin{equation*}
\begin{aligned}
C_{252}&\lesssim\sum_{0\neq m' \leq m} \| \partial^m w \|_{L^4} \left\| \nabla \partial^{m-m'} \left(\frac{1}{1+\theta}\right) \right\|_{L^2} \| \partial^{m'} (-A | \nabla d |^2 + \tilde{g}(U)) \|_{L^4} \\
&\quad + \| \partial^m w \|_{L^4} \left\| \nabla \left(\frac{1}{1+\theta}\right) \right\|_{L^4} \| \partial^{m'} (-A | \nabla d |^2 + \tilde{g}(U)) \|_{L^2}\\&
\lesssim \| w \|_{H^{s+1}} \left\| \nabla \left( \frac{1}{1+\theta} \right) \right\|_{H^{s-1}} \left( \left\| |\nabla d|^2 \right\|_{H^s} + \left\| \tilde{g}(U) \right\|_{H^s} \right) \\&
\lesssim \left( \| \theta \|_{H^s} + \| \theta \|_{H^s}^s \right) \left( \| \nabla u \|_{H^s} + \| \nabla \psi \|_{H^s} \right) \left( \| \nabla d \|_{H^s}^2 + \| \nabla \psi \|_{H^s}^2 \right) \\&
\lesssim (1+\| \theta \|_{H^s}^s) \left( \| \nabla d \|_{H^s} + \| \nabla \psi \|_{H^s} \right) \left( \| \nabla u \|_{H^s}^2 + \| \nabla d \|_{H^s}^2 + \| \nabla \psi \|_{H^s}^2 \right) \,,
\end{aligned}
\end{equation*}
and
\begin{equation*}
\begin{aligned}
C_{253}&\lesssim \left\| \partial^m w \right\|_{L^4} \left\| \partial^m \left( \frac{1}{1+\theta} \right) \right\|_{L^2} \left\| \nabla \left( -A |\nabla d|^2 + \tilde{g}(U) \right) \right\|_{L^4} \\&
\lesssim \| w \|_{H^{s+1}} \left\| \frac{1}{1+\theta} \right\|_{H^s} \left( \left\| \nabla d^2 \right\|_{H^2} + \left\| \tilde{g}(U) \right\|_{H^2} \right) \\&
\lesssim (1+\| \theta \|_{H^s}^s) \left( \| \nabla u \|_{H^s} + \| \nabla \psi \|_{H^s} \right) \left( \| \nabla d \|_{H^s}^2 + \| \nabla \psi \|_{H^s}^2 \right) \\&
\lesssim (1+\| \theta \|_{H^s}^s) \left( \| \nabla d \|_{H^s} + \| \nabla \psi \|_{H^s} \right) \left( \| \nabla u \|_{H^s}^2 + \| \nabla d \|_{H^s}^2 + \| \nabla \psi \|_{H^s}^2 \right) \,.
\end{aligned}
\end{equation*}
Hence,
\begin{equation}\label{gpsi-eq19}
\begin{aligned}
C_{25}&\lesssim (1+\| \theta \|_{H^s}^s) \left( \| \nabla d \|_{H^s} + \| \nabla \psi \|_{H^s} \right) \left( \| \nabla u \|_{H^s}^2 + \| \nabla d \|_{H^s}^2 + \| \nabla \psi \|_{H^s}^2 \right) \,.
\end{aligned}
\end{equation}

By arguments analogous to those in \eqref{gpsi-eq18}, we obtain
\begin{equation}\label{gpsi-eq20}
\begin{aligned}
C_{26}&=\frac{1}{\mu} \left\langle \partial^m w, \partial^m \left( \frac{\mu+\xi}{1+\theta} \nabla (\nabla \cdot u) \right) \right\rangle \\&
\leq -\left( \mu+\xi \right) \left\| \nabla \cdot \partial^m u \right\|_{L^2}^2 + \frac{\mu+\xi}{\mu} \left\langle \nabla \cdot \partial^m \psi, \nabla \cdot \partial^m u \right\rangle  \\&
+  C\left( \| \theta \|_{H^s} + \| \theta \|_{H^s}^s \right) \left( \| \nabla \psi \|_{H^s} + \| \nabla u \|_{H^s} \right) \| \nabla u \|_{H^s}  
\end{aligned}
\end{equation}
for some constant $C > 0$.

We now split $C_{27}$ as
\begin{align*}
C_{27}& = -\frac{1}{\mu} \left\langle \partial^m w, \partial^m \left[ \frac{2}{1+\theta} \nabla \psi \cdot \nabla \theta \right] \right\rangle \\&
= \underset{C_{271}}{\underbrace{\frac{1}{\mu} \left\langle \nabla \cdot \partial^m w, \frac{2}{1+\theta} \nabla \psi \, \partial^m \theta \right\rangle}} + \underset{C_{272}}{\underbrace{\frac{1}{\mu} \left\langle \partial^m w, \nabla \cdot \left( \frac{2}{1+\theta} \nabla \psi \right) \partial^m \theta \right\rangle}} \\&
\quad\underset{C_{273}}{\underbrace{- \frac{1}{\mu} \sum_{\substack{0 \neq m' \leq m}} C_m^{m'} \left\langle \partial^m w, \partial^{m'} \left( \frac{2}{1+\theta} \nabla \psi \right) \cdot \nabla \partial^{m-m'} \theta \right\rangle}} \,.
\end{align*}
These terms are bounded as follows:
\begin{equation*}
\begin{aligned}
C_{271}&\lesssim \left\| \nabla \cdot \partial^m w \right\|_{L^2} \left\| \frac{1}{1+\theta} \nabla \psi \right\|_{L^\infty} \left\| \partial^m \theta \right\|_{L^2} \\
&\lesssim (\|\nabla u\|_{H^s} + \|\nabla \psi\|_{H^s}) \|\nabla \psi\|_{H^s} \|\theta\|_{H^s} \,,
\end{aligned}
\end{equation*}
and
\begin{equation*}
\begin{aligned}
C_{272}&\lesssim \left\| \partial^m w \right\|_{L^4} \left\| \nabla \cdot \left(\frac{1}{1+\theta} \nabla \psi\right) \right\|_{L^4} \|\partial^m \theta\|_{L^2} \\
&\lesssim (\|\nabla u\|_{H^s} + \|\nabla \psi\|_{H^s}) (\|\theta\|_{H^s} + 1) \|\nabla \psi\|_{H^s} \|\theta\|_{H^s} \,,
\end{aligned}
\end{equation*}
and
\begin{equation*}
\begin{aligned}
C_{273}&\lesssim \sum_{\substack{0\neq m' \leq m }} \langle |\partial^m w|, |\partial^{m'} \left(\frac{1}{1+\theta} \nabla \psi\right)| |\nabla \partial^{m-m'} \theta| \rangle \\
&\lesssim \sum_{\substack{m' \leq m \\ |m'| =1}} \|\partial^m w\|_{L^4} \left\| \partial^{m'} \left(\frac{1}{1+\theta} \nabla \psi\right) \right\|_{L^4} \|\nabla \partial^{m-m'} \theta\|_{L^2} \\
&\quad + \sum_{\substack{m' \leq m \\ |m'| \geq 2}} \|\partial^m w\|_{L^4} \left\| \partial^{m'} \left(\frac{1}{1+\theta} \nabla \psi\right) \right\|_{L^2} \|\nabla \partial^{m-m'} \theta\|_{L^4} \\
&\lesssim \|w\|_{H^{s+1}} \left\| \frac{1}{1+\theta} \nabla \psi \right\|_{H^s} \|\nabla \theta\|_{H^{s-1}} \\
&\lesssim (\|\nabla u\|_{H^s} + \|\nabla \psi\|_{H^s}) (1 + \|\theta\|_{H^s}^2) \|\nabla \psi\|_{H^s} \|\theta\|_{H^s} \\
&\lesssim (1 + \|\theta\|_{H^s}^2) \|\theta\|_{H^s} (\|\nabla u\|_{H^s}^2 + \|\nabla \psi\|_{H^s}^2) \,.
\end{aligned}
\end{equation*}
Therefore,
\begin{equation}\label{gpsi-eq21}
\begin{aligned}
C_{27}&\lesssim \left( 1 + \|\theta\|_{H^s}^2 \right) \|\theta\|_{H^s} (\|\nabla u\|_{H^s}^2 + \|\nabla \psi\|_{H^s}^2) \,.
\end{aligned}
\end{equation}

By similar reasoning, the terms $C_{28}$ and $C_{29}$ can be controlled as
\begin{equation}\label{gpsi-eq22}
\begin{aligned}
C_{28}&= -\frac{1}{\mu} \left\langle \partial^m w, \partial^m \left[\frac{\theta}{1+\theta} \nabla (\theta + \tilde{g}(U))\right] \right\rangle \\
&\lesssim (1 + \|\theta\|_{H^s}^s) \|\theta\|_{H^s} (\|\nabla u\|_{H^s} + \|\nabla \psi\|_{H^s}) (\|\nabla \theta\|_{H^{s-1}} + \|\nabla \psi\|_{H^s}^2) \,,
\end{aligned}
\end{equation}
and 
\begin{equation}\label{gpsi-eq23}
\begin{aligned}
C_{29}&=-\frac{1}{\mu} \left\langle \partial^m w, \partial^m \left[\frac{2A}{1+\theta} \nabla \cdot (\nabla d \odot \nabla d)\right] \right\rangle \\
&\lesssim (1 + \|\theta\|_{H^s}^s) \|\nabla d\|_{H^s}^2 (\|\nabla u\|_{H^s} + \|\nabla \psi\|_{H^s}) \,.
\end{aligned}
\end{equation}

Substituting the estimates \eqref{gpsi-eq9}, \eqref{gpsi-eq10}, \eqref{gpsi-eq17}, \eqref{gpsi-eq18}, \eqref{gpsi-eq19}, \eqref{gpsi-eq20}, \eqref{gpsi-eq21}, \eqref{gpsi-eq22} and \eqref{gpsi-eq23} into \eqref{gpsi-eq8}, we obtain
\begin{equation}\label{gpsi-eq24}
\begin{aligned}
C_2 = &\frac{1}{\mu} \langle \Delta \partial^m w, \partial^m \psi \rangle \leq -\frac{1}{2} \frac{\d}{\d t} \left( \|\partial^m(u-\overline{u})\|_{L^2}^2 - \frac{2}{\mu} \langle \partial^m(\psi - \overline{\psi}), \partial^m u \rangle \right) \\
& + ( \frac{c_p}{\mu} - \mu ) \|\nabla \partial^m u\|_{L^2}^2 - (\mu+\xi) \|\nabla\cdot \partial^m u\|_{L^2}^2 + \langle \nabla \partial^m \psi, \nabla \partial^m u \rangle \\
& + \frac{\mu+\xi}{\mu} \langle \nabla \cdot\partial^m \psi, \nabla \cdot\partial^m u \rangle + c_0 a \gamma ( \mu \| \nabla u \|_{H^s} + \| \nabla \psi \|_{H^s} ) \| \nabla \theta \|_{H^{s-1}_{\w (\theta)}}\\
& + C (1 + \|\theta\|_{H^s}^s) (\|u\|_{H^s} + \|\theta\|_{H^s} + \|\nabla d\|_{H^s} + \|\nabla \psi\|_{H^s}) \\
&\quad \times (\|\nabla u\|_{H^s}^2 + \|\nabla d\|_{H^s}^2 + \|\nabla \theta\|_{H^{s-1}}^2 + \|\nabla \psi\|_{H^s}^2) \,.
\end{aligned}
\end{equation}
Inserting the estimates \eqref{gpsi-eq5} and \eqref{gpsi-eq24} into \eqref{gpsi-eq4} and summing over $|m|\le s$, we deduce
\begin{align}\label{gpsi-eq-estimate-1}
\no &\frac{1}{2} \frac{\d}{\d t} \left( \|\nabla \psi\|_{H^s}^2 + \|u - \bar{u}\|_{H^s}^2 - \frac{2}{\mu} \sum_{|m| \leq s} \langle \partial^m (\psi - \bar{\psi}), \partial^m u \rangle \right) + \frac{1}{ \mu} \|\nabla \psi\|_{H^s}^2 \\
\no & \quad + ( \mu - \frac{c_p}{\mu} ) \|\nabla u\|_{H^s}^2 + (\mu + \xi) \|\nabla \cdot u\|_{H^s}^2 - \frac{\mu + \xi}{\mu} \sum_{|m| \leq s} \langle \nabla \cdot \partial^m \psi, \nabla \cdot \partial^m u \rangle \\
& \quad- \sum_{|m| \leq s} \langle \nabla \partial^m \psi, \nabla \partial^m u \rangle - c_0 a \gamma ( \mu \| \nabla u \|_{H^s} + \| \nabla \psi \|_{H^s} ) \| \nabla \theta \|_{H^{s-1}_{\w (\theta)}} \\
\no & \lesssim \left( 1 + \|\theta\|_{H^s}^s \right) \left( \|u\|_{H^s} + \|\theta\|_{H^s} + \|\nabla d\|_{H^s} + \|\nabla \psi\|_{H^s} \right) \\
\no & \quad\times \left( \|\nabla u\|_{H^s}^2 + \|\nabla d\|_{H^s}^2 + \|\nabla \theta\|_{H^{s-1}}^2 + \|\nabla \psi\|_{H^s}^2 \right) \,.
\end{align}
Recalling the definitions of $\mathbf{E}_s(t)$ and $\mathbf{D}_s(t)$ in \eqref{close-eq5-1} and \eqref{close-eq6-1}, we obtain estimate \eqref{gpsi-eq-estimate} from \eqref{gpsi-eq-estimate-1}. This completes the proof of the lemma.
\end{proof}

\subsection{Closing the energy estimates: Proof of Proposition \ref{Prop-Global}}

We first add equation \eqref{d13} to \eqref{nd4}. This yields
\begin{equation}\label{close-eq1}
\begin{aligned}
&\frac{1}{2} \frac{\d}{\d t} \left( |\bar{d}|^2 + \|d - \bar{d}\|_{H^s}^2 + \|\nabla d\|_{H^s}^2 \right) + 2A \lambda \|\nabla d\|_{H^s}^2 + 2A \lambda \|\Delta d\|_{H^s}^2 \\
& \lesssim ( 1 + \mathbf{E}_s^\frac{1}{2} (t) ) \mathbf{E}_s^\frac{1}{2} (t) \mathbf{D}_s (t) \,.
\end{aligned}
\end{equation}
Next, we add \eqref{nd16} to \eqref{u-eq-estimate}. It then follows that
\begin{equation}\label{close-eq2}
\begin{aligned}
&\frac{1}{2} \frac{\d}{\d t} \left( \| \theta\|_{H_{\frac{P^\prime(1+\theta)}{1+\theta}}^s}^2 + \|u\|_{H_{1+\theta}^s}^2 \right) + \mu \|\nabla u\|_{H^s}^2 + (\mu + \xi) \|\nabla \cdot u\|_{H^s}^2 - \sum_{|m| \leq s} \langle \nabla \partial^m \psi, \nabla \partial^m u \rangle \\&
\lesssim ( 1 + \mathbf{E}_s^\frac{s}{2} (t) ) \mathbf{E}_s^\frac{1}{2} (t) \mathbf{D}_s (t) \,.
\end{aligned}
\end{equation}
Third, we multiply \eqref{theta-eq-estimate} by an arbitrary small $\eta\in (0,1)$ and add the result to \eqref{close-eq2}. This gives
\begin{equation}\label{close-eq3}
\begin{aligned}
&\frac{1}{2} \frac{\d}{\d t} \left( \|\theta\|_{H_{\frac{P^\prime(1+\theta)}{1+\theta}}^s}^2+ \|u\|_{H_{1+\theta}^s}^2 + \eta \|u + \nabla \theta\|_{H^{s-1}}^2 - \eta \|u\|_{H^{s-1}}^2 - \eta \|\nabla \theta\|_{H^{s-1}}^2 \right) \\
& \quad+ \mu \|\nabla u\|_{H^s}^2 + (\mu + \xi) \|\nabla \cdot u\|_{H^s}^2 + a \gamma \eta \|\nabla \theta\|_{H_{\w (\theta)}^{s-1}}^2 - \sum_{|m| \leq s} \langle \nabla \partial^m \psi, \nabla \partial^m u \rangle \\
& \quad- c_0 \eta \left( \mu \|\nabla u\|_{H^s} + (\mu + \xi) \| \nabla \cdot u \|_{H^s} + \|\nabla \psi\|_{H^s} \right) \| \nabla \theta \|_{ H^{ s-1 }_{ \w (\theta) }} \\
& \lesssim ( 1 + \mathbf{E}_s^\frac{s}{2} (t) ) \mathbf{E}_s^\frac{1}{2} (t) \mathbf{D}_s (t)  \,.
\end{aligned}
\end{equation}
Fourth, for any $\delta \in (0,1)$ and $\epsilon > 0$, we multiply \eqref{gpsi-eq-estimate} by $\delta$ and add the resulting relation together with $\epsilon$ times \eqref{close-eq3} to \eqref{close-eq1}. This yields
\begin{align}\label{close-eq4}
\no & \frac{1}{2} \frac{\d}{\d t} ( \epsilon a \gamma \|\theta\|_{H_{ \w (\theta) }^s}^2 + \epsilon \|u\|_{H_{1+\theta}^s}^2 + \epsilon \eta \|u + \nabla \theta\|_{H^{s-1}}^2 - \epsilon \eta \|u\|_{H^{s-1}}^2 - \epsilon \eta \|\nabla \theta\|_{H^{s-1}}^2 + |\bar{d}|^2 \\
\no & \quad + \|d - \bar{d}\|_{H^s}^2 + \|\nabla d\|_{H^s}^2 + \delta \|\nabla \psi\|_{H^s}^2 + \delta \|u - \bar{u}\|_{H^s}^2 - \frac{2\delta}{\mu} \sum_{|m| \leq s} \langle \partial^m (\psi - \bar{\psi}), \partial^m u \rangle ) \\
\no & \quad+ (\epsilon \mu + \delta \mu  - \delta \frac{c_p}{\mu} ) \|\nabla u\|_{H^s}^2 + (\epsilon + \delta) (\mu + \xi ) \|\nabla \cdot u\|_{H^s}^2 + \epsilon \eta a \gamma \|\nabla \theta\|_{H_{ \w (\theta) }^s }^2 + \frac{\delta}{\mu} \| \nabla \psi\|_{H^s}^2 \\
\no & \quad- (\epsilon + \delta) \sum_{|m| \leq s} \langle \nabla \partial^m \psi, \nabla \partial^m u \rangle - \frac{\mu + \xi}{\mu} \delta \sum_{|m| \leq s} \langle \nabla \cdot \partial^m \psi, \nabla \cdot \partial^m u \rangle \\
& \quad - c_0 \epsilon \eta [ \mu \| \nabla u \|_{H^s} + (\mu + \xi) \| \nabla \cdot u \|_{H^s} + \| \nabla \psi \|_{H^s} ] \| \nabla \theta \|_{ H^{s-1}_{ \w (\theta) } } \\
\no & \quad - c_1 \delta a \gamma ( \mu \| \nabla u \|_{H^s} + \| \nabla \psi \|_{H^s} ) \| \nabla \theta \|_{ H^{s-1}_{ \w (\theta) } }  + 2A\lambda \|\nabla d\|_{H^s}^2 + 2A\lambda \|\Delta d\|_{H^s}^2 \\
\no & \lesssim ( 1 + \mathbf{E}_s^\frac{s}{2} (t) ) \mathbf{E}_s^\frac{1}{2} (t) \mathbf{D}_s (t) \,.
\end{align}
Together with Lemma \ref{Lmm-Glob-ED}, estimate \eqref{close-eq4} implies the bound \eqref{close-eq21}. This completes the proof of Proposition \ref{Prop-Global}.

\section*{Acknowledgments}

This work was supported by National Key R\&D Program of China under the grant 2023YFA 1010300 of Y. -L. Luo. B. L. Guo is supported by National Science Foundation of China with Grant No. 12471228. N. Jiang was supported by grants from the National Natural Science Foundation of China under contract No. 11971360 and No.11731008, and also supported by the Strategic Priority Research Program of Chinese Academy of Sciences under grant No. XDA25010404.  H. Liu is supported by Scientific Research Project of Universities in Anhui Province under contract No. 2024AH051744 and Huainan Normal University Research Fund Program under contract No. 824001.  Y.-L. Luo was supported by grants from the Guang Dong Basic and Applied Basic Research Foundation under contract No. 2024A1515012358, and the Fundamental Research Funds for the Central Universities under contract No. 531118011008. T.-F. Zhang was supported by the National Natural Science Foundation of China No. 12371228, and the Guang Dong Basic and Applied Basic Research Foundation No. 2024A1515011040. 
 
\bibliography{reference}

\end{document}